\documentclass{article}
\pdfoutput=1
\usepackage{preamble-default}

\usepackage{amsmath}
\usepackage{amssymb}
\usepackage{mathrsfs}
\usepackage{hyperref}
\usepackage{mathpartir}
\usepackage{graphicx}
\usepackage{tikz}
\usetikzlibrary{cd}
\usepackage[style=authoryear-comp]{biblatex}
\usepackage{comment}
\usepackage[capitalize]{cleveref}
\usepackage{autobreak}

\newcommand{\AUTHOR}{Taichi Uemura}
\newcommand{\TITLE}{A General Framework for the Semantics of Type Theory}

\newcommand{\argu}{{-}} 
\newcommand{\s}[1]{\mathsf{#1}}
\newcommand{\cat}[1][C]{\mathcal{#1}}
\newcommand{\catI}[1][D]{\cat[#1]}
\newcommand{\obj}[1][A]{#1}
\newcommand{\objI}[1][B]{\obj[#1]}
\newcommand{\objII}[1][C]{\obj[#1]}
\newcommand{\objIII}[1][D]{\obj[#1]}
\newcommand{\objIV}[1][E]{\obj[#1]}
\newcommand{\arr}[1][f]{#1}
\newcommand{\arrI}[1][g]{\arr[#1]}
\newcommand{\arrII}[1][h]{\arr[#1]}
\newcommand{\map}[1][f]{\arr[#1]}
\newcommand{\mapI}[1][g]{\map[#1]}
\newcommand{\mapII}[1][h]{\map[#1]}
\newcommand{\fun}[1][F]{#1} 
\newcommand{\funI}[1][G]{\fun[#1]}
\newcommand{\funII}[1][H]{\fun[#1]}
\newcommand{\trans}[1][\sigma]{#1} 
\newcommand{\transI}[1][\tau]{\trans[#1]}
\newcommand{\transII}[1][\varphi]{\trans[#1]}
\newcommand{\psh}[1][P]{\fun[#1]} 
\newcommand{\bcat}[1][S]{\cat[#1]}
\newcommand{\bcatI}[1][T]{\bcat[#1]}
\newcommand{\bobj}[1][I]{\obj[#1]}
\newcommand{\bobjI}[1][J]{\bobj[#1]}
\newcommand{\bobjII}[1][K]{\bobj[#1]}
\newcommand{\bobjIII}[1][L]{\bobj[#1]}
\newcommand{\barr}[1][u]{\arr[#1]}
\newcommand{\barrI}[1][v]{\barr[#1]}
\newcommand{\barrII}[1][w]{\barr[#1]}
\newcommand{\dfib}[1][A]{\obj[#1]} 
\newcommand{\dfibI}[1][B]{\dfib[#1]}
\newcommand{\dfibII}[1][C]{\dfib[#1]}
\newcommand{\pr}[1][p]{\fun[#1]} 
\newcommand{\prI}[1][q]{\pr[#1]}
\newcommand{\el}[1][a]{#1} 
\newcommand{\elI}[1][b]{\el[#1]}
\newcommand{\darr}[1][m]{\arr[#1]} 
\newcommand{\darrI}[1][n]{\darr[#1]}
\newcommand{\model}[1][S]{\bcat[#1]}
\newcommand{\modelI}[1][T]{\model[#1]}
\newcommand{\twocat}[1][A]{\mathfrak{#1}}
\newcommand{\twocatI}[1][B]{\twocat[#1]}
\newcommand{\idxtwocat}[1][I]{\twocat[#1]}
\newcommand{\idx}[1][i]{\obj[#1]}
\newcommand{\iarr}[1][u]{\arr[#1]}
\newcommand{\iarrI}[1][v]{\iarr[#1]}
\newcommand{\tth}[1][T]{\mathbb{#1}} 
\newcommand{\tthI}[1][S]{\tth[#1]}
\newcommand{\theory}[1][K]{\fun[#1]}
\newcommand{\thmap}[1][m]{\map[#1]} 
\newcommand{\thmapI}[1][n]{\thmap[#1]}
\newcommand{\const}[1]{\mathrm{#1}}
\newcommand{\catconst}[1]{\mathbf{#1}}
\newcommand{\twocatconst}[1]{\mathfrak{#1}}
\newcommand{\univ}[1][U]{\mathscr{#1}}
\newcommand{\Cat}{\twocatconst{Cat}}
\newcommand{\RMCat}{\twocatconst{Rep}} 
\newcommand{\rep}{\const{r}}
\newcommand{\DFib}{\catconst{DFib}}
\DeclareMathOperator{\dom}{\mathbf{dom}}
\DeclareMathOperator{\cod}{\mathbf{cod}}
\newcommand{\To}{\Rightarrow}

\newcommand{\adj}{\dashv}
\newcommand{\unit}{\eta}
\newcommand{\counit}{\varepsilon}
\newcommand{\radj}{\delta} 
\newcommand{\proj}{\pi}
\newcommand{\Mod}{\twocatconst{Mod}}
\newcommand{\iM}{\model[I]}       
\newcommand{\sM}{\mathbf{M}}
\DeclareMathOperator*{\colim}{colim}
\DeclareMathOperator*{\plim}{plim}
\DeclareMathOperator*{\pcolim}{pcolim}
\newcommand{\op}{\const{op}}
\newcommand{\I}{\mathbb{I}}
\newcommand{\sig}{\s{sig}}
\newcommand{\ctx}{\s{ctx}}
\newcommand{\smalltype}{{*}}
\newcommand{\largetype}{{\Box}}
\newcommand{\sRM}{\tth[R]}      
\newcommand{\under}{/}
\newcommand{\Sect}{\const{Sect}}
\newcommand{\Theory}{\catconst{Th}}
\newcommand{\iL}{\const{L}}     
\newcommand{\Set}{\catconst{Set}}
\newcommand{\dem}{\const{dem}}
\newcommand{\judg}{\mathcal{J}}
\newcommand{\act}{\cdot} 
\newcommand{\id}{\const{id}}
\newcommand{\poly}{\const{P}} 
\newcommand{\inc}{\fun[I]} 
\newcommand{\Ty}{U}
\newcommand{\El}{E}
\newcommand{\typeof}{\partial}
\newcommand{\ctxof}{\mathrm{ft}}
\newcommand{\heart}{\heartsuit}
\newcommand{\cst}[1][c]{\el[#1]} 
\newcommand{\cstI}[1][d]{\cst[#1]}
\newcommand{\idxcat}[1][I]{\cat[#1]}
\newcommand{\idxcatI}[1][J]{\idxcat[#1]}
\newcommand{\inj}{\iota}
\newcommand{\diag}{\Delta} 

\newcommand{\thgat}{\tth[G]}
\newcommand{\ev}{\const{ev}}
\newcommand{\twoCAT}{2\twocatconst{CAT}}

\author{\AUTHOR}
\title{\TITLE}

\hypersetup{
  pdfauthor={\AUTHOR},
  pdftitle={\TITLE}
}

\theoremstyle{definition}
\newtheorem{definition}{Definition}[section]
\theoremstyle{plain}
\newtheorem{theorem}[definition]{Theorem}
\newtheorem{proposition}[definition]{Proposition}
\newtheorem{lemma}[definition]{Lemma}
\newtheorem{corollary}[definition]{Corollary}
\newtheorem{question}[definition]{Question}
\theoremstyle{remark}
\newtheorem{remark}[definition]{Remark}
\newtheorem{example}[definition]{Example}

\allowdisplaybreaks

\Crefname{diagram}{Diagram}{Diagrams}

\DeclareNameAlias{sortname}{family-given}

\begin{document}

\maketitle

\begin{abstract}
  We propose an abstract notion of a type theory to unify the
  semantics of various type theories including Martin-L\"{o}f type
  theory, two-level type theory and cubical type theory. We establish
  basic results in the semantics of type theory: every type theory has
  a bi-initial model; every model of a type theory has its internal
  language; the category of theories over a type theory is
  bi-equivalent to a full sub-2-category of the 2-category of models
  of the type theory.
\end{abstract}

\section{Introduction}
\label{sec:introduction}

One of the key steps in the semantics of type theory and logic is to
establish a correspondence between \emph{theories} and
\emph{models}. Every theory generates a model called its
\emph{syntactic model}, and every model has a theory called its
\emph{internal language}. Classical examples are: simply typed
\(\lambda\)-calculi and cartesian closed categories
\parencite{lambek1986higher}; generalized algebraic theories and
contextual categories \parencite{cartmell1978generalised}; extensional
Martin-L\"{o}f theories and locally cartesian closed categories
\parencite{seely1984locally}; first-order theories and hyperdoctrines
\parencite{seely1983hyperdoctrines}; higher-order theories and
elementary toposes \parencite{lambek1986higher}. Recently, homotopy
type theory \parencite{hottbook} is expected to provide an internal
language for what should be called ``elementary
\((\infty, 1)\)-toposes''. As a first step, \textcite{kapulkin2017internal}
showed that there is an equivalence between intensional Martin-L{\"o}f
theories and finitely complete \((\infty,1)\)-categories.

As there exist correspondences between theories and models for almost
all type theories and logics, it is natural to ask if one can define a
general notion of a type theory or logic and establish correspondences
between theories and models uniformly. First, we clarify what we
informally mean by ``type theory'', ``logic'', ``theory'' and
``model''. By a \emph{type theory} or \emph{logic} we mean a formal
system for deriving judgments which is specified by a collection of
inference rules. For example, first-order logic is a logic and has
inference rules for logical connectives and quantifiers such as
\(\land\) and \(\forall\). For a type theory or logic \(\tth\), by a
\emph{\(\tth\)-theory} or \emph{theory over \(\tth\)} we mean a set of
constants and axioms written in \(\tth\), while by a \emph{model of
  \(\tth\)} we mean a mathematical structure that admits an
interpretation of the inference rules of \(\tth\). For example, a
first-order theory is a theory over first-order logic and consists of
type constants, term constants, predicate constants and axioms, while
a hyperdoctrine is a model of first-order logic and interprets, for
instance, the universal quantifier \(\forall\) as a right adjoint.

A successful syntactic approach to defining general type theories and
logics is a \emph{logical framework} such as the Edinburgh Logical
Framework \parencite{harper1993framework} and Martin-L{\"o}f's logical
framework \parencite{nordstrom1990programming}. A logical framework is
a special kind of type theory such that a variety of type theories are
encoded by a theory over the logical framework which is also called a
signature. However, a logical framework often lacks a good notion of a
model of a signature. Models of a signature may not even form a
category \parencite{capriotti2016models}.

In this paper we propose an abstract notion of a type theory from a
semantic point of view and establish a correspondence between theories
and models. Our notion of a type theory includes a wide range of type
theories: Martin-L\"{o}f type theory
\parencite{martin-lof1975intuitionistic}; two-level type theory
\parencite{altenkirch2016extending,annenkov2017two-level,voevodsky2013simple}
; cubical type theory \parencite{cohen2016cubical}. Roughly speaking,
our type theories are the type theories that admit semantics based on
\emph{categories with families (cwfs)}
\parencite{dybjer1996internal}. Our contribution is to establish a
correspondence between theories and cwf-like models for a wide variety
of type theories.

Our notion of a type theory is inspired by the notion of a
\emph{natural model} of homotopy type theory given by
\textcite{awodey2018natural}. He pointed out that a category with
families is the same thing as a \emph{representable map} of
presheaves\footnote{This fact is also observed independently by
  \textcite{fiore2012discrete}.} and that type and term constructors
are modeled by algebraic operations on presheaves. Thus, a cwf-model
of a type theory is a diagram in a presheaf category in which some
maps are required to be representable. In other words, a cwf-model is
a functor \(\fun\) from a category \(\tth\) to a presheaf category
such that some arrows in \(\tth\) are marked `representable' and
\(\fun\) sends representable arrows in \(\tth\) to representable maps
of presheaves. The category \(\tth\) is considered to encode
derivations as arrows, and the functor \(\fun\) is considered to
interpret derivations as maps of presheaves. From this observation, we
\emph{define} a type theory to be a category with representable arrows
and a model of a type theory to be a functor to a presheaf category
that sends representable arrows to representable maps of presheaves.

For a type theory \(\tth\) in this sense, a \emph{\(\tth\)-theory} is
also defined as a functor from \(\tth\) but to the category of
sets. The intuition behind this definition is different from that of a
model of \(\tth\): for a model of \(\tth\), the values at arrows are
relevant; for a \(\tth\)-theory, the values at objects are
relevant. Objects in \(\tth\) are domains and codomains of derivations
and thought of as \emph{judgment forms}. We identify a \(\tth\)-theory
\(\theory\) with the assignment to each object \(\obj \in \tth\) of the
set of closed derivations of judgment form \(\obj\) that are derivable
using inference rules of \(\tth\) and constants of \(\theory\).

With these definitions of a type theory, a model of a type theory and
a theory over a type theory, we establish a correspondence between
theories and models in a purely categorical way. For a type theory
\(\tth\), the models of \(\tth\) form a \(2\)-category \(\Mod_{\tth}\)
and the theories over \(\tth\) form a category \(\Theory_{\tth}\). We
construct a \(2\)-functor
\(\iL_{\tth} : \Mod_{\tth} \to \Theory_{\tth}\) (regarding
\(\Theory_{\tth}\) as a locally discrete \(2\)-category) which assigns
to each model of \(\tth\) its \emph{internal language}. We have two
main results. The first main result is that the \(2\)-functor
\(\iL_{\tth}\) has a left bi-adjoint \(\sM_{\tth}\)
(\cref{thm:internal-language-adjunction}) which assigns to each theory
over \(\tth\) its \emph{syntactic model}. It will turn out that the
left bi-adjoint \(\sM_{\tth}\) is locally an equivalence and thus
induces a bi-equivalence between \(\Theory_{\tth}\) and the
bi-essential image of \(\sM_{\tth}\). The second main result is a
characterization of the models of \(\tth\) that belong to the
bi-essential image of \(\sM_{\tth}\). We introduce a notion of a
\emph{democratic model} of \(\tth\), generalizing the notion of a
democratic cwf \parencite{clairambault2014biequivalence}, and show
that the bi-essential image of \(\sM_{\tth}\) is precisely the class
of democratic models. Consequently, we have a bi-equivalence between
the locally discrete \(2\)-category \(\Theory_{\tth}\) and the full
sub-\(2\)-category \(\Mod_{\tth}^{\dem} \subset \Mod_{\tth}\)
consisting of democratic models (\cref{thm:bi-equivalence-th-mod}).

A logical framework is still useful to construct concrete examples of
our type theories. We introduce a logical framework whose signatures
can be identified with type theories in our sense. This logical
framework is semantically motivated and thus designed to have a nice
\(2\)-category of models of a signature. At the same time, this
logical framework is sufficiently expressive to encode various type
theories including Martin-L{\"o}f type theory, two-level type theory,
and cubical type theory as promised.

\subsection{Organization}
\label{sec:organization}

In \cref{sec:natural-models-type} we review natural models of type
theories. Natural models are described in terms of presheaves, but we
will work with \emph{discrete fibrations} instead of presheaves. In
\cref{sec:type-theories-their} we introduce a notion of a category
equipped with a class of representable arrows and call it a
\emph{representable map category}. A type theory is then defined to be
a (small) representable map category. We also define the
\(2\)-category \(\Mod_{\tth}\) of models of a type theory \(\tth\).

The rest of the paper splits into two branches independent of each
other. One branch (\cref{sec:logical-framework}) is devoted to giving
examples of our type theories. We introduce a logical framework whose
signatures can be identified with representable map categories. We
construct a \emph{syntactic representable map category} from a
signature of the logical framework and show that the syntactic
representable map category of a signature has an appropriate universal
property (\cref{thm:syntactic-LF-category}). Using the universal
property, we concretely describe the \(2\)-category of models of a
type theory defined in the logical framework
(\cref{thm:models-of-syntactic-lf-category}).

On the other branch
(\cref{sec:bi-initial-models,sec:internal-languages}), we develop the
semantics of our type theories. In \cref{sec:bi-initial-models} we
construct a \emph{bi-initial model} of a type theory
(\cref{thm:initial-model}). We also introduce the notion of a
democratic model here. In \cref{sec:internal-languages} we define the
category \(\Theory_{\tth}\) of theories over a type theory \(\tth\)
and show the main results. Using bi-initial models we construct the
left bi-adjoint of the internal language \(2\)-functor
\(\iL_{\tth} : \Mod_{\tth} \to \Theory_{\tth}\)
(\cref{thm:internal-language-adjunction}). We then show that this
bi-adjunction induces a bi-equivalence
\(\Mod_{\tth}^{\dem} \simeq \Theory_{\tth}\)
(\cref{thm:bi-equivalence-th-mod}).

\subsection{Related Work}
\label{sec:related-work}

\Textcite{bauer2020general} propose a general definition of syntax and
rules for dependent type theories. Properties on sets of rules such as
admissibility of substitution are proved at a high level of
generality. Such properties are not in the scope of our semantic
framework because we take substitution for granted as part of the
structure of a model of a type theory. The author's PhD thesis
\parencite[Chapter 4]{uemura2021thesis} contains a generalization of
\citeauthor{bauer2020general}'s approach as syntactic counterparts of
representable map categories.

Our style of the semantics of type theories is a variant of the
\emph{functorial semantics} initiated by
\textcite{lawvere2004functorial}. The original work is the functorial
semantics of algebraic theories in which an algebraic theory is
identified with a category of some sort and a model of an algebraic
theory is identified with a set-valued functor from that category. A
noticeable difference is that a model in our functorial semantics of
type theories is a functor valued in presheaves over a category
instead of sets. The base category plays the role of the category of
contexts and substitutions and is essential to interpretation of
context extensions.

One limitation of our framework is that non-trivial operations on
contexts are not allowed. Thus, type theories with ``dual-contexts''
\parencite[e.g.][]{licata2018internal,pfenning2001judgmental,shulman2018brouwer}
or modal type theories
\parencite[e.g.][]{birkedal2020modal,gratzer2021multimodal} are not
covered by our definition. The framework of
\textcite{licata2017fibrational} is suitable for defining simple type
theories with operations on contexts, but its dependently typed
version \parencite{licata2019fibrational} has not been finished.

After the manuscript of this paper had been written, Hoang Kim Nguyen
and the author have developed a theory of an \(\infty\)-categorical
generalization of type theories called \emph{\(\infty\)-type theories}
\parencite{nguyen2022type-arxiv}. The results of
\cref{sec:bi-initial-models,sec:internal-languages} are subsumed by
analogous results on \(\infty\)-type theories. Nevertheless, it is still
worth presenting the \(1\)-categorical case because all the
constructions in this paper are explicit while the
\(\infty\)-categorical proofs are non-constructive in the current
foundations for \(\infty\)-category theory using quasicategories
\parencite{lurie2009higher,cisinski2019higher}. Also, a logical
framework for \(\infty\)-type theories has not yet been developed.

\section{Preliminaries}
\label{sec:preliminaries}

We fix terminology and notation on categories and \(2\)-categories.

\begin{enumerate}
\item We refer the reader to \textcite{kelly1974review} for basic
  concepts of \(2\)-category theory.
\item In general we use prefix \mbox{``\(2\)-''} for strict
  \(2\)-categorical notions and prefix \mbox{``bi-''} or
  \mbox{``pseudo-''} for weak \(2\)-categorical notions: the
  composition of \(1\)-cells in a \(2\)-category is associative up to
  equality, while that in a bi-category is associative only up to
  coherent isomorphism; a \(2\)-functor preserves composition of
  \(1\)-cells on the nose, while a pseudo-functor does only up to
  coherent isomorphism. An exception is that pseudo-(co)limits satisfy
  strict \(2\)-categorical universal properties.
\item Let \(P\) be some property on a functor. We say a \(2\)-functor
  \(\fun : \twocat \to \twocatI\) is \emph{locally \(P\)} if the
  functor \(\fun : \twocat(\obj, \obj') \to \twocatI(\fun\obj, \fun\obj')\)
  satisfies \(P\) for all objects \(\obj, \obj' \in \twocat\). For
  example, \(\fun\) is locally fully faithful if the functor
  \(\fun : \twocat(\obj, \obj') \to \twocatI(\fun\obj, \fun\obj')\) is fully
  faithful for all objects \(\obj, \obj' \in \twocat\).
\item Let \(\fun : \twocat \to \twocatI\) be a \(2\)-functor. We say
  \(\fun\) is \emph{bi-essentially surjective on objects} if, for any
  object \(\objI \in \twocatI\), there exists an object
  \(\obj \in \twocat\) such that \(\fun\obj\) is equivalent to \(\objI\) in
  \(\twocatI\). We say \(\fun\) is a \emph{bi-equivalence} if it is
  bi-essentially surjective on objects, locally essentially surjective
  on objects and locally fully faithful.
\item A \(2\)-functor \(\fun : \twocat \to \twocatI\) is said to have
  a \emph{left bi-adjoint} if, for any object \(\objI \in \twocatI\),
  there exist an object \(\funI\objI\in \twocat\) and a \(1\)-cell
  \(\unit_{\objI} : \objI \to \fun \funI\objI\) such that, for any object
  \(\obj \in \twocat\), the composite
  \[
    \begin{tikzcd}
      \twocat(\funI\objI, \obj)
      \arrow[r,"\fun"] &
      \twocatI(\fun\funI\objI, \fun\obj)
      \arrow[r,"\unit_{\objI}^{*}"] &
      \twocatI(\objI, \fun\obj)
    \end{tikzcd}
  \]
  is an equivalence of categories. The \(1\)-cell
  \(\unit_{\objI} : \objI \to \fun\funI\objI\) is called the \emph{unit}. For an object
  \(\obj \in \twocat\), we have a \(1\)-cell
  \(\counit_{\obj} : \funI\fun\obj \to \obj\) called the \emph{counit} such that
  \(\fun \counit_{\obj} \circ \unit_{\fun\obj}\) is isomorphic to the identity
  on \(\fun\obj\).
\item One can show that if a \(2\)-functor \(\fun\) has a left bi-adjoint
  and the unit and counit are equivalences, then \(\fun\) is a
  bi-equivalence.
\item We say a category \(\cat\) is \emph{contractible} if the
  unique functor \(\cat \to 1\) into the terminal category is an
  equivalence. In other words, \(\cat\) has some object and, for
  any objects \(\obj, \objI \in \cat\), there exists a unique arrow \(\obj
  \to \objI\).
\item An object \(\obj\) of a \(2\)-category \(\twocat\) is
  \emph{bi-initial} if the category \(\twocat(\obj, \objI)\) is
  contractible for any object \(\objI \in \twocat\).
\item For a category \(\cat\), we denote by \(|\cat|\) the
  largest groupoid contained in \(\cat\), that is, the
  subcategory of \(\cat\) consisting of all the objects and the
  isomorphisms.
\item A \emph{cartesian category} is a category that has finite
  limits. A \emph{cartesian functor} between cartesian categories is a
  functor that preserves finite limits.
\item We fix a Grothendieck universe \(\univ\). By ``small'' we mean
  ``\(\univ\)-small''.
\item \(\Set\) denotes the category of small sets. \(\Cat\) denotes
  the \(2\)-category of small categories.
\item For a functor \(\theory : \cat \to \Set\) and an arrow
  \(\arr : \obj \to \obj'\) in \(\cat\), we denote by \((\arr \act \argu)\) the
  map \(\theory(\arr) : \theory(\obj) \to \theory(\obj')\). For a contravariant functor
  \(\psh : \cat^{\op} \to \Set\), we denote by \((\argu \act \arr)\) the map
  \(\psh(\arr) : \psh(\obj') \to \psh(\obj)\) for an arrow \(\arr : \obj \to \obj'\) in
  \(\cat\). We use similar notation for pseudo-functors
  \(\cat \to \Cat\).
\end{enumerate}

\section{Natural Models of Type Theory}
\label{sec:natural-models-type}

We review natural models of dependent type theory
\parencite{awodey2018natural}. Natural models are described in terms
of presheaves and representable natural transformations, but we prefer
to work with discrete fibrations instead of presheaves. While
presheaves are intuitive and convenient to describe concrete examples
of models of a type theory, discrete fibrations are convenient for the
study of the \(2\)-category of models of a type theory. Concretely, a
model of a type theory will simply be a \(\Cat\)-valued \(2\)-functor
(\cref{rem:base-of-model}). This section is mostly devoted to
rephrasing the theory of natural models in terms of discrete
fibrations. \Cref{prop:dfib-rep-pushforward} might be new to the
reader: the pushforward along a representable map of discrete
fibrations is given by the pullback along the right adjoint of the
representable map.

\subsection{Discrete Fibrations}
\label{sec:discrete-fibrations}

\begin{definition}
  A \emph{discrete fibration} is a functor \(\pr : \dfib \to \bcat\) such
  that, for any object \(\el \in \dfib\) and arrow \(\barr : \bobjI \to \pr(\el)\) in
  \(\bcat\), there exists a unique arrow \(\darr : \elI \to \el\) such that
  \(\pr(\darr) = \barr\). Such a unique arrow is denoted by
  \(\overline{\barr}_{\el} : \barr^{*}\el \to \el\) or \(\el \act \barr \to \el\). When
  \(\pr : \dfib \to \bcat\) is a discrete fibration, we say \(\dfib\) is a
  discrete fibration over \(\bcat\) and refer to the functor \(\pr\)
  as \(\pr_{\dfib}\). For a discrete fibration \(\dfib\) over \(\bcat\), a
  discrete fibration \(\dfibI\) over \(\bcatI\) and a functor
  \(\fun : \bcat \to \bcatI\), a \emph{map \(\dfib \to \dfibI\) of discrete
    fibrations over \(\fun\)} is a functor \(\map : \dfib \to \dfibI\) such that
  \(\pr_{\dfibI} \circ \map = \fun \circ \pr_{\dfib}\). A map of discrete fibrations over
  the identity functor on \(\bcat\) is called a map of discrete
  fibrations \emph{over \(\bcat\)}. For discrete fibrations \(\dfib\)
  and \(\dfibI\) over \(\bcat\), we denote by \(\DFib_{\bcat}(\dfib, \dfibI)\)
  the class of maps \(\dfib \to \dfibI\) of discrete fibrations over
  \(\bcat\). For a small category \(\bcat\), we will refer to the
  category of small discrete fibrations over \(\bcat\) and their
  maps as \(\DFib_{\bcat}\).
\end{definition}

We recall some basic properties on discrete fibrations.

\begin{proposition}
  \label{prop:discrete-fibration-char}
  For a functor \(\pr : \dfib \to \bcat\), the following are equivalent.
  \begin{enumerate}
  \item \label{item:23} \(\pr\) is a discrete fibration.
  \item \label{item:24} The diagram
    \[
      \begin{tikzcd}
        {\dfib^{\to}}
        \arrow[r,"{\pr^{\to}}"]
        \arrow[d,"\cod"'] &
        {\bcat^{\to}}
        \arrow[d,"\cod"]\\
        \dfib
        \arrow[r,"\pr"'] &
        {\bcat}
      \end{tikzcd}
    \]
    is a pullback.
  \item \label{item:25} For any object \(\el \in \dfib\), the functor
    \(\dfib/\el \to \bcat/\pr(\el)\) induced by \(\pr\) is an
    isomorphism.
  \end{enumerate}
\end{proposition}
\begin{proof}
  The implications \(\labelcref{item:24} \To \labelcref{item:25}\) and
  \(\labelcref{item:25} \To \labelcref{item:23}\) are immediate. To
  see \(\labelcref{item:23} \To \labelcref{item:24}\), suppose that
  \(\pr\) is a discrete fibration. By definition, the functor
  \(\dfib^{\to} \to \dfib \times_{\bcat} \bcat^{\to}\) is bijective on
  objects. To see that this functor is also fully faithful, let
  \(\arr_{1} : \el_{1} \to \el_{1}'\) and
  \(\arr_{2} : \el_{2} \to \el_{2}'\) be objects in \(\dfib^{\to}\), let
  \(\arrI : \el_{1}' \to \el_{2}'\) be an arrow in \(\dfib\), let
  \(\barr : \pr(\el_{1}) \to \pr(\el_{2})\) be an arrow in \(\bcat\),
  and suppose that
  \(\pr(\arr_{2}) \circ \barr = \pr(\arrI) \circ \pr(\arr_{1})\). We have to
  show that there exists a unique arrow
  \(\widehat{\barr} : \el_{1} \to \el_{2}\) in \(\dfib\) such that
  \(\pr(\widehat{\barr}) = \barr\) and
  \(\arr_{2} \circ \widehat{\barr} = \arrI \circ \arr_{1}\). Such a
  \(\widehat{\barr}\) must be
  \(\overline{\barr}_{\el_{2}} : \barr^{*} \el_{2} \to \el_{2}\) by the
  condition \(\pr(\widehat{\barr}) = \barr\), and it indeed satisfies
  \(\arr_{2} \circ \overline{\barr}_{\el_{2}} = \arrI \circ \arr_{1}\) since
  both \(\arr_{2} \circ \overline{\barr}_{\el_{2}}\) and
  \(\arrI \circ \arr_{1}\) have the same codomain and are sent by
  \(\pr\) to \(\pr(\arr_{2}) \circ \barr = \pr(\arrI) \circ \pr(\arr_{1})\).
\end{proof}

\begin{proposition}
  \label{prop:dfib-conservative}
  A discrete fibration \(\pr : \dfib \to \bcat\) is faithful and reflects
  isomorphisms: an arrow \(\arr : \el \to \el'\) in \(\dfib\) is an isomorphism
  whenever \(\pr(\arr)\) is.
\end{proposition}
\begin{proof}
  Let \(\arr, \arrI : \el \to \el'\) be arrows in \(\dfib\) such that
  \(\pr(\arr) = \pr(\arrI)\). Then both \(\arr\) and \(\arrI\) must be
  equal to
  \(\overline{\pr(\arr)}_{\el'} : \pr(\arr)^{*}\el' \to \el'\), and thus
  \(\pr\) is faithful. A morphism \(\arr : \el \to \el'\) in \(\dfib\)
  is an isomorphism if and only if it is the terminal object in
  \(\dfib/\el'\). Therefore, by \cref{item:25} of
  \cref{prop:discrete-fibration-char}, \(\pr\) reflects isomorphisms.
\end{proof}

For a small category \(\bcat\), the category \(\DFib_{\bcat}\) is
equivalent to the category of presheaves over \(\bcat\): for a
presheaf \(\psh\) over \(\bcat\), its category of elements
\(\int_{\bcat}\psh\) together with the projection
\(\int_{\bcat}\psh \to \bcat\) is a discrete fibration over
\(\bcat\); for a discrete fibration \(\dfib\) over \(\bcat\), we have
a presheaf \(\bobj \mapsto \dfib(\bobj)\) where \(\dfib(\bobj)\) denotes the fiber
\(\pr_{\dfib}^{-1}(\bobj)\). A representable presheaf \(\bcat(\argu, \bobj)\)
corresponds to the slice category \(\bcat/\bobj\) with domain functor
\(\bcat/\bobj \to \bcat\). We say a discrete fibration \(\dfib\) over
\(\bcat\) is \emph{representable} if it is isomorphic to
\(\bcat/\bobj\) for some \(\bobj \in \bcat\). We call the functor
\(\bcat \ni \bobj \mapsto \bcat/\bobj \in \DFib_{\bcat}\) the
\emph{Yoneda embedding}. The Yoneda Lemma for discrete fibrations is
formulated as follows.

\begin{theorem}[The Yoneda Lemma]
  Let \(\bcat\) be a category and \(\dfib\) a discrete fibration over
  \(\bcat\). For any object \(\bobj \in \bcat\), the map
  \[
    \DFib_{\bcat}(\bcat/\bobj, \dfib) \ni \map \mapsto
    \map(\id_{\bobj}) \in \dfib(\bobj)
  \]
  is bijective. \qed
\end{theorem}

By the Yoneda Lemma, we identify an element \(\el \in \dfib(\bobj)\) with the
corresponding map \(\bcat/\bobj \to \dfib\) of discrete fibrations over
\(\bcat\). We also recall the following criterion for representability.

\begin{proposition}
  \label{prop:representability}
  Let \(\bcat\) be a category and \(\dfib\) a discrete fibration over
  \(\bcat\). Then \(\dfib\) is representable if and only if it has a
  terminal object. More precisely, for any object \(\el \in \dfib\), the
  corresponding map \(\bcat/\pr_{\dfib}(\el) \to \dfib\) of discrete
  fibrations over \(\bcat\) is an isomorphism if and only if \(\el\)
  is the terminal object. \qed
\end{proposition}

Discrete fibrations are stable under ``base change''.

\begin{proposition}
  \label{prop:dfib-base-change}
  Let \(\pr_{\dfib} : \dfib \to \bcat\) be a discrete fibration.
  \begin{enumerate}
  \item If
    \[
      \begin{tikzcd}
        \dfib'
        \arrow[r,"\overline{\fun}"]
        \arrow[d,"\pr_{\dfib'}"'] &
        \dfib \arrow[d,"\pr_{\dfib}"] \\
        \bcat'
        \arrow[r,"\fun"'] &
        \bcat
      \end{tikzcd}
    \]
    is a pullback of categories, then
    \(\pr_{\dfib'} : \dfib' \to \bcat'\) is a discrete fibration called
    the \emph{base change of \(\dfib\) along \(\fun\)} and denoted by
    \(\fun^{*}\dfib\).
  \item If \(\trans : \fun \To \funI : \bcat' \to \bcat\) is a natural
    transformation and
    \[
      \begin{tikzcd}
        \dfib'_{1}
        \arrow[r,"\overline{\fun}"]
        \arrow[d,"\pr_{\dfib'_{1}}"'] &
        \dfib \arrow[d,"\pr_{\dfib}"] \\
        \bcat'
        \arrow[r,"\fun"'] &
        \bcat
      \end{tikzcd}
      \quad
      \begin{tikzcd}
        \dfib'_{2}
        \arrow[r,"\overline{\funI}"]
        \arrow[d,"\pr_{\dfib'_{2}}"'] &
        \dfib \arrow[d,"\pr_{\dfib}"] \\
        \bcat'
        \arrow[r,"\funI"'] &
        \bcat
      \end{tikzcd}
    \]
    are pullbacks, then there exists a unique pair \((\trans^{*}_{\dfib},
    \overline{\trans}_{\dfib})\) consisting of a map \(\trans^{*}_{\dfib} : \dfib'_{2} \to
    \dfib'_{1}\) of discrete fibrations over \(\bcat'\) and a natural
    transformation \(\overline{\trans}_{\dfib} : \overline{\fun} \trans^{*}_{\dfib} \To
    \overline{\funI}\) such that \(\pr_{\dfib}\overline{\trans}_{\dfib} = \trans
    \pr_{\dfib'_{2}}\).
    \[
      \begin{tikzcd}
        \dfib'_{2}
        \arrow[dr,bend left,"\overline{\funI}",""'{name=bG}]
        \arrow[d,"\trans^{*}_{\dfib}"]
        \arrow[dd,bend right,"\pr_{\dfib_{2}'}"'] &
        [2ex] \\
        \dfib'_{1}
        \arrow[r,"\overline{\fun}"',""{name=bF}]
        \arrow[from=bF,to=bG,Rightarrow,"\overline{\trans}_{\dfib}"']
        \arrow[d,"\pr_{\dfib_{1}'}"] &
        \dfib \arrow[d,"\pr_{\dfib}"] \\
        \bcat'
        \arrow[r,bend right,"\fun"'] &
        \bcat
      \end{tikzcd}
      =
      \begin{tikzcd}
        \dfib'_{2}
        \arrow[dr,bend left,"\overline{\funI}"]
        \arrow[dd,bend right,"\pr_{\dfib_{2}'}"'] & \\
        & \dfib \arrow[d,"\pr_{\dfib}"] \\
        \bcat'
        \arrow[r,bend right,"\fun"',""{name=F}]
        \arrow[r,bend left,"\funI",""'{name=G}]
        \arrow[from=F,to=G,Rightarrow,"\trans"'] &
        \bcat
      \end{tikzcd}
    \]
  \end{enumerate}
\end{proposition}
\begin{proof}
  These are special cases of the base change of fibrations
  \parencite[e.g.][Lemma 1.5.1 and Lemma
  1.7.10]{jacobs1999categorical}.
\end{proof}

\begin{corollary}
  \label{cor:dfib-functorial}
  The assignment \(\bcat \mapsto \DFib_{\bcat}\) determines a pseudo-functor
  from \(\Cat\) to the \(2\)-category of large categories that is
  contravariant on both \(1\)-cells and \(2\)-cells. More precisely, a
  functor \(\fun : \bcat' \to \bcat\) is mapped to the base change
  functor \(\fun^{*} : \DFib_{\bcat} \to \DFib_{\bcat'}\), and a natural
  transformation \(\trans : \fun \To \funI : \bcat' \to \bcat\) is
  mapped to the natural transformation
  \(\trans^{*} : \funI^{*} \To \fun^{*}\) determined by
  \cref{prop:dfib-base-change}.
\end{corollary}
\begin{proof}
  Let \(\DFib \subset \Cat^{\to}\) denote the full sub-\(2\)-category spanned
  by the discrete fibrations. \Cref{prop:dfib-base-change} implies
  that the codomain functor \(\DFib \to \Cat\) is a so-called
  \(2\)-fibration \parencite{hermida1999fib,buckley2014fibred} fibred
  in categories. Then the claim is a special case of
  \parencite[Theorem 2.2.11]{buckley2014fibred}.
\end{proof}

\subsection{Representable Maps of Discrete Fibrations}
\label{sec:repr-map-discr}

\begin{definition}
  \label{def:representable-map}
  Let \(\map : \dfib \to \dfibI\) be a map of discrete fibrations over a category
  \(\bcat\). We say \(\map\) is \emph{representable} if it has a right
  adjoint as a functor \(\dfib \to \dfibI\).
\end{definition}

\begin{remark}
  Representable maps of presheaves are usually defined by the
  equivalent condition of \cref{cor:representable-map} below
  \parencite[\href{https://stacks.math.columbia.edu/tag/0023}{Tag
    0023}]{stacks-project}. \Cref{def:representable-map} is a more
  natural definition when working with discrete fibrations.
\end{remark}

\begin{remark}
  \label{rem:repr-obj-and-repr-map}
  For a discrete fibration \(\dfib\) over a category \(\bcat\), the
  following are \emph{not} equivalent in general.
  \begin{enumerate}
  \item \label{item:1}
    The discrete fibration \(\dfib\) is representable.
  \item \label{item:2}
    The unique map \(\dfib \to \bcat\) of discrete fibrations over
    \(\bcat\) is representable.
  \end{enumerate}
  It follows from \cref{cor:representable-map} below that if \(\bcat\)
  has a terminal object then \labelcref{item:2} is equivalent to that
  the discrete fibration \(\dfib\) over \(\bcat\) is representable by,
  say, \(\bobj\) and \(\bcat\) has products with \(\bobj\). In
  particular, if \(\bcat\) has finite products then \labelcref{item:1}
  and \labelcref{item:2} are equivalent.
\end{remark}

\begin{proposition}
  For a category \(\bcat\), the identity maps of discrete fibrations
  over \(\bcat\) are representable and representable maps of
  discrete fibrations over \(\bcat\) are closed under composition.
\end{proposition}
\begin{proof}
  By definition.
\end{proof}

\begin{proposition}
  \label{prop:representable-map}
  Let \(\map : \dfib \to \dfibI\) be a map of discrete fibrations over a category
  \(\bcat\). For objects \(\elI \in \dfibI\) and \(\el \in \dfib\) and an arrow
  \(\darr : \map\el \to \elI\) in \(\dfibI\), the following are equivalent.
  \begin{enumerate}
  \item \label{item:18} \(\darr : \map\el \to \elI\) is the terminal
    object in the comma category \((\map \downarrow \elI)\).
  \item \label{item:19} The square
    \[
      \begin{tikzcd}
        \bcat/\pr_{\dfib}(\el) \arrow[r,"\el"]
        \arrow[d,"\bcat/\pr_{\dfibI}(\darr)"'] &
        \dfib \arrow[d,"\map"] \\
        \bcat/\pr_{\dfibI}(\elI) \arrow[r,"\elI"'] &
        \dfibI
      \end{tikzcd}
    \]
    is a pullback.
  \end{enumerate}
\end{proposition}
\begin{proof}
  By \cref{item:25} of \cref{prop:discrete-fibration-char},
  \cref{item:19} is equivalent to that the square
  \[
    \begin{tikzcd}
      \dfib/\el
      \arrow[r, "\dom"]
      \arrow[d] &
      \dfib
      \arrow[d,"\map"] \\
      \dfibI/\elI
      \arrow[r, "\dom"'] &
      \dfibI
    \end{tikzcd}
  \]
  is a pullback, where the left functor sends
  \((\darrI : \el' \to \el) \in \dfib/\el\) to
  \((\darr \circ \map\darrI : \map\el' \to \elI) \in \dfibI/\elI\). This is
  equivalent to that the map
  \(\dfib/\el \ni (\darrI : \el' \to \el) \mapsto (\darr \circ \map\darrI) :
  \map\el' \to \elI) \in \map^{*}(\dfibI/\elI)\) of discrete fibrations
  over \(\dfib\) is an isomorphism. By \cref{prop:representability},
  this is equivalent to that \(\darr \circ \map(\id_{\el}) = \darr\) is
  the terminal object in \(\map^{*}(\dfibI/\elI) \cong (\map \downarrow \elI)\).
\end{proof}

\begin{corollary}
  \label{cor:representable-map}
  Let \(\map : \dfib \to \dfibI\) be a map of discrete fibrations over a
  category \(\bcat\). Then \(\map\) is representable if and only if,
  for any object \(\bobj \in \bcat\) and element
  \(\elI : \bcat/\bobj \to \dfibI\), the pullback \(\elI^{*}\dfib\) is a
  representable discrete fibration over \(\bcat\). More precisely, the
  right adjoint \(\radj : \dfibI \to \dfib\) and the counit
  \(\counit\) of \(\map\) fit into the pullback square
  \[
    \begin{tikzcd}
      \bcat/\pr_{\dfib}(\radj(\elI))
      \arrow[r, "\radj(\elI)"]
      \arrow[d, "\bcat/\pr_{\dfibI}(\counit_{\elI})"'] &
      \dfib
      \arrow[d, "\map"] \\
      \bcat/\bobj
      \arrow[r, "\elI"'] &
      \dfibI
    \end{tikzcd}
  \]
  for any element \(\elI : \bcat/\bobj \to \dfibI\).
\end{corollary}
\begin{proof}
  Recall that the right adjoint and the counit of \(\map\) assign the
  terminal object of \((\map \downarrow \elI)\) to each element
  \(\elI \in \dfibI\). Then apply \cref{prop:representable-map}.
\end{proof}

\begin{definition}
  Let
  \[
    \begin{tikzcd}
      \dfib' \arrow[r,"\mapI"] \arrow[d,"\map'"'] &
      \dfib \arrow[d,"\map"] \\
      \dfibI' \arrow[r,"\mapII"'] &
      \dfibI
    \end{tikzcd}
  \]
  be a square of categories that commutes up to isomorphism, and
  suppose that \(\map\) and \(\map'\) have right adjoints \(\radj\) and
  \(\radj'\) respectively. We say this square satisfies the
  \emph{Beck-Chevalley condition} if the canonical natural
  transformation \(\mapI \radj' \To \radj \mapII\) defined by the following
  diagram is an isomorphism.
  \[
    \begin{tikzcd}
      &
      \dfib' \arrow[rr,"\mapI"]
      \arrow[dr,"\map'"] & &
      \dfib \arrow[rr,equal,""'{name=e2}]
      \arrow[dr,"\map"']
      \arrow[dl,phantom,"\cong" description] &
      \arrow[from=e2,to=d,Rightarrow,"\unit"] &
      \dfib \\
      \dfibI' \arrow[ur,"\radj'"]
      \arrow[rr,equal,""{name=e1}] &
      \arrow[from=u,to=e1,Rightarrow,"\counit'"] &
      \dfibI' \arrow[rr,"\mapII"'] & &
      \dfibI \arrow[ur,"\radj"'] &
    \end{tikzcd}
  \]
  Here \(\counit'\) is the counit of the adjunction \(\map' \adj
  \radj'\) and \(\unit\) is the unit of \(\map \adj \radj\).
\end{definition}

\begin{corollary}
  \label{cor:representable-map-pullback}
  \label{cor:rep-pullback-stable}
  Let
  \[
    \begin{tikzcd}
      \dfib' \arrow[r,"\mapI"] \arrow[d,"\map'"'] &
      \dfib \arrow[d,"\map"] \\
      \dfibI' \arrow[r,"\mapII"'] &
      \dfibI
    \end{tikzcd}
  \]
  be a commutative square of discrete fibrations over a category
  \(\bcat\) where \(\map\) is representable. If this square is a
  pullback, then \(\map'\) is representable and this square satisfies
  the Beck-Chevalley condition.
\end{corollary}
\begin{proof}
  Let \(\radj : \dfibI \to \dfib\) denote the right adjoint of
  \(\map\). Let \(\elI : \bcat/\bobj \to \dfibI'\) be an arbitrary
  element. Since \(\dfib' \cong \mapII^{*}\dfib\), we have
  \(\elI^{*} \dfib' \cong (\mapII\elI)^{*}\dfib\). Thus, by
  \cref{cor:representable-map}, \(\elI^{*}\dfib'\) is representable by
  \(\pr_{\dfib}(\radj(\mapII\elI))\). Let \(\radj'(\elI)\) denote the
  composite
  \(\bcat/\pr_{\dfib}(\radj(\mapII\elI)) \cong \elI^{*}\dfib' \to
  \dfib'\). Then, again by \cref{cor:representable-map}, \(\radj'\)
  determines a right adjoint of \(\map'\). By construction, we have
  \(\mapI\radj' \cong \radj\mapII\), and thus the Beck-Chevalley condition
  is satisfied.
\end{proof}

\begin{remark}
  The converse of \cref{cor:representable-map-pullback} also holds: if
  \(\map\) and \(\map'\) are representable and if the square satisfies
  the Beck-Chevalley condition, then the square is a pullback. See
  \parencite[Proposition 3.1.15]{uemura2021thesis} for a proof.
\end{remark}

\subsection{Modeling Type Theory}
\label{sec:modeling-type-theory}

A representable map \(\map : \dfib \to \dfibI\) of discrete fibrations over
\(\bcat\) is considered to be a model of dependent type theory. We
think of objects \(\bobj \in \bcat\) as \emph{contexts}, elements
\(\elI \in \dfibI(\bobj)\) as \emph{types over \(\bobj\)} and elements \(\el \in \dfib(\bobj)\)
as \emph{terms over \(\bobj\)}. For a term \(\el \in \dfib(\bobj)\), the type of
\(\el\) is \(\map(\el) \in \dfibI(\bobj)\). The representability of \(\map\) is used for
modeling \emph{context extensions}.

\begin{definition}
  Let \(\map : \dfib \to \dfibI\) be a representable map of discrete
  fibrations over \(\bcat\). We denote by
  \(\radj^{\map} : \dfibI \to \dfib\) the right adjoint to \(\map\) and
  by \(\counit^{\map}\) the counit of the adjunction
  \(\map \adj \radj^{\map}\). For an object \(\bobj \in \bcat\) and an
  element \(\elI \in \dfibI(\bobj)\), we write \(\{\elI\}^{\map}\) for
  the object \(\pr_{\dfib}(\radj^{\map}_{\elI}) \in \bcat\). Let
  \(\proj^{\map}_{\elI} = \pr_{\dfibI}(\counit^{\map}_{\elI}) :
  \{\elI\}^{\map} \to \bobj\). We call \(\{\elI\}^{\map}\) the
  \emph{context extension of \(\elI\) with respect to \(\map\)}. By
  \cref{cor:representable-map}, these fit into the pullback
  \[
    \begin{tikzcd}
      \bcat/\{\elI\}^{\map} \arrow[r,"\radj^{\map}_{\elI}"]
      \arrow[d,"\proj^{\map}_{\elI}"'] &
      \dfib \arrow[d,"\map"] \\
      \bcat/\bobj \arrow[r,"\elI"'] &
      \dfibI.
    \end{tikzcd}
  \]
\end{definition}

Syntactically, the context extension \(\dfibI(\bobj) \ni \elI \mapsto \{\elI\}^{\map}
\in \bcat\) models the rule for extending a context by a type
\[
  \inferrule
  {\bobj \vdash \elI \ \s{Type}}
  {\bobj, x : \elI \vdash \ctx}
  \ \text{(\(x\) is a fresh variable)}
\]
and \(\proj_{\elI}^{\map}\) corresponds to the context morphism
\((\bobj, x : \elI) \to \bobj\) and \(\radj_{\elI}^{\map}\) corresponds to the term
\(\bobj, x : \elI \vdash x : \elI\).

We demonstrate how to model type constructors using representable maps
of discrete fibrations. First we recall the notions of a pushforward
and a polynomial functor.

\begin{definition}
  Let \(\cat\) be a cartesian category. We say an arrow
  \(\arr : \obj \to \objI\) in \(\cat\) is \emph{exponentiable} if the
  pullback functor \(\arr^{*} : \cat/\objI \to \cat/\obj\) has a right
  adjoint. When \(\arr\) is exponentiable, the right adjoint of
  \(\arr^{*}\) is called the \emph{pushforward} or \emph{dependent
    product} along \(\arr\) and denoted by \(\arr_{*}\).
\end{definition}

\begin{definition}
  \label{def:polynomial}
  For an exponentiable arrow \(\arr : \obj \to \objI\) in a cartesian
  category \(\cat\), we define a functor
  \(\poly_{\arr} : \cat \to \cat\) called the \emph{polynomial functor
    associated with \(\arr\)} to be the composite
  \[
    \begin{tikzcd}
      \cat \arrow[r,"\obj^{*}"] &
      \cat/\obj \arrow[r,"\arr_{*}"] &
      \cat/\objI \arrow[r,"\objI_{!}"] &
      \cat,
    \end{tikzcd}
  \]
  where \(\obj^{*}\) is the pullback functor along \(\obj \to 1\) and
  \(\objI_{!}\) is the forgetful functor.
\end{definition}

Since the category of discrete fibrations over a category
\(\bcat\) is equivalent to the category of presheaves over
\(\bcat\), the pushforward along an arbitrary map exists. But the
pushforward along a representable map has a simple description.

\begin{lemma}
  \label{lem:dfib-cancelation}
  Let
  \[
    \begin{tikzcd}
      \dfibI \arrow[r,"\map"] \arrow[dr,"\prI"'] &
      \dfib \arrow[d,"\pr"] \\
      & \bcat
    \end{tikzcd}
  \]
  be a commutative triangle of categories and suppose that \(\pr\) is
  a discrete fibration. Then \(\map\) is a discrete fibration if and
  only if \(\prI\) is. Consequently, the isomorphism
  \((\Cat/\bcat)/\dfib \cong \Cat/\dfib\) restricts to an isomorphism
  \((\DFib_{\bcat})/\dfib \cong \DFib_{\dfib}\).
\end{lemma}
\begin{proof}
  By \cref{item:24} of \cref{prop:discrete-fibration-char}, this
  follows from the cancellation property of pullback squares.
\end{proof}

\begin{proposition}
  \label{prop:dfib-rep-pushforward}
  Let \(\map : \dfib \to \dfibI\) be a representable map of discrete
  fibrations over a category \(\bcat\). The pushforward along \(\map\)
  is given by the base change along the right adjoint
  \(\radj^{\map} : \dfibI \to \dfib\) to \(\map\).
\end{proposition}
\begin{proof}
  By \cref{cor:dfib-functorial}, the adjunction \(\map \adj
  \radj^{\map}\) is mapped to the adjunction
  \[
    \begin{tikzcd}
      \DFib_{\dfib}
      \arrow[rr, bend right, "(\radj^{\map})^{*}"'] &
      \rotatebox[origin=c]{270}{\(\adj\)} &
      \DFib_{\dfibI}.
      \arrow[ll, bend right, "\map^{*}"']
    \end{tikzcd}
  \]
  By \cref{lem:dfib-cancelation}, \(\map^{*}\) is isomorphic to the
  pullback functor
  \((\DFib_{\bcat})/\dfibI \to (\DFib_{\bcat})/\dfib\), and thus
  \((\radj^{\map})^{*}\) is the pushforward along \(\map\).
\end{proof}

Let \(\map : \dfib \to \dfibI\) be a representable map of discrete
fibrations over a category \(\bcat\). Consider the discrete fibration
\(\poly_{\map}\dfibI\) over \(\bcat\). It is the pullback
\[
  \begin{tikzcd}
    \poly_{\map}\dfibI \arrow[r] \arrow[d] &
    \dfib \times_{\bcat} \dfibI \arrow[d] \\
    \dfibI \arrow[r,"\radj^{\map}"'] &
    \dfib
  \end{tikzcd}
\]
by the construction of \(\map_{*}\) given in
\cref{prop:dfib-rep-pushforward}. Hence, an element of
\(\poly_{\map}\dfibI\) over \(\bobj \in \bcat\) is a pair \((\elI_{1}, \elI_{2})\) of
elements \(\elI_{1} \in \dfibI(\bobj)\) and \(\elI_{2} \in \dfibI(\{\elI_{1}\}^{\map})\). One
can think of \(\elI_{2}\) as a \emph{type family over \(\elI_{1}\)}. Then a
map \(\mapI : \poly_{\map}\dfibI \to \dfibI\) is thought of as a type constructor that
takes types \(\elI_{1} \in \dfibI(\bobj)\) and \(\elI_{2} \in \dfibI(\{\elI_{1}\}^{\map})\) and
returns a type \(\mapI(\elI_{1}, \elI_{2}) \in \dfibI(\bobj)\). Syntactically, \(\mapI\) is a
type constructor of the form
\[
  \inferrule
  {\bobj \vdash \elI_{1} \ \s{Type} \\
    \bobj, x : \elI_{1} \vdash \elI_{2} \ \s{Type}}
  {\bobj \vdash \mapI(\elI_{1}, x.\elI_{2}) \ \s{Type}}
\]
where the expression \(x.\elI_{2}\) means that the variable \(x\) is
bound. Similarly, a map \(\mapII : \poly_{\map}\dfib \to \dfib\) is a term constructor
that takes a type \(\elI_{1} \in \dfibI(\bobj)\) and a term
\(\el_{2} \in \dfib(\{\elI_{1}\}^{\map})\) and returns a term
\(\mapII(\elI_{1}, \el_{2}) \in \dfib(\bobj)\). For example, dependent products are
modeled by maps \(\Pi : \poly_{\map}\dfibI \to \dfibI\) and \(\s{abs} : \poly_{\map}\dfib \to \dfib\)
of discrete fibrations over \(\bcat\) such that the square
\[
  \begin{tikzcd}
    \poly_{\map}\dfib \arrow[r,"\s{abs}"]
    \arrow[d,"\poly_{\map}(\map)"'] &
    \dfib \arrow[d,"\map"] \\
    \poly_{\map}\dfibI \arrow[r,"\Pi"'] &
    \dfibI
  \end{tikzcd}
\]
commutes and is a pullback \parencite[Section 2.1]{awodey2018natural}. The commutativity means that \(\s{abs}\)
has a typing rule
\[
  \inferrule
  {\bobj \vdash \elI_{1} \ \s{Type} \\
    \bobj, x : \elI_{1} \vdash \elI_{2} \ \s{Type} \\
    \bobj, x : \elI_{1} \vdash \el_{2} : \elI_{2}}
  {\bobj \vdash \s{abs}(\elI_{1}, x.\el_{2}) : \Pi(\elI_{1}, x.\elI_{2})}
\]
and being a pullback means that \(\s{abs}\) induces a bijection
between the set of terms \(\bobj, x : \elI_{1} \vdash \el_{2} : \elI_{2}\) and the
set of terms \(\bobj \vdash \el : \Pi(\elI_{1}, x.\elI_{2})\). We refer the reader
to \textcite{awodey2018natural,newstead2018thesis} for further examples.

\section{Type Theories and Their Models}
\label{sec:type-theories-their}

In this section we introduce notions of a \emph{type theory} and a
\emph{model of a type theory}.

We have seen in \cref{sec:natural-models-type} that the vocabulary for
describing models of type theories is:
\begin{itemize}
\item representable maps;
\item finite limits;
\item pushforwards along representable maps
\end{itemize}
in categories of discrete fibrations. The idea is to identify a type
theory with a category equipped with structures of representable
arrows, finite limits and pushforwards along representable arrows so
that a model of a type theory is just a structure-preserving functor
into a category of discrete fibrations.

\begin{definition}
  Let \(\cat\) be a cartesian category. A \emph{stable class of
    exponentiable arrows in \(\cat\)} is a class \(R\) of arrows in
  \(\cat\) satisfying the following conditions:
  \begin{itemize}
  \item identity arrows are in \(R\) and \(R\) are closed under
    composition;
  \item arrows in \(R\) are stable under pullbacks: if
    \[
      \begin{tikzcd}
        \obj'
        \arrow[r]
        \arrow[d,"\arr'"'] &
        \obj \arrow[d,"\arr"] \\
        \objI'
        \arrow[r] &
        \objI
      \end{tikzcd}
    \]
    is a pullback square and \(\arr\) is in \(R\), then \(\arr'\) is
    in \(R\);
  \item arrows in \(R\) are exponentiable.
  \end{itemize}
\end{definition}

\begin{definition}
  A \emph{representable map category} consists of the following data:
  \begin{itemize}
  \item a cartesian category \(\cat\);
  \item a stable class of exponentiable arrows of \(\cat\). Arrows in
    this class are called \emph{representable arrows} or
    \emph{representable maps}.
  \end{itemize}
  A \emph{representable map functor} \(\cat \to \catI\) between
  representable map categories is a functor
  \(\fun : \cat \to \catI\) preserving finite limits, representable
  arrows and pushforwards along representable arrows. For
  representable map categories \(\cat\) and \(\catI\), we write
  \(\RMCat(\cat, \catI)\) for the category of representable map
  functors \(\cat \to \catI\) and natural transformations. We
  will refer to the \(2\)-category of small representable map
  categories, representable map functors and natural transformations
  as \(\RMCat\).
\end{definition}

\begin{example}
  For a small category \(\bcat\), the category \(\DFib_{\bcat}\) is a
  representable map category where the representable maps are defined
  in \cref{def:representable-map}.
\end{example}

\begin{definition}
  \label{def:type-theory}
  A \emph{type theory} is a small representable map category.
\end{definition}

\begin{definition}
  \label{def:model-of-type-theory}
  Let \(\tth\) be a type theory. A \emph{model of
    \(\tth\)} consists of the following data:
  \begin{itemize}
  \item a small category \(\model\) called the \emph{base category}
    with a terminal object \(1\);
  \item a representable map functor \(\tth \to \DFib_{\model}\)
    denoted by \(\obj \mapsto \obj^{\model}\).
  \end{itemize}
\end{definition}

\begin{definition}
  Let \(\model\) be a model of a type theory \(\tth\). For a
  representable arrow \(\arr : \obj \to \objI\) in \(\tth\) and an object
  \(\elI \in \objI^{\model}\), we simply write \(\{\elI\}^{\arr}\) for the
  context extension \(\{\elI\}^{\arr^{\model}}\) of \(\elI\) with respect to
  \(\arr^{\model} : \obj^{\model} \to \objI^{\model}\) and use similar
  notations for \(\proj^{\arr^{\model}}_{\elI}\) and
  \(\radj^{\arr^{\model}}_{\elI}\).
\end{definition}

\begin{remark}
  \label{rem:base-of-model}
  Since a model \(\model\) of a type theory \(\tth\) sends the
  terminal object \(1 \in \tth\) to the terminal discrete fibration
  which is the identity functor \(\id_{\model} : \model \to \model\),
  the base category \(\model\) is a redundant piece of data. We may
  define a model of \(\tth\) as a \(2\)-functor
  \(\model : \tth \to \Cat\) satisfying, among other things, that
  \(\model(\obj) \to \model(1)\) is a discrete fibration for any
  \(\obj \in \tth\). Thus, the \(2\)-category of models of \(\tth\)
  (\cref{def:morphism-model}) will simply be a sub-\(2\)-category of
  the \(2\)-category of \(\Cat\)-valued \(2\)-functors. This clear
  description of the \(2\)-category of models is the reason why we
  prefer discrete fibrations to presheaves.
\end{remark}

We will give interesting examples of representable map categories in
\cref{sec:exampl-basic-depend,sec:logical-framework}. Here we
introduce a couple of constructions of representable map categories.

\begin{example}
  Let \(\cat\) be a representable map category. For an object
  \(\obj \in \cat\), the slice category \(\cat/\obj\) carries a
  structure of a representable map category: an arrow in \(\cat/\obj\)
  is representable if it is a representable arrow in \(\cat\). For
  an arrow \(\arr : \obj \to \objI\), the pullback functor
  \(\arr^{*} : \cat/\objI \to \cat/\obj\) is a representable map
  functor. Thus, \(\obj \mapsto \cat/\obj\) is part of a pseudo-functor
  \(\cat^{\op} \to \RMCat\) when \(\cat\) is small.
\end{example}

\begin{example}
  It is known that exponentiable arrows in a cartesian category
  \(\cat\) are stable under pullbacks
  \parencite[Corollary 1.4]{niefield1982cartesianness}. By
  definition all the identity arrows are exponentiable and
  exponentiable arrows are closed under composition. Hence, for any
  class \(E\) of exponentiable arrows in \(\cat\), we can take the
  smallest stable class of exponentiable arrows containing \(E\), that
  is, the class of composites of pullbacks of arrows from \(E\).
\end{example}

We introduce some notations and terminology for future use.

\begin{definition}
  Let \(\cat\) be a representable map category. We denote by
  \((\cat^{\to})_{\rep}\) the full subcategory of \(\cat^{\to}\)
  consisting of the representable arrows. For an object \(\obj \in
  \cat\), we denote by \((\cat/\obj)_{\rep}\) the full subcategory
  of \(\cat/\obj\) consisting of the representable arrows \(\objI \to
  \obj\).
\end{definition}

\begin{definition}
  Let \(\cat_{0}\) be a representable map category. A
  \emph{representable map category under \(\cat_{0}\)} is a
  representable map category \(\cat\) equipped with a representable
  map functor \(\inc_{\cat} : \cat_{0} \to \cat\). A
  \emph{representable map functor \(\cat \to \catI\) under
    \(\cat_{0}\)} between representable map categories under
  \(\cat_{0}\) is a pair \((\fun, \trans)\) consisting of a
  representable map functor \(\fun : \cat \to \catI\) and a natural
  isomorphism \(\trans : \fun \inc_{\cat} \cong \inc_{\catI}\). A
  \emph{natural transformation \((\fun, \trans) \To (\funI, \transI)\)
    under \(\cat_{0}\)} between representable map functors under
  \(\cat_{0}\) is a natural transformation
  \(\transII : \fun \To \funI\) such that
  \(\transI \circ \transII \inc_{\cat} = \trans\). For representable
  map categories \(\cat\) and \(\catI\) under \(\cat_{0}\), we denote
  by \((\cat_{0} \under \RMCat)(\cat, \catI)\) the category of
  representable map functors under \(\cat_{0}\) and natural
  transformations under \(\cat_{0}\). For a representable map category
  \(\cat\) under \(\cat_{0}\) and a representable map functor
  \(\fun : \cat_{0} \to \catI\), we say \(\fun\) \emph{extends to a
    representable map functor \(\funI : \cat \to \catI\)} when
  \(\funI\) is part of a representable map functor \(\cat \to \catI\)
  under \(\cat_{0}\).
\end{definition}

\subsection{Example: Basic Dependent Type Theory}
\label{sec:exampl-basic-depend}

Most examples of type theories in the sense of
\cref{def:type-theory} are constructed using a logical framework
introduced in \cref{sec:logical-framework}. Here we only give a simple
example which naturally arises from the syntax of dependent type
theory.

Let \(\thgat\) denote the opposite of the category of finite
\emph{generalized algebraic theories}
\parencite{cartmell1978generalised} and interpretations between
them. We do not need the precise definition of a generalized algebraic
theory, but remember that a generalized algebraic theory consists of
sets of type constants, term constants, type equations and term
equations. \(\thgat\) has finite limits, or equivalently
\(\thgat^{\op}\) has finite colimits: coproducts of generalized
algebraic theories are given by disjoint union; coequalizers of
generalized algebraic theories are obtained by adjoining equations.

Let \(\Ty_{n}\) be the generalized algebraic theory consisting of type
constants
\(({} \vdash A_{0}), (x_{0} : A_{0} \vdash A_{1}), \dots, (x_{0} : A_{0}, \dots,
x_{n-1} : A_{n-1}(x_{0}, \dots, x_{n-2}) \vdash A_{n})\) and let
\(\El_{n}\) be the extension of \(\Ty_{n}\) by a term constant
\((x_{0} : A_{0}, \dots, x_{n-1} : A_{n-1}(x_{0}, \dots, x_{n-2}) \vdash
a_{n} : A_{n}(x_{0}, \dots, x_{n-1}))\). We denote by
\(\typeof_{n} : \Ty_{n} \to \El_{n}\) and
\(\ctxof_{n} : \Ty_{n-1} \to \Ty_{n}\) the obvious inclusions (in
\(\thgat^{\op}\)), where we define \(\Ty_{-1}\) to be the empty
theory. We also regard \(\typeof_{n}\) and \(\ctxof_{n}\) as arrows
\(\typeof_{n} : \El_{n} \to \Ty_{n}\) and
\(\ctxof_{n} : \Ty_{n} \to \Ty_{n-1}\), respectively, in \(\thgat\). The
category \(\thgat\) is ``freely generated by an exponentiable arrow''
in the following sense.

\begin{theorem}[\textcite{uemura2019exponentiability}]
  \label{thm:gat-exp}
  \begin{enumerate}
  \item \label{item:5} The arrow \(\typeof_{0} : \El_{0} \to \Ty_{0}\)
    in \(\thgat\) is exponentiable.
  \item \label{item:17} For any cartesian category \(\cat\) and
    exponentiable arrow \(\arr : \obj \to \objI\) in \(\cat\), there
    exists a unique, up to unique isomorphism, cartesian functor
    \(\fun : \thgat \to \cat\) that sends \(\typeof_{0}\) to \(\arr\)
    and pushforwards along \(\typeof_{0}\) to those along \(\arr\).
  \item \label{item:20} \(\poly_{\typeof_{0}}\Ty_{n} \cong \Ty_{n+1}\)
    and \(\poly_{\typeof_{0}}\El_{n} \cong \El_{n+1}\).
  \end{enumerate}
\end{theorem}

By \cref{item:5} of \cref{thm:gat-exp}, we regard \(\thgat\) as a
representable map category with the smallest stable class of
exponentiable arrows containing \(\typeof_{0} : \El_{0} \to \Ty_{0}\)
and call it the \emph{basic dependent type theory}. \Cref{item:17} of
\cref{thm:gat-exp} can be rephrased as follows.

\begin{corollary}
  \label{thm:thgat-ump}
  For any representable map category \(\cat\) and representable
  arrow \(\arr : \obj \to \objI\) in \(\cat\), there exists a unique,
  up to unique isomorphism, representable map functor \(\fun : \thgat
  \to \cat\) that sends \(\typeof_{0}\) to \(\arr\). \qed
\end{corollary}

By this universal property, a model \(\model\) of the type theory
\(\thgat\) consists of the following data:
\begin{itemize}
\item a category \(\model\) with a terminal object;
\item a representable map \(\typeof_{0}^{\model} : \El_{0}^{\model}
  \to \Ty_{0}^{\model}\) of discrete fibrations over \(\model\).
\end{itemize}
This is precisely a natural model (category with families).

\subsection{The \(2\)-category of Models of a Type Theory}
\label{sec:2-category-models}

The models of a type theory are part of a \(2\)-category.

\begin{definition}
  \label{def:morphism-model}
  Let \(\model\) and \(\modelI\) be models of a type theory
  \(\tth\). A \emph{morphism \(\model \to \modelI\) of models of
    \(\tth\)} consists of the following data:
  \begin{itemize}
  \item a functor \(\fun : \model \to \modelI\) between the base categories;
  \item for each object \(\obj \in \tth\), a map
    \(\fun_{\obj} : \obj^{\model} \to \obj^{\modelI}\) of discrete fibrations
    over \(\fun : \model \to \modelI\)
  \end{itemize}
  satisfying the following conditions:
  \begin{itemize}
  \item the functor \(\fun : \model \to \modelI\) preserves terminal objects;
  \item \(\obj \mapsto \fun_{\obj}\) is natural: for any arrow \(\arr : \obj \to \objI\) in
    \(\tth\), the diagram
    \begin{equation}
      \begin{tikzcd}
        \obj^{\model} \arrow[r,"\fun_{\obj}"]
        \arrow[d,"\arr^{\model}"'] &
        \obj^{\modelI} \arrow[d,"\arr^{\modelI}"] \\
        \objI^{\model} \arrow[r,"\fun_{\objI}"'] &
        \objI^{\modelI}
      \end{tikzcd}
      \label{eq:1}
    \end{equation}
    commutes;
  \item for any representable arrow \(\arr : \obj \to \objI\) in \(\tth\),
    the naturality square \eqref{eq:1} satisfies the Beck-Chevalley
    condition.
  \end{itemize}
  A \emph{\(2\)-morphism
    \(\trans : \fun \To \funI : \model \to \modelI\)} of morphisms of models
  of \(\tth\) is a natural transformation \(\trans : \fun \To \funI\)
  between the underlying functors such that, for any object
  \(\obj \in \tth\), there exists a (necessarily unique) natural
  transformation \(\trans_{\obj} : \fun_{\obj} \To \funI_{\obj}\) such that
  \(\pr_{(\obj^{\modelI})} \trans_{\obj} = \trans \pr_{(\obj^{\model})}\).
  \[
    \begin{tikzcd}
      \obj^{\model}
      \arrow[r,bend left,"\fun_{\obj}",""'{name=FA}]
      \arrow[r,bend right,"\funI_{\obj}"',""{name=GA}]
      \arrow[from=FA,to=GA,Rightarrow,"\trans_{\obj}"]
      \arrow[d,"\pr_{(\obj^{\model})}"'] &
      \obj^{\modelI}
      \arrow[d,"\pr_{(\obj^{\modelI})}"] \\
      \model
      \arrow[r,bend right,"\funI"'] &
      \modelI
    \end{tikzcd}
    =
    \begin{tikzcd}
      \obj^{\model}
      \arrow[r,bend left,"\fun_{\obj}"]
      \arrow[d,"\pr_{(\obj^{\model})}"'] &
      \obj^{\modelI}
      \arrow[d,"\pr_{(\obj^{\modelI})}"] \\
      \model
      \arrow[r,bend left,"\fun",""'{name=F}]
      \arrow[r,bend right,"\funI"',""{name=G}]
      \arrow[from=F,to=G,Rightarrow,"\trans"] &
      \modelI
    \end{tikzcd}
  \]
  We denote by \(\Mod_{\tth}\) the \(2\)-category of models of
  \(\tth\) and their morphisms and \(2\)-morphisms.
\end{definition}

\begin{remark}
  The Beck-Chevalley condition for a morphism
  \(\fun : \model \to \modelI\) of models of a type theory \(\tth\)
  forces \(\fun\) to \emph{preserve context extensions} up to
  isomorphism: for a representable arrow \(\arr : \obj \to \objI\) in
  \(\tth\) and an element \(\elI \in \objI^{\model}(\bobj)\), the
  canonical arrow
  \(\fun \{\elI\}^{\arr} \to \{\fun_{\objI} \elI\}^{\arr}\) is an
  isomorphism. As a special case, it will turn out in
  \cref{exm:models-of-dtt} that morphisms of models of \(\thgat\)
  correspond to pseudo cwf-morphisms of cwfs
  \parencite{clairambault2014biequivalence} and, equivalently, to weak
  morphisms of natural models \parencite{newstead2018thesis}.
\end{remark}

\subsection{Another Definition of Morphisms of Models}
\label{sec:anoth-defin-morph}

Results in this subsection are required only in proofs of some
propositions, so the reader may skip this subsection until needed.

A model of a type theory is defined to be a representable map functor
into a category of discrete fibrations. In this subsection we see that
morphisms and \(2\)-morphisms of models of a type theory are also
regarded as representable map functors into suitable representable map
categories.

\begin{definition}
  \label{def:representable-map-diagram}
  Let \(\idxtwocat\) be a small \(2\)-category and
  \(\bcat : \idxtwocat \to \Cat\) a \(2\)-functor. We define a category
  \((\DFib^{\idxtwocat})_{\bcat}\) as follows:
  \begin{itemize}
  \item the objects are the \(2\)-functors \(\dfib : \idxtwocat \to \Cat\)
    equipped with a \(2\)-natural transformation \(\pr_{\dfib} : \dfib \To \bcat\)
    such that each component \((\pr_{\dfib})_{\idx} : \dfib \idx \to \bcat \idx\) is a discrete
    fibration;
  \item the maps \(\dfib \to \dfibI\) are the \(2\)-natural transformations
    \(\map : \dfib \To \dfibI\) such that \(\pr_{\dfibI} \map = \pr_{\dfib}\).
  \end{itemize}
  We say a map \(\map : \dfib \to \dfibI\) in \((\DFib^{\idxtwocat})_{\bcat}\) is
  \emph{representable} if every component \(\map_{\idx} : \dfib \idx \to \dfibI \idx\) is a
  representable map of discrete fibrations over \(\bcat \idx\) and, for any
  \(1\)-cell \(\iarr : \idx \to \idx'\) in \(\idxtwocat\), the square
  \[
    \begin{tikzcd}
      \dfib \idx
      \arrow[r,"\dfib \iarr"]
      \arrow[d,"\map_{\idx}"'] &
      \dfib \idx'
      \arrow[d,"\map_{\idx'}"] \\
      \dfibI \idx
      \arrow[r,"\dfibI \iarr"'] &
      \dfibI \idx'
    \end{tikzcd}
  \]
  satisfies the Beck-Chevalley condition.
\end{definition}

\begin{proposition}
  Representable maps in \((\DFib^{\idxtwocat})_{\bcat}\) are stable
  under pullbacks.
\end{proposition}
\begin{proof}
  Let
  \[
    \begin{tikzcd}
      \dfib' \arrow[r,"\mapI"] \arrow[d,"\map'"'] &
      \dfib \arrow[d,"\map"] \\
      \dfibI' \arrow[r,"\mapII"'] &
      \dfibI
    \end{tikzcd}
  \]
  be a pullback in \((\DFib^{\idxtwocat})_{\bcat}\) and suppose that
  \(\map\) is representable. One can check that pullbacks in
  \((\DFib^{\idxtwocat})_{\bcat}\) are componentwise, which means that
  \[
    \begin{tikzcd}
      \dfib'\idx \arrow[r,"\mapI_{\idx}"] \arrow[d,"\map'_{\idx}"'] &
      \dfib\idx \arrow[d,"\map_{\idx}"] \\
      \dfibI'\idx \arrow[r,"\mapII_{\idx}"'] &
      \dfibI\idx
    \end{tikzcd}
  \]
  is a pullback in \(\DFib_{\bcat\idx}\) for every
  \(\idx \in \idxtwocat\). By \cref{cor:rep-pullback-stable} each
  \(\map'_{\idx}\) is representable. To see the Beck-Chevalley
  condition, let \(\iarr : \idx \to \idx'\) be a \(1\)-cell in
  \(\idxtwocat\). Consider the following commutative diagram.
  \[
    \begin{tikzcd}
      \dfib'\idx
      \arrow[rr,"\dfib'\iarr"]
      \arrow[dr,"\mapI_{\idx}"]
      \arrow[dd,"\map'_{\idx}"'] & &
      \dfib'\idx'
      \arrow[dr,"\mapI_{\idx'}"]
      \arrow[dd,"\map'_{\idx'}"'{near start}] \\
      & \dfib\idx
      \arrow[rr,crossing over,"\dfib\iarr"{near start}] & &
      \dfib\idx'
      \arrow[dd,"\map_{\idx'}"] \\
      \dfibI'\idx
      \arrow[rr,"\dfibI'\iarr"{near start}]
      \arrow[dr,"\mapII_{\idx}"'] & &
      \dfibI'\idx'
      \arrow[dr,"\mapII_{\idx'}"] \\
      & \dfibI\idx
      \arrow[rr,"\dfibI\iarr"']
      \arrow[from=uu,crossing over,"\map_{\idx}"{near start}] & &
      \dfibI\idx'
    \end{tikzcd}
  \]
  The front square satisfies the Beck-Chevalley condition by
  assumption. The left and right squares satisfy the Beck-Chevalley
  condition by \cref{cor:representable-map-pullback}. The functor
  \(\mapI_{\idx'} : \dfib'\idx' \to \dfib\idx'\) is a discrete
  fibration by \cref{lem:dfib-cancelation} and thus reflects
  isomorphisms by \cref{prop:dfib-conservative}. Hence, it follows
  that the back square satisfies the Beck-Chevalley condition.
\end{proof}

Let \(\map : \dfib \to \dfibI\) be a representable map in
\((\DFib^{\idxtwocat})_{\bcat}\). Although the right adjoint
\(\radj^{\map} : \dfibI \to \dfib\) to \(\map\) is only a pseudo-natural
transformation, we can define the base change
\((\radj^{\map})^{*}\dfibII \to \dfibI\) of \(\dfibII\) along \(\radj^{\map}\) for a map
\(\dfibII \to \dfib\) in \((\DFib^{\idxtwocat})_{\bcat}\) as follows:
\begin{itemize}
\item for an object \(\idx \in \idxtwocat\), we define
  \(((\radj^{\map})^{*}\dfibII)\idx = (\radj^{\map}_{\idx})^{*}(\dfibII \idx)\);
\item for a \(1\)-cell \(\iarr : \idx \to \idx'\) in \(\idxtwocat\), we have a
  natural isomorphism
  \[
    \begin{tikzcd}
      \dfibI \idx
      \arrow[r,"\radj^{\map}_{\idx}"]
      \arrow[d,"\dfibI \iarr"'] &
      \dfib \idx
      \arrow[d,"\dfib \iarr"] \\
      \dfibI \idx'
      \arrow[r,"\radj^{\map}_{\idx'}"']
      \arrow[ur,phantom,"\underset{\radj^{\map}_{\iarr}}{\cong}"{description}] &
      \dfib \idx'.
    \end{tikzcd}
  \]
  Using \cref{prop:dfib-base-change}, we have a unique
  pair \(((\radj^{\map}_{\iarr})^{*}(\dfibII \iarr), \overline{\radj}^{\map}_{\iarr})\)
  consisting of a map \((\radj^{\map}_{\iarr})^{*}(\dfibII \iarr) :
  (\radj^{\map}_{\idx})^{*}(\dfibII \idx) \to (\radj^{\map}_{\idx'})^{*}(\dfibII \idx')\) of
  discrete fibrations over \(\dfibI \iarr\) and a natural isomorphism
  \[
    \begin{tikzcd}
      (\radj^{\map}_{\idx})^{*}(\dfibII \idx)
      \arrow[r]
      \arrow[d,"(\radj^{\map}_{\iarr})^{*}(\dfibII \iarr)"'] &
      \dfibII \idx
      \arrow[d,"\dfibII \iarr"] \\
      (\radj^{\map}_{\idx'})^{*}(\dfibII \idx')
      \arrow[r]
      \arrow[ur,phantom,"\underset{\overline{\radj}^{\map}_{\iarr}}{\cong}"{description}] &
      \dfibII \idx'
    \end{tikzcd}
  \]
  over \(\radj^{\map}_{\iarr}\). We define \(((\radj^{\map})^{*} \dfibII) \iarr =
  (\radj^{\map}_{\iarr})^{*}(\dfibII \iarr)\);
\item the \(2\)-cell part is defined in a natural way.
\end{itemize}
Then we can prove the following in the same way as
\cref{prop:dfib-rep-pushforward}.

\begin{proposition}
  Let \(\map : \dfib \to \dfibI\) be a representable map in
  \((\DFib^{\idxtwocat})_{\bcat}\). Then the pushforward along \(\map\)
  exists and is given by the base change along the right adjoint
  \(\radj^{\map} : \dfibI \to \dfib\) to \(\map\). \qed
\end{proposition}

\begin{corollary}
  \begin{enumerate}
  \item For a \(2\)-functor \(\bcat : \idxtwocat \to \Cat\), the
    category \((\DFib^{\idxtwocat})_{\bcat}\) is a representable map
    category where the representable maps are defined in
    \cref{def:representable-map-diagram}.
  \item For \(2\)-functors \(\bcat : \idxtwocat \to \Cat\) and \(\fun :
    \idxtwocat' \to \idxtwocat\), the precomposition with \(\fun\)
    induces a representable map functor \(\fun^{*} :
    (\DFib^{\idxtwocat})_{\bcat} \to (\DFib^{\idxtwocat'})_{\bcat \fun}\).
  \end{enumerate}
  \qed
\end{corollary}

We consider the case that \(\idxtwocat\) is the category
\(\{0 \to 1\}\). In this case we write \((\DFib^{\to})_{\fun}\) for
\((\DFib^{\idxtwocat})_{\fun}\). Let \(\fun : \bcat \to \bcatI\) be a
functor between small categories. By definition, an object of
\((\DFib^{\to})_{\fun}\) is a commutative square
\[
  \begin{tikzcd}
    \dfib \arrow[r,"\funI"]
    \arrow[d,"\pr_{\dfib}"'] &
    \dfibI \arrow[d,"\pr_{\dfibI}"] \\
    \bcat \arrow[r,"\fun"'] &
    \bcatI
  \end{tikzcd}
\]
of categories such that \(\pr_{\dfib}\) and \(\pr_{\dfibI}\) are small discrete
fibrations, and a map \((\dfib, \dfibI, \funI) \to (\dfib', \dfibI', \funI')\) is a square
\[
  \begin{tikzcd}
    \dfib
    \arrow[r,"\funI"]
    \arrow[d,"\map"'] &
    \dfibI
    \arrow[d,"\mapI"] \\
    \dfib'
    \arrow[r,"\funI'"'] &
    \dfibI'
  \end{tikzcd}
\]
of categories such that \(\map\) and \(\mapI\) are maps of discrete
fibrations over \(\bcat\) and \(\bcatI\) respectively. Such a
square is representable if \(\map\) and \(\mapI\) are representable maps of
discrete fibrations and the square satisfies the Beck-Chevalley
condition. There are representable map functors
\(\dom : (\DFib^{\to})_{\fun} \to \DFib_{\bcat}\) and
\(\cod : (\DFib^{\to})_{\fun} \to \DFib_{\bcatI}\) induced by the
inclusions \(\{0\} \to \{0 \to 1\}\) and \(\{1\} \to \{0 \to 1\}\)
respectively.

\begin{proposition}
  \label{prop:morphism-of-models}
  Let \(\tth\) be a type theory, \(\model\) and \(\modelI\) models of
  \(\tth\), and \(\fun : \model \to \modelI\) be a functor between the
  base categories preserving terminal objects. There is a bijection
  between the following sets:
  \begin{enumerate}
  \item \label{item:26} the set of morphisms of models
    \(\model \to \modelI\) whose underlying functor is \(\fun\);
  \item \label{item:27} the set of representable map functors
    \(\fun_{(\argu)} : \tth \to (\DFib^{\to})_{\fun}\) such that
    \(\dom \fun_{(\argu)} = (\argu)^{\model}\) and
    \(\cod \fun_{(\argu)} = (\argu)^{\modelI}\).
    \[
      \begin{tikzcd}
        &
        [8ex]
        {(\DFib^{\to})_{\fun}}
        \arrow[d,"{(\dom, \cod)}"] \\
        \tth
        \arrow[ur,"\fun_{(\argu)}"]
        \arrow[r,"{((\argu)^{\model}, (\argu)^{\modelI})}"'] &
        \DFib_{\model} \times \DFib_{\modelI}
      \end{tikzcd}
    \]
  \end{enumerate}
\end{proposition}
\begin{proof}
  By the definition of representable maps in \((\DFib^{\to})_{\fun}\),
  \cref{item:27} is just another way of describing a morphism of
  models of \(\tth\). Note that, because \((\dom, \cod)\) reflects
  isomorphisms, \(\fun_{(\argu)}\) automatically preserves finite
  limits and pushforwards along representable maps whenever it
  preserves representable maps.
\end{proof}

\begin{remark}
  \label{rem:morphism-of-model-up-to-iso}
  Note that the functor
  \((\dom, \cod) : (\DFib^{\to})_{\fun} \to \DFib_{\model} \times
  \DFib_{\modelI}\) is a discrete isofibration: for any object
  \(\funI \in (\DFib^{\to})_{\fun}\) and isomorphisms
  \(\map : \dom \funI \cong \dfib\) in \(\DFib_{\model}\) and
  \(\mapI : \cod \funI \cong \dfibI\) in \(\DFib_{\modelI}\), there exists
  a unique isomorphism \(\mapII : \funI \cong \funII\) in
  \(\DFib_{\model}\) such that \(\dom \mapII = \map\) and
  \(\cod \mapII = \mapI\). Thus, to extend the functor \(\fun\) to a
  morphism \(\model \to \modelI\) of models of \(\tth\), it suffices to
  give a representable map functor
  \(\fun'_{(\argu)} : \tth \to (\DFib^{\to})_{\fun}\) and natural
  isomorphisms \(\dom \fun'_{(\argu)} \cong (\argu)^{\model}\) and
  \(\cod \fun'_{(\argu)} \cong (\argu)^{\modelI}\).
\end{remark}

Consider the \(2\)-category \(\Theta\) consisting of two \(0\)-cells
\(0\) and \(1\), two \(1\)-cells \(a, b : 0 \to 1\), and one
\(2\)-cell \(a \To b\).
\[
  \begin{tikzcd}
    0
    \arrow[r,bend left,"a",""'{name=f0}]
    \arrow[r,bend right,"b"',""{name=f1}]
    \arrow[from=f0,to=f1,Rightarrow] &
    1.
  \end{tikzcd}
\]
For a natural transformation
\(\trans : \fun \To \funI : \model \to \modelI\), we have the
representable map functors
\(\dom : (\DFib^{\Theta})_{\trans} \to (\DFib^{\to})_{\fun}\) and
\(\cod : (\DFib^{\Theta})_{\trans} \to (\DFib^{\to})_{\funI}\) induced by the
inclusions \(\{0 \overset{a}{\to} 1\} \to \Theta\) and
\(\{0 \overset{b}{\to} 1\} \to \Theta\) respectively.

\begin{proposition}
  \label{prop:2-morphisms-of-model}
  Let \(\tth\) be a type theory, \(\model\) and \(\modelI\) models of
  \(\tth\), \(\fun, \funI : \model \to \modelI\) morphism of models of
  \(\tth\), and \(\trans : \fun \To \funI\) a natural transformation
  between the underlying functors. Then \(\trans\) (necessarily
  uniquely) extends to a \(2\)-morphism \(\fun \To \funI\) of models
  of \(\tth\) if and only if there exists a (necessarily unique)
  representable map functor
  \(\tilde{\trans} : \tth \to (\DFib^{\Theta})_{\trans}\) such that
  \(\dom \tilde{\trans} = \fun_{(\argu)}\) and
  \(\cod \tilde{\trans} = \funI_{(\argu)}\).
  \[
    \begin{tikzcd}
      &
      [8ex]
      {(\DFib^{\Theta})_{\trans}}
      \arrow[d,"{(\dom, \cod)}"] \\
      \tth
      \arrow[ur,"\tilde{\trans}"]
      \arrow[r,"{(\fun_{(\argu)}, \funI_{(\argu)})}"'] &
      {(\DFib^{\to})_{\fun} \times_{(\DFib_{\model} \times
          \DFib_{\modelI})} (\DFib^{\to})_{\funI}}
    \end{tikzcd}
  \]
\end{proposition}
\begin{proof}
  This is a rephrasing of the definition of a \(2\)-morphism of models
  of \(\tth\).
\end{proof}

\section{Logical Framework}
\label{sec:logical-framework}

This section is devoted to giving examples of representable map
categories using a \emph{logical framework}. We will not use the
results of this section in the rest of the paper other than for giving
examples.

Our logical framework is classified as a \emph{semantic} logical
framework and close to Martin-L{\"o}f's logical framework
\parencite{nordstrom1990programming}. Another kind of logical
framework is a \emph{syntactic} logical framework whose typical
example is the Edinburgh Logical Framework
\parencite{harper1993framework}. A syntactic logical framework does
not have equality types so that terms exactly correspond to
derivations in the object language. A semantic logical framework has
equality types and terms correspond to semantic derivations, that is,
derivations modulo equality in the object language. Since the subject
of this paper is the semantics of type theory, a semantic logical
framework is our choice.

We design our logical framework to give syntactic counterparts of
representable map categories. Since representable map categories have
finite limits, the logical framework is a dependent type theory with
equality types \(a = b\). Corresponding to representable arrows, the
logical framework has a notion of a \emph{representable type}. We
write \(A : \largetype\) when \(A\) is a type and \(A : \smalltype\)
when \(A\) is a representable type. \(\smalltype\) is considered to be
a subsort of \(\largetype\) in the sense that the rule
\[
  \inferrule
  {\Gamma \vdash A : \smalltype}
  {\Gamma \vdash A : \largetype}
\]
is derivable. Pushforwards along representable arrows correspond to
dependent product types of the form
\[
  \inferrule
  {\Gamma \vdash A : \smalltype \\
    \Gamma, x : A \vdash B : \largetype}
  {\Gamma \vdash (x : A) \to B : \largetype}.
\]

The way to encode a type theory in our framework is to declare
\emph{symbols}. Each symbol \(\alpha\) must have its
\emph{context} \(\Gamma\) and \emph{sort} \(s\). When \(\alpha\) is a
symbol of a sort \(s\) over a context \(\Gamma\), we write
\(\alpha : \Gamma \To s\). So an encoding of a type theory is a
well-ordered set of symbols like
\begin{align*}
  \begin{autobreak}
    \MoveEqLeft
    \alpha_{0} :
    \Gamma_{0}
    \To s_{0}
  \end{autobreak}
  \\
  \begin{autobreak}
    \MoveEqLeft
    \alpha_{1} :
    \Gamma_{1}
    \To s_{1}
  \end{autobreak}
  \\
  \begin{autobreak}
    \MoveEqLeft
    \alpha_{2} :
    \Gamma_{2}
    \To s_{2}
  \end{autobreak}
  \\
  \vdots
\end{align*}
which we call a \emph{signature}. \(\Gamma_{i}\) must be a context
defined only using symbols \((\alpha_{j})_{j < i}\). The sort
\(s_{i}\) can be \(\largetype\), \(\smalltype\) or a type \(A\) over
\(\Gamma_{i}\) defined only using \((\alpha_{j})_{j < i}\). An
equation is encoded to a symbol of the form
\(\alpha : \Gamma \To a = b\), but we just write
\(\_ : \Gamma \To a = b\) when the name \(\alpha\) of the equation is
irrelevant.

We give a formal definition of our logical framework in
\cref{sec:formal-definition}. In \cref{sec:coding-type-theories}, we
give several examples of encodings of type theories in our logical
framework. In \cref{sec:synt-repr-map}, the \emph{syntactic
  representable map category} of a signature is constructed and shown
to satisfy an appropriate universal property. We further describe the
\(2\)-category of models of the syntactic representable map category
in \cref{sec:semantic-adequacy}.

\subsection{Formal Definition}
\label{sec:formal-definition}

We assume that we are given an infinite set of \emph{variables}
\(x, y, \dots\) and sufficiently many \emph{symbols}
\(\alpha, \beta, \dots\).

\begin{definition}
  \emph{Pre-terms} are defined by the following grammar:
  \begin{align*}
    A, B, a, b ::= {}
    & \smalltype \mid \largetype \mid \alpha(a_{1}, \dots, a_{n}) \mid x
      \mid \\
    & \Pi(A, x.B) \mid \s{abs}(A, x.b) \mid \s{app}(A, x.B, b, a)
      \mid \\
    & \s{Eq}(A, a, b) \mid \s{refl}_{a}
  \end{align*}
  The expression \(x.B\) means that the variable \(x\) is considered
  to be bound. We always identify \(\alpha\)-equivalent
  pre-terms. We use the following notations:
  \begin{itemize}
  \item \(((x : A) \to B) :\equiv \Pi(A, x.B)\)
  \item \((\lambda(x : A).b) :\equiv \s{abs}(A, x.b)\)
  \item \((a =_{A} b) :\equiv \s{Eq}(A, a, b)\)
  \end{itemize}
  We write \(b a\) for \(\s{app}(A, x.B, b, a)\) when the terms \(A\)
  and \(x.B\) are clear from the context. The type annotations in
  \(\lambda(x : A).b\) and \(a =_{A} b\) are often omitted, and we
  simply write \(\lambda x.b\) and \(a = b\) respectively. For
  pre-terms \(a, b\) and a variable \(x\), the \emph{substitution}
  \(a[b/x]\) is defined in the ordinary way.
\end{definition}

\begin{definition}
  A \emph{pre-context} is a finite sequence of the form
  \[
    (x_{1} : A_{1}, \dots, x_{n} : A_{n})
  \]
  with pre-terms \(A_{1}, \dots, A_{n}\) and distinct variables
  \(x_{1}, \dots, x_{n}\). We denote pre-contexts by \(\Gamma,
  \Delta, \dots\). We write \(x \in \Gamma\) when \(\Gamma = (x_{1} :
  A_{1}, \dots, x_{n} : A_{n})\) and \(x = x_{i}\) for some \(i\).
\end{definition}

\begin{definition}
  A \emph{pre-signature} \(\Sigma\) consists of a well-ordered set
  \(|\Sigma|\) of symbols and a function \(\Sigma\) that assigns to each symbol in
  \(|\Sigma|\) a pair consisting of a pre-context and a pre-term. We write
  \[
    \Sigma = (\alpha_{0} : \Gamma_{0} \To A_{0}, \alpha_{1} : \Gamma_{1} \To A_{1}, \alpha_{2} :
    \Gamma_{2} \To A_{2}, \dots)
  \]
  to mean that
  \(|\Sigma| = \{\alpha_{0} < \alpha_{1} < \alpha_{2} < \dots\}\) and the value
  \(\Sigma(\alpha_{i})\) is \((\Gamma_{i}, A_{i})\) for each
  \(\alpha_{i} \in |\Sigma|\). We write
  \((\alpha : \Gamma \To A) \in \Sigma\) when \(\alpha \in |\Sigma|\) and
  \(\Sigma(\alpha) = (\Gamma, A)\). For a symbol
  \(\alpha \in |\Sigma|\), we write \(\Sigma|_{\alpha}\) for the restriction of
  \(\Sigma\) to \(\{\beta \in |\Sigma| \mid \beta < \alpha\}\).
\end{definition}

\begin{definition}
  A \emph{judgment} is one of the following forms.
  \begin{mathpar}
    \Sigma \vdash \sig
    \and
    \Sigma \mid \Gamma \vdash \ctx
    \and
    \Sigma \mid \Gamma \vdash a : A
    \and
    \Sigma \mid \Gamma \vdash a \equiv b : A
  \end{mathpar}
  For pre-contexts \(\Gamma\) and \(\Delta = (y_{1} : B_{1}, \dots,
  y_{m} : B_{m})\) and a finite sequence of pre-terms \(f = (f_{1},
  \dots, f_{m})\), we write \(\Sigma \mid f : \Gamma \to \Delta\) for
  the finite sequence of judgments
  \[
    ((\Sigma \mid \Gamma \vdash f_{1} : B_{1}), \dots, (\Sigma \mid
    \Gamma \vdash f_{m} : B_{m}[f_{1}/y_{1}, \dots,
    f_{m-1}/y_{m-1}])).
  \]
  For such \(\Sigma \mid f : \Gamma \to \Delta\), we write \([f]\) for
  the substitution operator \([f_{1}/y_{1}, \dots, f_{m}/y_{m}]\). We
  define the set of \emph{legal judgments} to be the smallest set
  of judgments closed under the rules listed in
  \cref{fig:legal-judgments}. Here we omit the obvious rules for
  \(\equiv\) to be a congruence relation.
  \begin{figure}
    \begin{mathpar}
      \inferrule
      {\Sigma \mid \Gamma \vdash \ctx}
      {\Sigma, \alpha : \Gamma \To s \vdash \sig}
      \ (\alpha \not\in |\Sigma|, s = \smalltype, \largetype)
      \and
      \inferrule
      {\Sigma \mid \Gamma \vdash A : \largetype}
      {\Sigma, \alpha : \Gamma \To A \vdash \sig}
      \ (\alpha \not\in |\Sigma|)
      \and
      \inferrule
      {(\Sigma|_{\alpha} \vdash \sig)_{\alpha \in |\Sigma|}}
      {\Sigma \vdash \sig}
      \ (\text{\(|\Sigma|\) is unbounded})
      \and
      \inferrule
      {\Sigma \vdash \sig}
      {\Sigma \mid () \vdash \ctx}
      \and
      \inferrule
      {\Sigma \mid \Gamma \vdash A : \largetype}
      {\Sigma \mid \Gamma, x : A \vdash \ctx}
      \ (x \not\in \Gamma)
      \and
      \inferrule
      {\Sigma \mid \Gamma \vdash A : \smalltype}
      {\Sigma \mid \Gamma \vdash A : \largetype}
      \and
      \inferrule
      {\Sigma \mid \Gamma \vdash a : A \\
        \Sigma \mid \Gamma \vdash A \equiv B : \largetype}
      {\Sigma \mid \Gamma \vdash a : B}
      \and
      \inferrule
      {\Sigma \mid \Gamma \vdash a \equiv b : A \\
        \Sigma \mid \Gamma \vdash A \equiv B : \largetype}
      {\Sigma \mid \Gamma \vdash a \equiv b : B}
      \and
      \inferrule
      {\Sigma \mid f : \Gamma \to \Delta}
      {\Sigma \mid \Gamma \vdash \alpha(f) : A[f]}
      \ ((\alpha : \Delta \To A) \in \Sigma)
      \and
      \inferrule
      {\Sigma \mid \Gamma, x : A, \Delta \vdash \ctx}
      {\Sigma \mid \Gamma, x : A, \Delta \vdash x : A}
      \and
      \inferrule
      {\Sigma \mid \Gamma \vdash A : \smalltype \\
        \Sigma \mid \Gamma, x : A \vdash B : \largetype}
      {\Sigma \mid \Gamma \vdash (x : A) \to B : \largetype}
      \and
      \inferrule
      {\Sigma \mid \Gamma \vdash A : \smalltype \\
        \Sigma \mid \Gamma, x : A \vdash b : B}
      {\Sigma \mid \Gamma \vdash \lambda(x : A).b : (x : A) \to B}
      \and
      \inferrule
      {\Sigma \mid \Gamma \vdash A : \smalltype \\
        \Sigma \mid \Gamma, x : A \vdash B : \largetype \\
        \Sigma \mid \Gamma \vdash b : (x : A) \to B \\
        \Sigma \mid \Gamma \vdash a : A}
      {\Sigma \mid \Gamma \vdash \s{app}(A, x.B, b, a) : B[a/x]}
      \and
      \inferrule
      {\Sigma \mid \Gamma \vdash A : \smalltype \\
        \Sigma \mid \Gamma, x : A \vdash b : B \\
        \Sigma \mid \Gamma \vdash a : A}
      {\Sigma \mid \Gamma \vdash \s{app}(A, x.B, \lambda(x : A). b, a)
        \equiv b[a/x] : B[a/x]}
      \and
      \inferrule
      {\Sigma \mid \Gamma \vdash b : (x : A) \to B \\
        \Sigma \mid \Gamma \vdash b' : (x : A) \to B \\
        \Sigma \mid \Gamma, x' : A \vdash \s{app}(A, x.B, b, x')
        \equiv \s{app}(A, x.B, b', x') : B[x'/x]}
      {\Sigma \mid \Gamma \vdash b \equiv b' : (x : A) \to B}
      \and
      \inferrule
      {\Sigma \mid \Gamma \vdash A : \largetype \\
        \Sigma \mid \Gamma \vdash a : A \\
        \Sigma \mid \Gamma \vdash b : A}
      {\Sigma \mid \Gamma \vdash a =_{A} b : \largetype}
      \and
      \inferrule
      {\Sigma \mid \Gamma \vdash A : \largetype \\
        \Sigma \mid \Gamma \vdash a : A}
      {\Sigma \mid \Gamma \vdash \s{refl}_{a} : a =_{A} a}
      \and
      \inferrule
      {\Sigma \mid \Gamma \vdash c : a =_{A} b}
      {\Sigma \mid \Gamma \vdash a \equiv b : A}
      \and
      \inferrule
      {\Sigma \mid \Gamma \vdash b : a =_{A} a}
      {\Sigma \mid \Gamma \vdash b \equiv \s{refl}_{a} : a =_{A} a}
    \end{mathpar}
    \caption{\label{fig:legal-judgments}
      Legal judgments}
  \end{figure}
\end{definition}

\begin{definition}
  A \emph{signature} is a pre-signature \(\Sigma\) such that
  \(\Sigma \vdash \sig\) is a legal judgment. A \emph{context over
    \(\Sigma\)} is a pre-context \(\Gamma\) such that
  \(\Sigma \mid \Gamma \vdash \ctx\) is a legal judgment. For contexts
  \(\Gamma\) and \(\Delta\) over \(\Sigma\), a \emph{context morphism}
  \(\Gamma \to \Delta\) is a finite sequence of pre-terms \(f\)
  such that \(\Sigma \mid f : \Gamma \to \Delta\) is a finite sequence
  of legal judgments. Assume
  \(\Delta = (y_{1} : B_{1}, \dots, y_{m} : B_{m})\).  We say context
  morphisms \(f, g : \Gamma \to \Delta\) are \emph{equivalent},
  written \(f \equiv g\), if
  \((\Sigma \mid \Gamma \vdash f_{1} \equiv g_{1} : B_{1}), \dots,
  (\Sigma \mid \Gamma \vdash f_{m} \equiv g_{m} : B_{m}[f_{1}/y_{1},
  \dots, f_{m-1}/y_{m-1}])\) are legal judgments. A \emph{type} over a
  context \(\Gamma\) over a signature \(\Sigma\) is a pre-term
  \(A\) such that \(\Sigma \mid \Gamma \vdash A : \largetype\) is a
  legal judgment. We say a type \(A\) is \emph{representable} if
  \(\Sigma \mid \Gamma \vdash A : \smalltype\) is a legal
  judgment. For a type \(A\), a \emph{term} of \(A\) is a pre-term
  \(a\) such that \(\Sigma \mid \Gamma \vdash a : A\) is a legal
  judgment.
\end{definition}

Our logical framework has the usual weakening and substitution
properties. The following weakening on signatures might be
non-standard since signatures can be infinite.

\begin{proposition}[Weakening on signatures]
  Let \(\Sigma, \Sigma', \Sigma''\) be pre-signatures with pairwise
  disjoint domains. If \(\Sigma, \Sigma' \vdash \sig\) and
  \(\Sigma, \Sigma'' \vdash \judg\) are legal judgments, then so is
  \(\Sigma, \Sigma', \Sigma'' \vdash \judg\), where
  \(\Sigma \vdash \judg\) denotes a judgment of the form
  \(\Sigma \vdash \sig\), \(\Sigma \mid \Gamma \vdash \ctx\),
  \(\Sigma \mid \Gamma \vdash a : A\) or
  \(\Sigma \mid \Gamma \vdash a \equiv b : A\).
\end{proposition}
\begin{proof}
  By induction on the derivation of \(\Sigma, \Sigma'' \vdash
  \judg\).
\end{proof}

\subsection{Coding Type Theories}
\label{sec:coding-type-theories}

We give several examples of encodings of type theories in our logical
framework.

\begin{example}
  \label{exm:sig-basic-dtt}
  We define \(\Sigma_{DTT}\) to be the following signature.
  \begin{align*}
    \begin{autobreak}
      \MoveEqLeft
      \s{Type} :
      ()
      \To \largetype
    \end{autobreak}
    \\
    \begin{autobreak}
      \MoveEqLeft
      \s{el} :
      (A : \s{Type})
      \To \smalltype
    \end{autobreak}
  \end{align*}
  (\(\s{el}\) stands for \underline{el}ement.) We call
  \(\Sigma_{DTT}\) the \emph{signature for basic dependent type theory}. We
  will see that \(\Sigma_{DTT}\) represents the basic dependent type theory
  \(\thgat\) (\cref{exm:dtt-ump}).
\end{example}

We consider extending \(\Sigma_{DTT}\) by adjoining type
constructors.

\begin{example}
  \(\Pi\)-types in \(\s{Type}\) are encoded as follows.
  \begin{align*}
    \begin{autobreak}
      \MoveEqLeft
      \Pi :
      (A : \s{Type},
      B : \s{el}(A) \to \s{Type})
      \To \s{Type}
    \end{autobreak}
    \\
    \begin{autobreak}
      \MoveEqLeft
      \s{abs} :
      (A : \s{Type},
      B : \s{el}(A) \to \s{Type},
      b : (x : \s{el}(A)) \to \s{el}(B x))
      \To \s{el}(\Pi(A, B))
    \end{autobreak}
    \\
    \begin{autobreak}
      \MoveEqLeft
      \s{app} :
      (A : \s{Type},
      B : \s{el}(A) \to \s{Type},
      b : \s{el}(\Pi(A, B)),
      a : \s{el}(A))
      \To \s{el}(B a)
    \end{autobreak}
    \\
    \begin{autobreak}
      \MoveEqLeft
      \_ :
      (A : \s{Type},
      B : \s{el}(A) \to \s{Type},
      b : (x : \s{el}(A)) \to \s{el}(B x),
      a : \s{el}(A))
      \To \s{app}(A, B, \s{abs}(A, B, b), a) = b a
    \end{autobreak}
    \\
    \begin{autobreak}
      \MoveEqLeft
      \_ :
      (A : \s{Type},
      B : \s{el}(A) \to \s{Type},
      b : \s{el}(\Pi(A, B)),
      b' : \s{el}(\Pi(A, B)),
      p : (x : \s{el}(A)) \to \s{app}(A, B, b, x) = \s{app}(A, B, b', x))
      \To b = b'
    \end{autobreak}
  \end{align*}
\end{example}

\begin{example}
  Identity types in \(\s{Type}\) are encoded as follows.
  \begin{align*}
    \begin{autobreak}
      \MoveEqLeft
      \s{Id} :
      (A : \s{Type},
      a : \s{el}(A),
      b : \s{el}(A))
      \To \s{Type}
    \end{autobreak}
    \\
    \begin{autobreak}
      \MoveEqLeft
      \s{refl} :
      (A : \s{Type},
      a : \s{el}(A))
      \To \s{el}(\s{Id}(A, a, a))
    \end{autobreak}
    \\
    \begin{autobreak}
      \MoveEqLeft
      \s{ind_{Id}} :
      (A : \s{Type},
      a : \s{el}(A),
      b : \s{el}(A),
      p : \s{el}(\s{Id}(A, a, b)),
      C : (x : \s{el}(A), y : \s{el}(\s{Id}(A, a, x))) \to \s{Type},
      c : \s{el}(C(a, \s{refl}(A, a))))
      \To \s{el}(C(b, p))
    \end{autobreak}
    \\
    \begin{autobreak}
      \MoveEqLeft
      \_ :
      (A : \s{Type},
      a : \s{el}(A),
      C : (x : \s{el}(A), y : \s{el}(\s{Id}(A, a, x))) \to \s{Type},
      c : \s{el}(C(a, \s{refl}(A, a))))
      \To \s{ind_{Id}}(A, a, a, \s{refl}(A, a), C, c) = c
    \end{autobreak}
  \end{align*}
  Equality types are encoded as identity types with the following
  equations called equality reflection.
  \begin{align*}
    \begin{autobreak}
      \MoveEqLeft
      \_ :
      (A : \s{Type},
      a : \s{el}(A),
      b : \s{el}(A),
      p : \s{el}(\s{Id}(A, a, b)))
      \To a = b
    \end{autobreak}
    \\
    \begin{autobreak}
      \MoveEqLeft
      \_ :
      (A : \s{Type},
      a : \s{el}(A),
      p : \s{el}(\s{Id}(A, a, a)))
      \To p = \s{refl}(A, a)
    \end{autobreak}
  \end{align*}
\end{example}

\begin{example}
  A universe (\`{a} la Tarski) is encoded by the following symbols.
  \begin{align*}
    \begin{autobreak}
      \MoveEqLeft
      \s{U} :
      ()
      \To \s{Type}
    \end{autobreak}
    \\
    \begin{autobreak}
      \MoveEqLeft
      \s{el_{U}} :
      (A : \s{el}(\s{U}))
      \To \s{Type}
    \end{autobreak}
  \end{align*}
  For nested universes \(\s{U}_{0} : \s{U}_{1}\), we add two pairs
  of such symbols \((\s{U}_{0}, \s{el}_{\s{U}_{0}})\) and
  \((\s{U}_{1}, \s{el}_{\s{U}_{1}})\) and a ``name'' of \(\s{U}_{0}\)
  in \(\s{U}_{1}\):
  \begin{align*}
    \begin{autobreak}
      \MoveEqLeft
      \s{u}_{0} :
      ()
      \To \s{el}(\s{U}_{1})
    \end{autobreak}
    \\
    \begin{autobreak}
      \MoveEqLeft
      \_ :
      ()
      \To \s{el}_{\s{U}_{1}}(\s{u}_{0}) = \s{U}_{0}.
    \end{autobreak}
  \end{align*}
  Since our logical framework can have an infinitely long signature,
  one can encode infinitely many universes
  \[
    \s{U}_{0} : \s{U}_{1} : \s{U}_{2} : \dots.
  \]

  One can add dependent products on \(\s{U}\) in two ways. In both
  ways we add a type constructor
  \[
    \Pi_{\s{U}} : (A : \s{el}(\s{U}), B : \s{el}(\s{el_{U}}(A)) \to
    \s{el}(\s{U})) \To \s{el}(\s{U}).
  \]
  One way is to add an equation
  \[
    \_ : (A : \s{el}(\s{U}), B :
    \s{el}(\s{el_{U}}(A)) \to \s{el}(\s{U})) \To
    \s{el_{U}}(\Pi_{\s{U}}(A, B)) = \Pi(\s{el_{U}}(A), \lambda
    x.\s{el_{U}}(B x))
  \]
  assuming \(\s{Type}\) has dependent products. The other way is to
  add symbols and equations in a similar manner to dependent
  products in \(\s{Type}\). In the latter way the equation
  \[
    A : \s{el}(\s{U}), B : \s{el}(\s{el_{U}}(A)) \to \s{el}(\s{U})
    \vdash \s{el_{U}}(\Pi_{\s{U}}(A, B)) \equiv \Pi(\s{el_{U}}(A),
    \lambda x.\s{el_{U}}(B x)) : \s{Type}
  \]
  need not hold, but one can show that \(\s{el_{U}}(\Pi_{\s{U}}(A,
  B))\) and \(\Pi(\s{el_{U}}(A), \lambda x.\s{el_{U}}(B x))\) are
  isomorphic in an appropriate sense.
\end{example}

\begin{example}
  Various \emph{two-level type theories}
  \parencite{altenkirch2016extending,annenkov2017two-level,voevodsky2013simple}
  have a sort of \emph{fibrant types} as well as a sort of types. We
  extend \(\Sigma_{DTT}\) by the symbols
  \begin{align*}
    \begin{autobreak}
      \MoveEqLeft
      \s{Fib} :
      ()
      \To \largetype
    \end{autobreak}
    \\
    \begin{autobreak}
      \MoveEqLeft
      \iota :
      (A : \s{Fib})
      \To \s{Type}
    \end{autobreak}
  \end{align*}
  and several type constructors. For readability, we think of
  \(\s{Fib}\) as a subtype of \(\s{Type}\) and often omit \(\iota\). A
  major difference between \(\s{Type}\) and \(\s{Fib}\) is that
  identity types \(\s{Id_{Type}}(A, a, b)\) in \(\s{Type}\) satisfy
  axiom K or equality reflection but identity types
  \(\s{Id_{Fib}}(A, a, b)\) in \(\s{Fib}\) do not. The induction
  principle for identity types in \(\s{Fib}\) only works for families
  of \(\s{Fib}\):
  \begin{align*}
    \begin{autobreak}
      \MoveEqLeft
      \s{ind_{Id_{Fib}}} :
      (A : \s{Fib},
      a : \s{el}(A),
      b : \s{el}(A),
      p : \s{el}(\s{Id_{Fib}}(A, a, b)),
      C : (x : \s{el}(A), y : \s{el}(\s{Id_{Fib}}(A, a, x))) \to \s{Fib},
      c : \s{el}(C(a, \s{refl}(A, a))))
      \To \s{el}(C(b, p)).
    \end{autobreak}
  \end{align*}
\end{example}

\begin{example}
  To encode propositional logic, we begin with the following signature.
  \begin{align*}
    \begin{autobreak}
      \MoveEqLeft
      \s{Prop} :
      ()
      \To \largetype
    \end{autobreak}
    \\
    \begin{autobreak}
      \MoveEqLeft
      \s{true} :
      (P : \s{Prop})
      \To \smalltype
    \end{autobreak}
    \\
    \begin{autobreak}
      \MoveEqLeft
      \s{mono} :
      (P : \s{Prop},
      x : \s{true}(P),
      y : \s{true}(P))
      \To x = y
    \end{autobreak}
  \end{align*}
  The equation \(\s{mono}\) implies that \(\s{true}(P)\) has at most
  one element. One may extend this signature by adding logical
  connectives like \(\top\), \(\bot\), \(\land\), \(\lor\) and
  \(\supset\). For example, \(\land\) and \(\lor\) are encoded as follows.
  \begin{align*}
    \begin{autobreak}
      \MoveEqLeft
      {\land} :
      (P : \s{Prop},
      Q : \s{Prop})
      \To \s{Prop}
    \end{autobreak}
    \\
    \begin{autobreak}
      \MoveEqLeft
      \s{in}_{\land} :
      (P : \s{Prop},
      Q : \s{Prop},
      p : \s{true}(P),
      q : \s{true}(Q))
      \To \s{true}({\land}(P, Q))
    \end{autobreak}
    \\
    \begin{autobreak}
      \MoveEqLeft
      \s{out}_{{\land}, 1} :
      (P : \s{Prop},
      Q : \s{Prop},
      r : \s{true}({\land}(P, Q)))
      \To \s{true}(P)
    \end{autobreak}
    \\
    \begin{autobreak}
      \MoveEqLeft
      \s{out}_{{\land}, 2} :
      (P : \s{Prop},
      Q : \s{Prop},
      r : \s{true}({\land}(P, Q)))
      \To \s{true}(Q)
    \end{autobreak}
    \\
    \begin{autobreak}
      \MoveEqLeft
      {\lor} :
      (P : \s{Prop},
      Q : \s{Prop})
      \To \s{Prop}
    \end{autobreak}
    \\
    \begin{autobreak}
      \MoveEqLeft
      \s{in}_{{\lor}, 1} :
      (P : \s{Prop},
      Q : \s{Prop},
      p : \s{true}(P))
      \To \s{true}({\lor}(P, Q))
    \end{autobreak}
    \\
    \begin{autobreak}
      \MoveEqLeft
      \s{in}_{{\lor}, 2} :
      (P : \s{Prop},
      Q : \s{Prop},
      q : \s{true}(Q))
      \To \s{true}({\lor}(P, Q))
    \end{autobreak}
    \\
    \begin{autobreak}
      \MoveEqLeft
      \s{out}_{\lor} :
      (P : \s{Ptop},
      Q : \s{Prop},
      r : \s{true}({\lor}(P, Q)),
      R : \s{Prop},
      p : \s{true}(P) \to \s{true}(R),
      q : \s{true}(Q) \to \s{true}(R))
      \To \s{true}(R)
    \end{autobreak}
  \end{align*}
\end{example}

\begin{example}
  \label{exm:predicate-logic}
  To encode predicate logic, we extend the union of the signatures for
  basic dependent type theory and propositional logic by adding
  equality and quantifiers. For example, equality and \(\forall\) are
  encoded as follows.
  \begin{align*}
    \begin{autobreak}
      \MoveEqLeft
      \s{eq} :
      (A : \s{Type},
      a : \s{el}(A),
      b : \s{el}(B))
      \To \s{Prop}
    \end{autobreak}
    \\
    \begin{autobreak}
      \MoveEqLeft
      \s{in_{eq}} :
      (A : \s{Type},
      a : \s{el}(A))
      \To \s{true}(\s{eq}(A, a, a))
    \end{autobreak}
    \\
    \begin{autobreak}
      \MoveEqLeft
      \s{out_{eq}} :
      (A : \s{Type},
      a : \s{el}(A),
      b : \s{el}(B),
      p : \s{eq}(A, a, b),
      Q : (x : \s{el}(A)) \to \s{Prop},
      q : \s{true}(Q a))
      \To \s{true}(Q b)
    \end{autobreak}
    \\
    \begin{autobreak}
      \MoveEqLeft
      \forall :
      (A : \s{Type},
      P : \s{el}(A) \to \s{Prop})
      \To \s{Prop}
    \end{autobreak}
    \\
    \begin{autobreak}
      \MoveEqLeft
      \s{in}_{\forall} :
      (A : \s{Type},
      P : \s{el}(A) \to \s{Prop},
      p : (x : \s{el}(A)) \to \s{true}(P x))
      \To \s{true}(\forall(A, P))
    \end{autobreak}
    \\
    \begin{autobreak}
      \MoveEqLeft
      \s{out}_{\forall} :
      (A : \s{Type},
      P : \s{el}(A) \to \s{Prop},
      p : \s{true}(\forall(A, P)),
      a : \s{el}(A))
      \To \s{true}(P a)
    \end{autobreak}
  \end{align*}
  In this encoding a term can depend on the validity of a proposition,
  allowing us to write a partial function. For example, a term of type
  \((x : \s{el}(A)) \to \s{true}(P x) \to \s{el}(B)\) is a partial
  function from \(A\) to \(B\) defined on those \(a : \s{el}(A)\)
  satisfying \(P\), where \(A : \s{Type}\), \(B : \s{Type}\) and
  \(P : \s{el}(A) \to \s{Prop}\). One may add the following symbols
  so that \(\bot\) becomes the initial object and \(P \lor Q\) the
  pushout of \(P\) and \(Q\) under \(P \land Q\).
  \begin{align*}
    \begin{autobreak}
      \MoveEqLeft
      \s{elim}^{\s{Type}}_{\bot} :
      (p : \s{true}(\bot),
      A : \s{Type})
      \To \s{el}(A)
    \end{autobreak}
    \\
    \begin{autobreak}
      \MoveEqLeft
      \_ :
      (p : \s{true}(\bot),
      A : \s{Type},
      a : \s{el}(A),
      b : \s{el}(A))
      \To a = b
    \end{autobreak}
    \\
    \begin{autobreak}
      \MoveEqLeft
      \s{elim}^{\s{Type}}_{\lor} :
      (P : \s{Prop},
      Q : \s{Prop},
      r : \s{true}({\lor}(P, Q)),
      A : \s{Type},
      a : \s{true}(P) \to \s{el}(A),
      b : \s{true}(Q) \to \s{el}(A),
      s : (p : \s{true}(P), q : \s{true}(Q)) \to a p = b q)
      \To \s{el}(A)
    \end{autobreak}
    \\
    \begin{autobreak}
      \MoveEqLeft
      \_ :
      (P : \s{Prop},
      Q : \s{Prop},
      p : \s{true}(P),
      A : \s{Type},
      a : \s{el}(A),
      b : \s{true}(Q) \to \s{el}(A),
      s : (q : \s{true}(Q)) \to a = b q)
      \To \s{elim}^{\s{Type}}_{\lor}(P, Q, \s{in}_{{\lor}, 1}(P, Q, p), A, \lambda x.a, b, \lambda xy.sy) = a
    \end{autobreak}
    \\
    \begin{autobreak}
      \MoveEqLeft
      \_ :
      (P : \s{Prop},
      Q : \s{Prop},
      q : \s{true}(Q),
      A : \s{Type},
      a : \s{true}(P) \to \s{el}(A),
      b : \s{el}(A),
      s : (p : \s{true}(P)) \to a p = b)
      \To \s{elim}^{\s{Type}}_{\lor}(P, Q, \s{in}_{{\lor}, 2}(P, Q, q), A, a, \lambda y.b, \lambda xy.sx) = b
    \end{autobreak}
    \\
    \begin{autobreak}
      \MoveEqLeft
      \_ :
      (P : \s{Prop},
      Q : \s{Prop},
      r : \s{true}({\lor}(P, Q)),
      A : \s{Type},
      a : \s{el}(A),
      b : \s{el}(A),
      s : \s{true}(P) \to a = b,
      t : \s{true}(Q) \to a = b)
      \To a = b
    \end{autobreak}
  \end{align*}
\end{example}

\begin{example}
  \label{exm:cubical-type-theory}
  \emph{Cubical type
    theory} \parencite{cohen2016cubical} is an extension of dependent type
  theory with a formal interval and cofibrant predicates. So we
  extend the signature for basic dependent type theory with
  \begin{align*}
    \begin{autobreak}
      \MoveEqLeft
      \I :
      ()
      \To \smalltype
    \end{autobreak}
    \\
    \begin{autobreak}
      \MoveEqLeft
      \s{Cof} :
      ()
      \To \largetype
    \end{autobreak}
    \\
    \begin{autobreak}
      \MoveEqLeft
      \s{true} :
      (P : \s{Cof})
      \To \smalltype
    \end{autobreak}
    \\
    \begin{autobreak}
      \MoveEqLeft
      \_ :
      (P : \s{Cof},
      x : \s{true}(P),
      y : \s{true}(P))
      \To x = y.
    \end{autobreak}
  \end{align*}
  \(\I\) carries a de Morgan algebra structure
  \((0, 1, \sqcap, \sqcup, (-)')\), and \(\s{Cof}\) has logical
  connectives \(\top, \land, \bot, \lor\) and equalities and
  quantifiers of the form
  \begin{align*}
    \begin{autobreak}
      \MoveEqLeft
      \s{eq}_{0} :
      (i : \I)
      \To \s{Cof}
    \end{autobreak}
    \\
    \begin{autobreak}
      \MoveEqLeft
      \s{eq}_{1} :
      (i : \I)
      \To \s{Cof}
    \end{autobreak}
    \\
    \begin{autobreak}
      \MoveEqLeft
      \forall_{\I} :
      (P : \I \to \s{Cof})
      \To \s{Cof}.
    \end{autobreak}
  \end{align*}
  Moreover, \(\s{eq}_{0}\) and \(\s{eq}_{1}\) satisfy equality
  reflection, elimination operators \(\s{elim}_{\bot}^{\s{Type}}\) and
  \(\s{elim}_{\lor}^{\s{Type}}\) are added as in
  \cref{exm:predicate-logic}, and \(\s{Cof}\) satisfies propositional
  extensionality:
  \begin{align*}
    \begin{autobreak}
      \MoveEqLeft
      \_ :
      (P : \s{Cof},
      Q : \s{Cof},
      f : \s{true}(P) \to \s{true}(Q),
      g : \s{true}(Q) \to \s{true}(P))
      \To P = Q.
    \end{autobreak}
  \end{align*}
  The \emph{composition operation} is encoded as follows.
  \begin{align*}
    \begin{autobreak}
      \MoveEqLeft
      \s{comp} :
      (A : \I \to \s{Type},
      P : \s{Cof},
      a : \s{true}(P) \to (i : \I) \to \s{el}(A i),
      a_{0} : \s{el}(A 0),
      q : (x : \s{true}(P)) \to a x 0 = a_{0})
      \To \s{el}(A 1)
    \end{autobreak}
    \\
    \begin{autobreak}
      \MoveEqLeft
      \_ :
      (A : \I \to \s{Type},
      P : \s{Cof},
      p : \s{true}(P),
      a : (i : \I) \to \s{el}(A i))
      \To \s{comp}(A, P, \lambda xi.ai, a 0, \lambda x.\s{refl}_{a 0}) = a 1
    \end{autobreak}
  \end{align*}
  Note that the type of \(\s{comp}\) is essentially the same as the
  type of composition structures in the axiomatic approach to the
  semantics of cubical type theory given by
  \textcite{orton2018axioms}. The \emph{gluing operation} is encoded
  as follows, assuming that \(\s{Type}\) has enough type constructors
  to define the type \(\s{Equiv}(A, B) : \s{Type}\) of equivalences
  between types \(A : \s{Type}\) and \(B : \s{Type}\).
  \begin{align*}
    \begin{autobreak}
      \MoveEqLeft
      \s{Glue} :
      (P : \s{Cof},
      A : \s{true}(P) \to \s{Type},
      B : \s{Type},
      f : (x : \s{true}(P)) \to \s{el}(\s{Equiv}(A x, B)))
      \To \s{Type}
    \end{autobreak}
    \\
    \begin{autobreak}
      \MoveEqLeft
      \_ :
      (A : \s{Type},
      B : \s{Type},
      f : \s{el}(\s{Equiv}(A, B)))
      \To \s{Glue}(\top, \lambda \_.A, B, \lambda \_.f) = A
    \end{autobreak}
    \\
    \begin{autobreak}
      \MoveEqLeft
      \s{unglue} :
      (P : \s{Cof},
      A : \s{true}(P) \to \s{Type},
      B : \s{Type},
      f : (x : \s{true}(P)) \to \s{el}(\s{Equiv}(A x, B)),
      a : \s{el}(\s{Glue}(P, A, B, f)))
      \To \s{el}(B)
    \end{autobreak}
    \\
    \begin{autobreak}
      \MoveEqLeft
      \_ :
      (A : \s{Type},
      B : \s{Type},
      f : \s{el}(\s{Equiv}(A, B)),
      a : A)
      \To \s{unglue}(\top, \lambda \_.A, B, \lambda \_.f, a) = f a
    \end{autobreak}
    \\
    \begin{autobreak}
      \MoveEqLeft
      \s{glue} :
      (P : \s{Cof},
      A : \s{true}(P) \to \s{Type},
      B : \s{Type},
      f : (x : \s{true}(P)) \to \s{el}(\s{Equiv}(A x, B)),
      a : (x : \s{true}(P)) \to \s{el}(A x),
      b : \s{el}(B),
      p : (x : \s{true}(P)) \to f x a = b)
      \To \s{el}(\s{Glue}(P, A, B, f))
    \end{autobreak}
    \\
    \begin{autobreak}
      \MoveEqLeft
      \_ :
      (A : \s{Type},
      B : \s{Type},
      f : \s{el}(\s{Equiv}(A, B)),
      a : A)
      \To \s{glue}(\top, \lambda \_.A, B, \lambda \_.f, \lambda \_.a, a, \lambda \_.\s{refl}_{\_}) = a
    \end{autobreak}
    \\
    \begin{autobreak}
      \MoveEqLeft
      \_ :
      (P : \s{Cof},
      A : \s{true}(P) \to \s{Type},
      B : \s{Type},
      f : (x : \s{true}(P)) \to \s{el}(\s{Equiv}(A x, B)),
      a : \s{el}(\s{Glue}(P, A, B, f)))
      \To \s{glue}(P, A, B, f, \lambda \_. a, \s{unglue}(P, A, B, f, a), \lambda \_.\s{refl}_{\_}) = a
    \end{autobreak}
    \\
    \begin{autobreak}
      \MoveEqLeft
      \_ :
      (P : \s{Cof},
      A : \s{true}(P) \to \s{Type},
      B : \s{Type},
      f : (x : \s{true}(P)) \to \s{el}(\s{Equiv}(A x, B)),
      a : (x : \s{true}(P)) \to \s{el}(A x),
      b : \s{el}(B),
      p : (x : \s{true}(P)) \to f x a = b)
      \To \s{unglue}(P, A, B, f, \s{glue}(P, A, B, f, a, b, p)) = b
    \end{autobreak}
  \end{align*}
\end{example}

\subsection{Syntactic Representable Map Categories}
\label{sec:synt-repr-map}

\begin{definition}
  \label{def:syntactic-representable-map}
  For a signature \(\Sigma\), we define a small category
  \(\sRM(\Sigma)\) as follows.
  \begin{itemize}
  \item The objects are the contexts over \(\Sigma\).
  \item The morphisms \(\Gamma \to \Delta\) are the equivalence
    classes of context morphisms \(\Gamma \to \Delta\).
  \item The identity on
    \(\Gamma = (x_{1} : A_{1}, \dots, x_{n} : A_{n})\) is represented
    by the context morphism \((x_{1}, \dots, x_{n})\).
  \item For morphisms \(f : \Gamma_{1} \to \Gamma_{2}\) and
    \(g : \Gamma_{2} \to \Gamma_{3}\), the composition \(g \circ f\)
    is represented by the substitution \(g[f]\).
  \end{itemize}
  A \emph{generating representable morphism} in \(\sRM(\Sigma)\) is a
  morphism isomorphic to the projection
  \[
    (\Gamma, x : A) \to \Gamma
  \]
  with \(\Sigma \mid \Gamma \vdash A : \smalltype\). A
  \emph{representable morphism} in \(\sRM(\Sigma)\) is a composite of
  generating representable morphisms.
\end{definition}

It is well-known that the syntactic category of a type theory with
dependent product types and equality types is locally
cartesian closed \parencite{seely1984locally}. The same argument shows
that \(\sRM(\Sigma)\) is a representable map category.

\begin{proposition}
  Let \(\Sigma\) be a signature.
  \begin{enumerate}
  \item \(\sRM(\Sigma)\) is a representable map category where the
    representable maps are defined in
    \cref{def:syntactic-representable-map}.
  \item For any \(\alpha \in |\Sigma|\), the functor
    \(\sRM(\Sigma|_{\alpha}) \to \sRM(\Sigma)\) induced by the
    weakening on signatures is a representable map functor.
  \end{enumerate}
  \qed
\end{proposition}

We call \(\sRM(\Sigma)\) the \emph{syntactic representable map category} of
\(\Sigma\). It is the representable map category ``freely generated by
\(\Sigma\)'' formulated by the following universal property.

\begin{theorem}
  \label{thm:syntactic-LF-category}
  Let \(\cat\) be a representable map category and \(\Sigma\) a
  signature.
  \begin{enumerate}
  \item \label{item:6}
    If \(|\Sigma|\) is unbounded, then the functor
    \(|\RMCat(\sRM(\Sigma), \cat)| \to \lim_{\alpha \in
      |\Sigma|}|\RMCat(\sRM(\Sigma|_{\alpha}, \cat))|\) induced by
    the weakening functors \(\sRM(\Sigma|_{\alpha}) \to \sRM(\Sigma)\)
    is an equivalence.
  \item \label{item:7}
    If \(\Sigma = (\Sigma', \alpha : \Gamma \To \largetype)\), then
    for any representable map functor \(\fun : \sRM(\Sigma') \to
    \cat\), the functor \(|(\sRM(\Sigma') \under
    \RMCat)(\sRM(\Sigma), \cat)| \ni \funI \mapsto \funI \alpha
    \in |\cat/\fun \Gamma|\) is an equivalence.
  \item \label{item:8}
    If \(\Sigma = (\Sigma', \alpha : \Gamma \To \smalltype)\), then
    for any representable map functor \(\fun : \sRM(\Sigma') \to
    \cat\), the functor
    \(|(\sRM(\Sigma') \under \RMCat)(\sRM(\Sigma), \cat)| \ni \funI
    \mapsto \funI \alpha \in |(\cat/\fun \Gamma)_{\rep}|\) is an
    equivalence.
  \item \label{item:9}
    If \(\Sigma = (\Sigma', \alpha : \Gamma \To A)\) with
    \(\Sigma' \mid \Gamma \vdash A : \largetype\), then for any
    representable map functor \(\fun : \sRM(\Sigma') \to \cat\), the
    functor \(|(\sRM(\Sigma') \under \RMCat)(\sRM(\Sigma), \cat)| \ni
    \funI \mapsto \funI \alpha \in \cat/\fun \Gamma(1, \fun A)\) is an
    equivalence.
  \end{enumerate}
\end{theorem}

\begin{example}
  \label{exm:dtt-ump}
  Consider the signature
  \(\Sigma_{DTT} = (\s{Type} : () \To \largetype, \s{el} : (A :
  \s{Type}) \To \smalltype)\) from \cref{exm:sig-basic-dtt}. By
  \cref{thm:syntactic-LF-category}, for a representable map category
  \(\cat\), the groupoid \(|\RMCat(\sRM(\Sigma_{DTT}), \cat)|\) is
  equivalent to the following:
  \begin{itemize}
  \item the objects are the representable arrows \(\arr : \obj \to
    \objI\) in \(\cat\);
  \item the arrows are isomorphisms in \(\cat^{\to}\).
  \end{itemize}
  Hence, \(\sRM(\Sigma_{DTT})\) has the same universal property as
  \(\thgat\) (\cref{thm:thgat-ump}), and thus we have an equivalence
  \(\sRM(\Sigma_{DTT}) \simeq \thgat\) in the \(2\)-category
  \(\RMCat\).
\end{example}

The idea of the proof of \cref{thm:syntactic-LF-category} is the same
as the interpretation of extensional Marin-L{\"o}f type theory in
locally cartesian closed categories \parencite{seely1984locally}. The
main difficulty is the \emph{coherence problem}: categorical
structures often satisfy equations up to isomorphism, but equations in
type theory are strict. A solution to the coherence problem
is the splitting technique of \textcite{hofmann1995interpretation}
which replaces a categorical structure by another structure satisfying
equations strictly. We adapt the splitting technique for representable
map categories.

\begin{proof}
  [Sketch of the proof of \cref{thm:syntactic-LF-category}]
  \labelcref{item:6}. If \(\Sigma = ()\), then \(\sRM(\Sigma)\) contains
  only the empty context, which is the terminal object, and thus
  \(\RMCat(\sRM(\Sigma), \cat)\) is contractible. If \(\Sigma\) is
  non-empty and unbounded, then the set of
  objects of \(\sRM(\Sigma)\) is the union of the sets of objects of
  \(\sRM(\Sigma|_{\alpha})\) indexed over \(\alpha \in |\Sigma|\) and,
  for objects \(\Gamma, \Delta \in \sRM(\Sigma|_{\alpha})\), the
  hom-set \(\sRM(\Sigma)(\Gamma, \Delta)\) is the filtered colimit
  \(\colim_{\substack{\beta \in |\Sigma| \\ \alpha <
      \beta}}\sRM(\Sigma|_{\beta})(\Gamma, \Delta)\). From this
  description one can see that the functor
  \(|\RMCat(\sRM(\Sigma), \cat)| \to \lim_{\alpha \in
    |\Sigma|}|\RMCat(\sRM(\Sigma|_{\alpha}), \cat)|\) is in fact an
  isomorphism.

  The others are proved using standard techniques of the semantics of
  dependent type theory. Consider \cref{item:7} which claims that,
  given a representable map functor \(\fun : \sRM(\Sigma') \to \cat\) and an
  object \(A \in \cat/\fun \Gamma\), one can extend \(\fun\) to a
  representable map functor \(\funI : \sRM(\Sigma) \to \cat\) such that
  \(\funI \alpha \cong A\) and such an extension is unique up to unique
  isomorphism. To avoid coherence problems, we use Hofmann's
  splitting technique \parencite{hofmann1995interpretation}. Since
  \(\cat\) has pullbacks, we have the pseudo-functor
  \(\cat/\argu : \cat^{\op} \to
  \Cat\). \Citeauthor{hofmann1995interpretation}'s splitting gives a
  \(2\)-functor \(U : \cat^{\op} \to \Cat\) and a pseudo-natural
  equivalence \( U \to (\cat/\argu)\). We define a \(2\)-functor
  \(R : \cat^{\op} \to \Cat\) by the pullback
  \[
    \begin{tikzcd}
      R
      \arrow[r,"\simeq"]
      \arrow[d,hook] &
      {(\cat/\argu)_{\rep}}
      \arrow[d,hook] \\
      U
      \arrow[r,"\simeq"'] &
      (\cat/\argu).
    \end{tikzcd}
  \]

  We interpret a context \(\Sigma \mid \Gamma \vdash \ctx\) as an object
  \(\funI \Gamma \in \cat\), a type
  \(\Sigma \mid \Gamma \vdash B : \largetype\) as an object
  \(\funI B \in U(\funI \Gamma)\), a representable type
  \(\Sigma \mid \Gamma \vdash B : \smalltype\) as an object
  \(\funI B \in R(\funI \Gamma)\) and a term
  \(\Sigma \mid \Gamma \vdash b : B\) as a section of
  \(\funI B \to \funI \Gamma\). \Textcite{hofmann1995interpretation}
  constructs sufficient structures on \(U\) to interpret equality
  types and dependent product types, when \(\cat\) is locally
  cartesian closed. The same construction works for a representable
  map category \(\cat\) to interpret dependent product types over
  representable types in \(U\) and \(R\). To interpret a symbol from
  \(\Sigma'\), apply \(\fun\) and choose a corresponding element of
  \(U\) via the pseudo-natural equivalence \(U \simeq (\cat/\argu)\). To
  interpret \(\alpha\), choose an element of \(U(\fun \Gamma)\) corresponding to
  \(A \in \cat / \fun \Gamma\). This completes the interpretation of the
  logical framework in \(U\). Applying the pseudo-natural equivalence
  \(U \simeq (\cat/\argu)\), we have a representable map functor
  \(\funI : \sRM(\Sigma) \to \cat\). All the choices made in this
  construction are unique up to unique isomorphism, so \(\funI\) is
  unique up to unique isomorphism.
\end{proof}

\subsection{Models of Syntactic Representable Map Categories}
\label{sec:semantic-adequacy}

Given a signature \(\Sigma\), we concretely describe the \(2\)-category
\(\Mod_{\sRM(\Sigma)}\) (\cref{thm:models-of-syntactic-lf-category}).

Suppose \(\Sigma\) is a signature of the form
\(\Sigma = (\Sigma', \alpha : \Gamma \To \largetype)\). We define a
\(2\)-category \((\DFib \downarrow \Gamma^{(-)})\) as follows:
\begin{itemize}
\item the objects are the pairs \((\model, \alpha^{\model})\)
  consisting of \(\model \in \Mod_{\sRM(\Sigma')}\) and
  \(\alpha^{\model} \in \DFib_{\model}/\Gamma^{\model}\);
\item the morphisms \((\model, \alpha^{\model}) \to (\modelI,
  \alpha^{\modelI})\) are the pairs \((\fun, \fun_{\alpha})\) consisting
  of a morphism \(\fun : \model \to \modelI\) of models of
  \(\sRM(\Sigma')\) and a map \(\fun_{\alpha} : \alpha^{\model} \to
  \alpha^{\modelI}\) of discrete fibrations over \(\fun : \model \to
  \modelI\) such that the diagram
  \begin{equation}
    \begin{tikzcd}
      \alpha^{\model}
      \arrow[r,"\fun_{\alpha}"]
      \arrow[d] &
      \alpha^{\modelI}
      \arrow[d] \\
      \Gamma^{\model}
      \arrow[r,"\fun_{\Gamma}"'] &
      \Gamma^{\modelI}
    \end{tikzcd}
    \label[diagram]{eq:3}
  \end{equation}
  commutes;
\item the \(2\)-morphisms \((\fun, \fun_{\alpha}) \To (\funI, \funI_{\alpha}) :
  (\model, \alpha^{\model}) \to (\modelI,
  \alpha^{\modelI})\) are the \(2\)-morphisms \(\trans : \fun \To \funI\)
  in \(\Mod_{\sRM(\Sigma')}\) such that there exists a (necessarily
  unique) natural transformation \(\trans_{\alpha} : \fun_{\alpha} \To
  \funI_{\alpha}\) over \(\trans\).
\end{itemize}
There is the obvious \(2\)-functor \(\Mod_{\sRM(\Sigma)} \to (\DFib
\downarrow \Gamma^{(-)})\).

When \(\Sigma = (\Sigma', \alpha : \Gamma \To \smalltype)\), the
\(2\)-functor
\(\Mod_{\sRM(\Sigma)} \to (\DFib \downarrow \Gamma^{(-)})\) factors
through the locally full sub-\(2\)-category
\((\DFib \downarrow \Gamma^{(-)})_{\rep} \subset (\DFib \downarrow
\Gamma^{(-)})\) consisting of those objects
\((\model, \alpha^{\model})\) such that
\(\alpha^{\model} \to \Gamma^{\model}\) is a representable map
of discrete fibrations over \(\model\) and those morphisms
\((\fun, \fun_{\alpha}) : (\model, \alpha^{\model}) \to (\modelI,
\alpha^{\modelI})\) such that \cref{eq:3} satisfies the
Beck-Chevalley condition.

Suppose \(\Sigma = (\Sigma', \alpha : \Gamma \To A)\) with \(\Sigma'
\mid \Gamma \vdash A : \largetype\). We define a \(2\)-category
\(\Sect(A^{(-)})\) as follows:
\begin{itemize}
\item the objects are the pairs \((\model, \alpha^{\model})\)
  consisting of a model \(\model\) of \(\sRM(\Sigma')\) and a
  section \(\alpha^{\model}\) of the map \(A^{\model} \to
  \Gamma^{\model}\) of discrete fibrations over \(\model\);
\item the morphisms \((\model, \alpha^{\model}) \to (\modelI,
  \alpha^{\modelI})\) are the morphisms \(\fun : \model \to
  \modelI\) of models of \(\sRM(\Sigma')\) such that the diagram
  \[
    \begin{tikzcd}
      \Gamma^{\model}
      \arrow[r,"\fun_{\Gamma}"]
      \arrow[d,"\alpha^{\model}"'] &
      \Gamma^{\modelI}
      \arrow[d,"\alpha^{\modelI}"] \\
      A^{\model}
      \arrow[r,"\fun_{A}"'] &
      A^{\modelI}
    \end{tikzcd}
  \]
  commutes;
\item the \(2\)-morphisms \(\fun \To \funI\) are the \(2\)-morphisms of
  models of \(\sRM(\Sigma')\).
\end{itemize}
There is the obvious \(2\)-functor \(\Mod_{\sRM(\Sigma')} \to
\Sect(A^{(-)})\).

\begin{theorem}
  \label{thm:models-of-syntactic-lf-category}
  Let \(\Sigma\) be a signature.
  \begin{enumerate}
  \item \label{item:16}
    If \(\Sigma = ()\), then the \(2\)-functor
    \(\Mod_{\sRM(\Sigma)} \to \Cat_{1}\) is a
    bi-equivalence, where \(\Cat_{1}\) denotes the \(2\)-category of
    small categories with terminal objects.
  \item \label{item:12}
    If \(|\Sigma|\) is non-empty and unbounded,
    then the induced \(2\)-functor \(\Mod_{\sRM(\Sigma)} \to
    \lim_{\alpha \in |\Sigma|}\Mod_{\sRM(\Sigma|_{\alpha})}\) is a
    bi-equivalence.
  \item \label{item:13}
    If \(\Sigma = (\Sigma', \alpha : \Gamma \To \largetype)\),
    then the \(2\)-functor \(\Mod_{\sRM(\Sigma)} \to (\DFib \downarrow
    \Gamma^{(-)})\) is a bi-equivalence.
  \item \label{item:14}
    If \(\Sigma = (\Sigma', \alpha : \Gamma \To \smalltype)\),
    then the \(2\)-functor \(\Mod_{\sRM(\Sigma)} \to (\DFib
    \downarrow \Gamma^{(-)})_{\rep}\) is a bi-equivalence.
  \item \label{item:15}
    If \(\Sigma = (\Sigma', \alpha : \Gamma \To A)\) with
    \(\Sigma' \mid \Gamma \vdash A : \largetype\), then the
    \(2\)-functor \(\Mod_{\sRM(\Sigma)} \to \Sect(A^{(-)})\) is a
    bi-equivalence.
  \end{enumerate}
\end{theorem}
\begin{proof}[Sketch of the proof]
  We use the fact that (\(2\)-)morphisms of models of a type theory
  are also representable map functors
  (\cref{sec:anoth-defin-morph}). Then we can apply
  \cref{thm:syntactic-LF-category} for building not only models of
  \(\sRM(\Sigma)\) but also (\(2\)-)morphisms of models of \(\sRM(\Sigma)\).

  Note that the \(2\)-functors in the statement are locally
  faithful. It remains to show that those \(2\)-functors are
  bi-essentially surjective on objects, locally essentially surjective
  on objects and locally full. We only demonstrate the
  statement~\labelcref{item:13}. The others can be proved using the same
  idea.

  To show that the \(2\)-functor
  \(\Mod_{\sRM(\Sigma)} \to (\DFib \downarrow \Gamma^{(-)})\) is
  bi-essentially surjective on objects, let \(\model\) be a model
  of \(\sRM(\Sigma')\) and \(p : A \to \Gamma^{\model}\) a map of
  small discrete fibrations over \(\model\). By
  \cref{thm:syntactic-LF-category}, the representable map
  functor \((-)^{\model} : \sRM(\Sigma') \to \DFib_{\model}\)
  extends to a representable map functor
  \((-)^{\widetilde{\model}} : \sRM(\Sigma) \to \DFib_{\model}\) such
  that \(\alpha^{\widetilde{\model}} \cong A\). This defines a model
  \(\widetilde{\model}\) of \(\sRM(\Sigma)\) such that the restriction of
  \(\widetilde{\model}\) to \(\sRM(\Sigma')\) is isomorphic to
  \(\model\) and \(\alpha^{\widetilde{\model}} \cong A\).

  To show that the \(2\)-functor
  \(\Mod_{\sRM(\Sigma)} \to (\DFib \downarrow \Gamma^{(-)})\) is
  locally essentially surjective on objects, let
  \(\model, \modelI\) be models of \(\sRM(\Sigma)\),
  \(\fun : \model \to \modelI\) a morphism of models of
  \(\sRM(\Sigma')\), and
  \(\funI : \alpha^{\model} \to \alpha^{\modelI}\) a map of discrete
  fibrations such that the diagram
  \[
    \begin{tikzcd}
      {\alpha^{\model}}
      \arrow[r,"\funI"]
      \arrow[d] &
      {\alpha^{\modelI}}
      \arrow[d] \\
      \Gamma^{\model}
      \arrow[r,"\fun_{\Gamma}"'] &
      \Gamma^{\modelI}
    \end{tikzcd}
  \]
  commutes. Then \(\fun_{(-)}\) is a representable map functor
  \(\sRM(\Sigma') \to (\DFib^{\to})_{\fun}\) by \cref{prop:morphism-of-models}
  and \(\funI\) is an object of
  \((\DFib^{\to})_{\fun}/\fun_{\Gamma}\). By
  \cref{thm:syntactic-LF-category}, the representable map functor
  \(\fun_{(-)} : \sRM(\Sigma') \to (\DFib^{\to})_{\fun}\) extends to a
  representable map functor
  \(\widetilde{\fun}_{(-)} : \sRM(\Sigma) \to (\DFib^{\to})_{\fun}\) such that
  \(\widetilde{\fun}_{\alpha} \cong \funI\). From
  \(\dom \widetilde{\fun}_{\alpha} \cong \alpha^{\model}\) and
  \(\cod \widetilde{\fun}_{\alpha} \cong \alpha^{\modelI}\), we have
  \(\dom \widetilde{\fun}_{(\argu)} \cong (\argu)^{\model}\) and
  \(\cod \widetilde{\fun}_{(\argu)} \cong (\argu)^{\modelI}\) again by
  \cref{thm:syntactic-LF-category}. By
  \cref{prop:morphism-of-models,rem:morphism-of-model-up-to-iso},
  \(\widetilde{\fun}_{(-)}\) defines a morphism \(\model \to \modelI\)
  of models of \(\sRM(\Sigma)\) such that the restriction of
  \(\widetilde{\fun}_{(-)}\) to \(\sRM(\Sigma')\) is \(\fun\) and
  \(\widetilde{\fun}_{\alpha} \cong \funI\).

  The local fullness is similarly proved using the
  representable map category \((\DFib^{\Theta})_{\trans}\) instead of
  \((\DFib^{\to})_{\fun}\) and \cref{prop:2-morphisms-of-model}.
\end{proof}

\begin{example}
  \label{exm:models-of-dtt}
  Consider the signature
  \(\Sigma_{DTT} = (\s{Type} : () \To \largetype, \s{el} : (A :
  \s{Type}) \To \smalltype)\) from \cref{exm:sig-basic-dtt}. By
  \cref{thm:models-of-syntactic-lf-category}, the
  \(2\)-category \(\Mod_{\sRM(\Sigma_{DTT})}\) is bi-equivalent to the
  \(2\)-category defined as follows:
  \begin{itemize}
  \item the objects are the triples
    \((\model, \s{Type}^{\model}, \s{el}^{\model})\)
    consisting of a small category \(\model\) with a terminal
    object, a small discrete fibration \(\s{Type}^{\model}\) over
    \(\model\) and a representable map
    \(\s{el}^{\model} \to \s{Type}^{\model}\) of small discrete
    fibrations over \(\model\), that is, the categories with
    families or the natural models;
  \item the morphisms \(\model \to \modelI\) are the triples
    \((\fun, \fun_{\s{Type}}, \fun_{\s{el}})\) consisting of a functor
    \(\fun : \model \to \modelI\) preserving terminal objects and
    maps
    \(\fun_{\s{Type}} : \s{Type}^{\model} \to \s{Type}^{\modelI}\)
    and \(\fun_{\s{el}} : \s{el}^{\model} \to \s{el}^{\modelI}\) of
    discrete fibrations over \(\fun\) such that the diagram
    \[
      \begin{tikzcd}
        \s{el}^{\model}
        \arrow[r,"\fun_{\s{el}}"]
        \arrow[d] &
        \s{el}^{\modelI}
        \arrow[d] \\
        \s{Type}^{\model}
        \arrow[r,"\fun_{\s{Type}}"'] &
        \s{Type}^{\modelI}
      \end{tikzcd}
    \]
    commutes and satisfies the Beck-Chevalley condition, that is, the
    pseudo cwf-morphisms;
  \item the \(2\)-morphisms \(\fun \To \funI : \model \to \modelI\) are
    the natural transformations \(\trans : \fun \To \funI\) between the
    underlying functors such that there exist (necessarily unique)
    natural transformations
    \(\trans_{\s{Type}} : \fun_{\s{Type}} \To \funI_{\s{Type}}\) and
    \(\trans_{\s{el}} : \fun_{\s{el}} \To \funI_{\s{el}}\) over \(\trans\).
  \end{itemize}
  Our choice of \(2\)-morphisms of categories with families is quite
  natural, but there is another choice: the indexed natural
  transformations between the associated indexed categories
  \parencite{clairambault2014biequivalence}. A difference is that a
  collection of types is regarded as a set in our definition while it
  is regarded as a category in the other definition.
\end{example}

\begin{example}
  We describe core components of a model \(\model\) of cubical type
  theory (\cref{exm:cubical-type-theory}). In addition to the natural
  model structure \((\model, \s{Type}^{\model}, \s{el}^{\model})\), it
  is equipped with a discrete fibration \(\I^{\model}\) over
  \(\model\) such that the map \(\I^{\model} \to \model\) is
  representable and a representable map
  \(\s{true}^{\model} \to \s{Cof}^{\model}\) of discrete fibrations over
  \(\model\) that is also a monomorphism.
  \[
    \begin{tikzcd}
      \s{true}^{\model}
      \arrow[d, hook] & &
      \s{el}^{\model}
      \arrow[d] \\
      \s{Cof}^{\model}
      \arrow[dr]&
      \I^{\model}
      \arrow[d] &
      \s{Type}^{\model}
      \arrow[dl] \\
      & \model
    \end{tikzcd}
  \]
  By \cref{rem:repr-obj-and-repr-map}, \(\I^{\model}\) is regarded as
  an object of the category \(\model\) with which \(\model\) has
  products. The cubical set model of cubical type theory
  \parencite[Section 8]{cohen2016cubical} is turned into this
  structure. \(\model\) is the category of cubical
  sets. \(\s{Type}^{\model}\) is the discrete fibration of
  \emph{fibrant} families of cubical sets. \(\s{el}^{\model}\) is the
  discrete fibration of sections of fibrant families of cubical
  sets. \(\I^{\model}\) is the formal interval. In this example,
  \(\s{Cof}^{\model}\) is also representable and constructed from what
  is called the face lattice. There is a map
  \(\s{Cof}^{\model} \to \Omega\) to the subobject classifier
  \(\Omega\) in \(\model\), and the map
  \(\s{true}^{\model} \to \s{Cof}^{\model}\) is the monomorphism
  classified by it.
\end{example}

\section{Bi-initial Models}
\label{sec:bi-initial-models}

In this section and the next section, we develop the semantics of type
theory with the definitions of a type theory and a model of a type
theory introduced in \cref{sec:type-theories-their}. The first
step is to construct a \emph{bi-initial model} of a type theory.

\subsection{Democratic Models}
\label{sec:democratic-models}

Usually a bi-initial model of a type theory is a syntactic one and
has a special property: every object is represented by a finite
sequence of types. We introduce a class of models of a type theory
satisfying this property, generalizing the notion of a democratic
category with families
\parencite{clairambault2014biequivalence}.

\begin{definition}
  Let \(\model\) be a model of a type theory \(\tth\). We
  inductively define \emph{contextual objects} in \(\model\) as
  follows:
  \begin{enumerate}
  \item the terminal object \(1 \in \model\) is a contextual object;
  \item if \(\bobj \in \model\) is a contextual object, \(\arr : \obj \to \objI\)
    is a representable arrow in \(\tth\) and
    \(\elI \in \objI^{\model}(\bobj)\) is an element over \(\bobj\), then the
    context extension \(\{\elI\}^{\arr} \in \model\) is a contextual
    object.
  \end{enumerate}
  Note that the terminal object \(1\) and the context extension
  \(\{\elI\}^{\arr}\) are determined only up to unique isomorphism. We
  include all the terminal objects and all the context extensions in
  the class of contextual objects so that the class of contextual
  objects is closed under isomorphisms. We say \(\model\) is
  \emph{democratic} if every object of \(\model\) is contextual and
  denote by \(\Mod_{\tth}^{\dem}\) the full sub-\(2\)-category
  of \(\Mod_{\tth}\) consisting of the democratic models.
\end{definition}

\begin{proposition}
  \label{prop:morphism-preserves-contextual-objects}
  Any morphism \(\fun : \model \to \modelI\) of models of a type
  theory \(\tth\) carries contextual objects in \(\model\) to
  contextual objects in \(\modelI\).
\end{proposition}
\begin{proof}
  Immediate from the definition.
\end{proof}

Democratic models have the following interesting property.

\begin{proposition}
  \label{thm:democratic-codiscrete}
  Let \(\tth\) be a type theory, \(\model\) a democratic model
  of \(\tth\) and \(\modelI\) an arbitrary model of
  \(\tth\). Let \(\fun, \funI : \model \to \modelI\) be morphisms of
  models of \(\tth\).
  \begin{enumerate}
  \item There is at most one \(2\)-morphism \(\fun \To \funI\).
  \item Every \(2\)-morphism \(\fun \To \funI\) is invertible.
  \end{enumerate}
  Consequently, \(\Mod_{\tth}(\model, \modelI)\) is equivalent to a
  discrete category.
\end{proposition}
\begin{proof}
  Let \(\trans : \fun \To \funI\) be a \(2\)-morphism. Each component
  \(\trans_{\bobj} : \fun \bobj \to \funI \bobj\) is uniquely determined
  and invertible by induction on the contextual object \(\bobj\).
  \begin{itemize}
  \item \(\trans_{1} : \fun 1 \to \funI 1\) must be the unique arrow into
    the terminal object \(\funI 1\). Since \(\fun 1\) is also the terminal
    object, \(\trans_{1}\) is invertible.
  \item Let \(\bobj \in \model\) be an object,
    \(\arr : \obj \to \objI\) a representable arrow in \(\tth\) and
    \(\elI \in \objI^{\model}(\bobj)\) an element. Since
    \(\funI(\{\elI\}^{\arr}) \cong \{\funI_{\objI}\elI\}^{\arr}\), the
    component
    \(\trans_{\{\elI\}^{\arr}} : \fun(\{\elI\}^{\arr}) \to
    \funI(\{\elI\}^{\arr})\) is uniquely determined by
    \(\trans_{\bobj}\) and by the property that the diagram
    \[
      \begin{tikzcd}
        \fun(\{\elI\}^{\arr})
        \arrow[r,"\trans_{\{\elI\}^{\arr}}"]
        \arrow[d,"{\fun(\proj^{\arr}_{\elI})}"'] &
        \funI(\{\elI\}^{\arr})
        \arrow[d,"{\funI(\proj^{\arr}_{\elI})}"] \\
        \fun \bobj
        \arrow[r,"\trans_{\bobj}"'] &
        \funI \bobj
      \end{tikzcd}
    \]
    commutes and
    \(\funI_{\obj}(\radj^{\arr}_{\elI}) \act \trans_{\{\elI\}^{\arr}}
    = \fun_{\obj}(\radj^{\arr}_{\elI})\). If \(\trans_{\bobj}\) is
    invertible, then so is \(\trans_{\{\elI\}^{\arr}}\): the inverse
    \(\trans_{\{\elI\}^{\arr}}^{-1} : \funI(\{\elI\}^{\arr}) \to
    \fun(\{\elI\}^{\arr})\) is the unique arrow such that
    \(\fun(\proj^{\arr}_{\elI}) \circ \trans_{\{\elI\}^{\arr}}^{-1} =
    \trans_{\bobj}^{-1} \circ \funI(\proj^{\arr}_{\elI})\) and
    \(\fun_{\obj}(\radj^{\arr}_{\elI}) \act
    \trans_{\{\elI\}^{\arr}}^{-1} =
    \funI_{\obj}(\radj^{\arr}_{\elI})\).
  \end{itemize}
\end{proof}

An easy way to construct a democratic model is to throw away
non-contextual objects from an arbitrary model.

\begin{definition}
  Let \(\model\) be a model of a type theory \(\tth\). We
  define a model \(\model^{\heart}\) of \(\tth\) as
  follows.
  \begin{itemize}
  \item The base category \(\model^{\heart}\) is the full subcategory
    of \(\model\) consisting of the contextual objects.
  \item For an object \(\obj \in \tth\), we define
    \(\obj^{\model^{\heart}}\) to be the pullback
    \[
      \begin{tikzcd}
        \obj^{\model^{\heart}}
        \arrow[r,hook] \arrow[d] &
        \obj^{\model} \arrow[d] \\
        \model^{\heart}
        \arrow[r,hook] &
        \model.
      \end{tikzcd}
    \]
  \end{itemize}
  \(\model^{\heart}\) is indeed a model of \(\tth\) because
  \(\model^{\heart}\) is closed under context extensions. We call
  \(\model^{\heart}\) the \emph{heart of \(\model\)}.
\end{definition}

Let \(\model\) be a model of a type theory \(\tth\). By
definition \(\model^{\heart}\) is a democratic model of
\(\tth\) and the obvious inclusion
\(\model^{\heart} \to \model\) is a morphism of models of
\(\tth\). The heart \(\model^{\heart}\) is the largest
democratic model contained in \(\model\) in the following sense.

\begin{proposition}
  Let \(\model\) be a model of a type theory \(\tth\). The
  inclusion \(\model^{\heart} \to \model\) induces an
  isomorphism of categories
  \[
    \Mod_{\tth}(\modelI, \model^{\heart}) \cong
    \Mod_{\tth}(\modelI, \model)
  \]
  for any democratic model \(\modelI\) of \(\tth\).
\end{proposition}
\begin{proof}
  By \cref{prop:morphism-preserves-contextual-objects},
  every morphism \(\modelI \to \model\) from a democratic model
  \(\modelI\) factors through \(\model^{\heart}\).
\end{proof}

\subsection{The Bi-initial Model of a Type Theory}
\label{sec:bi-initial-model}

The bi-initial model of a type theory \(\tth\) is obtained from the
Yoneda embedding \(\tth \to \DFib_{\tth}\).

\begin{lemma}
  \label{lem:yoneda-LF-functor}
  Let \(\bcat\) be a small cartesian category.
  \begin{enumerate}
  \item The Yoneda embedding \(\bcat \to \DFib_{\bcat}\) preserves
    finite limits.
  \item For any arrow \(\barr : \bobj \to \bobjI\) in \(\bcat\), the map of
    discrete fibrations \(\barr : \bcat/\bobj \to \bcat/\bobjI\) is
    representable with right adjoint \(\barr^{*} : \bcat/\bobjI \to
    \bcat/\bobj\).
  \item The Yoneda embedding preserves existing pushforwards.
  \end{enumerate}
\end{lemma}
\begin{proof}
  The first two claims are obvious. To prove the third, let
  \(\barr : \bobj \to \bobjI\) and \(\barrI : \bobjI \to \bobjII\) be arrows in \(\bcat\) and
  suppose that the pushforward \(\barrI_{*}\bobj \in \bcat/\bobjII\) exists. By
  definition, for any arrow \(\barrII : \bobjIII \to \bobjII\) in \(\bcat\), we have a
  bijection \(\bcat/\bobjII(\bobjIII, \barrI_{*}\bobj) \cong \bcat/\bobjI(\barrI^{*}\bobjIII, \bobj)\). This
  means that we have a pullback
  \[
    \begin{tikzcd}
      \bcat/\barrI_{*}\bobj \arrow[r] \arrow[d] &
      \bcat/\bobj \arrow[d,"\barr"] \\
      \bcat/\bobjII \arrow[r,"\barrI^{*}"'] & \bcat/\bobjI
    \end{tikzcd}
  \]
  and thus \(\bcat/\barrI_{*}\bobj\) is the pushforward of
  \(\barr : \bcat/\bobj \to \bcat/\bobjI\) along
  \(\barrI : \bcat/\bobjI \to \bcat/\bobjII\) by
  \cref{prop:dfib-rep-pushforward}.
\end{proof}

\begin{definition}
  Let \(\tth\) be a type theory. The Yoneda embedding
  \(\tth/(\argu) : \tth \to \DFib_{\tth}\) is a
  representable map functor by \cref{lem:yoneda-LF-functor}, so
  we have a model \((\tth, \tth/(\argu))\) of
  \(\tth\). We denote by \(\iM(\tth)\) the heart of
  \((\tth, \tth/(\argu))\) and call it the \emph{bi-initial
    model of \(\tth\)} due to \cref{thm:initial-model}
  below.
\end{definition}

We describe the model \(\iM(\tth)\) in more detail. For an
object \(\bobj \in \tth\), a representable arrow \(\arr : \obj \to \objI\) in
\(\tth\) and an object \((\elI : \bobj \to \objI) \in \tth/\objI\), the
context extension \(\{\elI\}^{\arr}\) in the model
\((\tth, \tth/(\argu))\) is the pullback in \(\tth\)
\[
  \begin{tikzcd}
    \{\elI\}^{\arr}
    \arrow[r,"\radj^{\arr}_{\elI}"]
    \arrow[d,"\proj^{\arr}_{\elI}"'] &
    \obj \arrow[d,"\arr"] \\
    \bobj \arrow[r,"\elI"'] &
    \objI.
  \end{tikzcd}
\]
Thus the base category \(\iM(\tth)\) is the full subcategory of
\(\tth\) consisting of those objects \(\obj \in \tth\) such that the
unique arrow \(\obj \to 1\) is representable. For an object
\(\obj \in \tth\), the discrete fibration \(\obj^{\iM(\tth)}\) is the
comma category \(\iM(\tth)/\obj\) defined by the pullback
\[
  \begin{tikzcd}
    \iM(\tth)/\obj \arrow[r,hook] \arrow[d] &
    \tth/\obj \arrow[d,"\dom"] \\
    \iM(\tth) \arrow[r,hook] &
    \tth.
  \end{tikzcd}
\]

\begin{example}
  \label{exm:bi-initial-model-ctx}
  Suppose that \(\tth\) contains a representable arrow
  \(\typeof : \El \to \Ty\) and that every representable arrow in
  \(\tth\) is a composite of pullbacks of \(\typeof\). For example,
  the basic dependent type theory \(\thgat\) and its slices
  \(\thgat/\obj\) satisfy this assumption. Then the base category of
  \(\iM(\tth)\) is equivalent to the following:
  \begin{itemize}
  \item the objects are the finite sequences \((\obj_{1}, \dots,
    \obj_{n})\) of arrows \(\obj_{\idx} : |\obj_{\idx-1}| \to \Ty\)
    where \(|\obj_{0}| = 1\) and \(|\obj_{\idx}| =
    \obj_{\idx}^{*}\El\) for \(\idx \ge 1\);
  \item the arrows \((\obj_{1}, \dots, \obj_{n}) \to (\objI_{1},
    \dots, \objI_{m})\) are the arrows \(|\obj_{n}| \to |\objI_{m}|\)
    in \(\tth\).
  \end{itemize}
  When we think of \(\Ty\) as the object of types and \(\El\) as the
  object of elements, an object in \(\iM(\tth)\) is a finite sequence
  of types, that is, a context. For an object \(\bobj \in \iM(\tth)\),
  elements of \(\Ty^{\iM(\tth)}(\bobj)\) and
  \(\El^{\iM(\tth)}(\bobj)\) are types and elements, respectively,
  indexed over the context \(\bobj\). Hence, the bi-initial model
  \(\iM(\tth)\) generalizes the usual syntactic models of dependent
  type theories.
\end{example}

\begin{example}
  Let \(\Sigma\) be a signature of the logical framework
  (\cref{sec:logical-framework}) and take the syntactic representable
  map category \(\sRM(\Sigma)\) (\cref{sec:synt-repr-map}). The base
  category of \(\iM(\sRM(\Sigma))\) is the full subcategory of
  \(\sRM(\Sigma)\) spanned by those contexts
  \((x_{1} : A_{1}, \dots, x_{n} : A_{n})\) such that
  \(A_{1}, \dots, A_{n}\) are all representable types. When \(\Sigma\) is
  \(\Sigma_{DTT}\) (\cref{exm:sig-basic-dtt}) or its extension by type
  constructors, all of \(A_{i}\)'s are of the form \(\s{el}(B)\) for
  \(B : \s{Type}\). When \(\Sigma\) is the signature for cubical type
  theory (\cref{exm:cubical-type-theory}), \(A_{i}\)'s are either
  \(\s{el}(B)\) for \(B : \s{Type}\), \(\I\), or \(\s{true}(P)\) for
  \(P : \s{Cof}\).
\end{example}

\begin{theorem}
  \label{thm:initial-model}
  For any type theory \(\tth\), the model \(\iM(\tth)\) is a
  bi-initial object of \(\Mod_{\tth}\). That is, the category
  \(\Mod_{\tth}(\iM(\tth), \model)\) is contractible for any model
  \(\model\) of \(\tth\).
\end{theorem}
\begin{proof}
  We show that there exists a morphism
  \(\iM(\tth) \to \model\) and that morphisms
  \(\iM(\tth) \to \model\) are unique up to unique
  isomorphism.

  We first show the existence of a morphism \(\iM(\tth) \to
  \model\). For every object \(\bobj \in \iM(\tth)\), the unique map
  \(\bobj^{\model} \to \model\) of discrete fibrations over \(\model\)
  is representable. In particular, the discrete fibration
  \(\bobj^{\model} \in \DFib_{\model}\) is representable because the
  category \(\model\) has a terminal object. Hence, the restriction of
  \((\argu)^{\model} : \tth \to \DFib_{\model}\) to \(\iM(\tth)\)
  factors, up to natural isomorphism, as a functor
  \(\fun : \iM(\tth) \to \model\) followed by the Yoneda embedding
  \(\model \to \DFib_{\model}\). For objects \(\obj \in \tth\) and
  \(\bobj \in \iM(\tth)\) and an arrow \(\el : \bobj \to \obj\), we
  define \(\fun_{\obj}(\el) \in \obj^{\model}(\fun\bobj)\) to be
  \(\model/\fun\bobj \cong \bobj^{\model}
  \overset{\el^{\model}}{\longrightarrow} \obj^{\model}\), yielding a
  map \(\fun_{A} : \iM(\tth)/\obj \to \obj^{\model}\) of discrete
  fibrations over \(\fun : \iM(\tth) \to \model\).

  Clearly \(\fun : \iM(\tth) \to \model\) preserves terminal
  objects and \(\obj \mapsto \fun_{\obj}\) is natural. To see the Beck-Chevalley
  condition, let \(\arr : \obj \to \objI\) be a representable arrow in
  \(\tth\). We have to show that the diagram
  \[
    \begin{tikzcd}
      \iM(\tth)/\obj \arrow[r,"\fun_{\obj}"]
      \arrow[d,"\arr"'] &
      \obj^{\model}
      \arrow[d,"\arr^{\model}"] \\
      \iM(\tth)/\objI \arrow[r,"\fun_{\objI}"'] &
      \objI^{\model}
    \end{tikzcd}
  \]
  satisfies the Beck-Chevalley condition. It suffices to show that the
  composite of squares
  \begin{equation}
    \label[diagram]{eq:6}
    \begin{tikzcd}
      \iM(\tth)/\elI^{*}\obj
      \arrow[r,"\arr^{*}\elI"]
      \arrow[d,"\elI^{*}\arr"'] &
      \iM(\tth)/\obj
      \arrow[r,"\fun_{\obj}"]
      \arrow[d,"\arr"] &
      \obj^{\model}
      \arrow[d,"\arr^{\model}"] \\
      \iM(\tth)/\bobj
      \arrow[r,"\elI"'] &
      \iM(\tth)/\objI
      \arrow[r,"\fun_{\objI}"'] &
      \objI^{\model}
    \end{tikzcd}
  \end{equation}
  satisfies the Beck-Chevalley condition for all objects \((\elI :
  \bobj \to \objI) \in \iM(\tth)/\objI\). By the definition of
  \(\fun_{(\argu)}\), \cref{eq:6} is equal to the following composite
  of squares.
  \begin{equation}
    \label[diagram]{eq:7}
    \begin{tikzcd}
      \iM(\tth)/\elI^{*}\obj \arrow[r,"\fun"]
      \arrow[d,"\elI^{*}\arr"'] & \model/\fun(\elI^{*}\obj)
      \arrow[r,"\cong"] \arrow[d,"\fun(\elI^{*}\arr)"'] &
      [-2ex]
      (\elI^{*}\obj)^{\model} \arrow[r,"(\arr^{*}\elI)^{\model}"]
      \arrow[d,"(\elI^{*}\arr)^{\model}"] & \obj^{\model}
      \arrow[d,"\arr^{\model}"] \\
      \iM(\tth)/\bobj \arrow[r,"\fun"'] & \model/\fun\bobj
      \arrow[r,"\cong"'] & \bobj^{\model} \arrow[r,"\elI^{\model}"'] &
      \objI^{\model}
    \end{tikzcd}
  \end{equation}
  To show that \cref{eq:7} satisfies the Beck-Chevalley condition for
  all \((\elI : \bobj \to \objI) \in \iM(\tth)/\objI\), it suffices to
  check that the canonical natural transformation
  \((\arr^{*}\elI)^{\model}\fun(\elI^{*}\arr)^{*} \To
  \radj^{\arr}\elI^{\model}\fun\) induced by \cref{eq:7} is invertible
  at \(\id_{\bobj} \in \iM(\tth)/\bobj\) for all
  \((\elI : \bobj \to \objI) \in \iM(\tth)/\objI\). This is
  straightforward because the right-most square of \cref{eq:7} is a
  pullback in \(\DFib_{\model}\) and thus satisfies the Beck-Chevalley
  condition by \cref{cor:representable-map-pullback}.

  To show the uniqueness of morphisms \(\iM(\tth) \to \model\), let
  \(\funI : \iM(\tth) \to \model\) be another morphism of models of
  \(\tth\). By \cref{thm:democratic-codiscrete}, it suffices to show
  that there exists a \(2\)-morphism \(\fun \To \funI\). Let
  \(\bobj \in \iM(\tth)\) be an object and let \(\barr : \bobj \to 1\)
  denote the unique arrow to the terminal object. By the
  Beck-Chevalley condition for the square
  \[
    \begin{tikzcd}
      \iM(\tth)/\bobj
      \arrow[r,"\funI_{\bobj}"]
      \arrow[d, "\barr_{!}"'] &
      \bobj^{\model}
      \arrow[d, "\barr^{\model}"] \\
      \iM(\tth)
      \arrow[r,"\funI"'] &
      \model,
    \end{tikzcd}
  \]
  we have the natural isomorphism
  \(\funI_{\bobj} \barr^{*} \cong \radj^{\barr} \funI : \iM(\tth) \to
  \bobj^{\model}\). Since \(\radj^{\barr} \funI\) preserves terminal
  objects,
  \(\funI_{\bobj}(\id_{\bobj}) \cong \funI_{\bobj}(\barr^{*} 1)\) is the
  terminal object. Thus, by \cref{prop:representability}, we have
  \(\model/\funI\bobj \cong \bobj^{\model}\). For an arrow
  \(\el : \bobj \to \obj\) in \(\tth\) with \(\bobj \in \iM(\tth)\), the
  diagram
  \[
    \begin{tikzcd}
      \iM(\tth)/\bobj
      \arrow[r,"\funI"]
      \arrow[dr,"\funI_{\bobj}"']
      \arrow[dd,"\el"'] &
      \model/\funI\bobj
      \arrow[d,"\cong"] \\
      & \bobj^{\model}
      \arrow[d,"\el^{\model}"] \\
      \iM(\tth)/\obj
      \arrow[r,"\funI_{\obj}"'] &
      \obj^{\model}
    \end{tikzcd}
  \]
  commutes. This means that
  \(\funI_{\obj}(\el) \in \obj^{\model}(\funI\bobj)\) is given by the
  composite
  \(\model/\funI\bobj \cong \bobj^{\model}
  \overset{\el^{\model}}{\longrightarrow} \obj^{\model}\). Hence,
  \(\funI\) has the same definition as \(\fun\), and thus
  \(\fun \cong \funI\).
\end{proof}

\section{Internal Languages}
\label{sec:internal-languages}

In this section we establish a correspondence between theories and
models for every type theory \(\tth\). We begin with a definition of a
\emph{theory over \(\tth\)} or \emph{\(\tth\)-theory}.

\begin{definition}
  Let \(\tth\) be a type theory. A \emph{theory over
    \(\tth\)} or \emph{\(\tth\)-theory} is a cartesian
  functor \(\theory : \tth \to \Set\). We denote by
  \(\Theory_{\tth}\) the category of \(\tth\)-theories and
  their maps, that is, natural transformations.
\end{definition}

For a \(\tth\)-theory \(\theory\), we think of \(\theory(\obj)\) for
an object \(\obj \in \tth\) as a set of closed derivations of judgment
form \(\obj\). While \(\theory\) respects finite limits, representable
maps and pushforward along them are disregarded. Indeed, for a
representable map \(\obj \to 1\) and an arbitrary object \(\objI\) in
\(\tth\), not all set-theoretic functions
\(\theory(\obj) \to \theory(\objI)\) should be closed derivations of the
exponential \(\obj \To \objI\).

\begin{example}
  \label{exm:G-theories}
  Consider the basic dependent type theory \(\thgat\)
  (\cref{sec:exampl-basic-depend}). The theory of locally presentable
  categories \parencite{adamek1994locally} shows that
  \(\Theory_{\thgat}\) is equivalent to the category of generalized
  algebraic theories and interpretations between them; see also
  \parencite[Remark 3.26]{uemura2019exponentiability}. Concretely, for
  a \(\thgat\)-theory \(\theory : \thgat \to \Set\), the corresponding
  generalized algebraic theory \(\Sigma_{\theory}\) can be described as
  follows:
  \begin{itemize}
  \item a closed type \(({} \vdash A)\) in \(\Sigma_{\theory}\)
    corresponds to an element of \(\theory(\Ty_{0})\). Thus,
    \(\theory(\Ty_{0})\) is the set of closed types;
  \item a closed term \(({} \vdash a : A)\) in \(\Sigma_{\theory}\)
    corresponds to an element of \(\theory(\El_{0})\) such that
    \(\typeof_{0} \act a = A\). Thus, \(\theory(\El_{0})\) is the set
    of closed terms;
  \item a type
    \((x_{0} : A_{0}, \dots, x_{n-1} : A_{n-1} \vdash A_{n})\) in
    \(\Sigma_{\theory}\) corresponds to an element of
    \(\theory(\Ty_{n})\) such that \(\ctxof_{i} \act A_{i} = A_{i-1}\)
    for \(i = n, \dots, 1\). Thus, \(\theory(\Ty_{n})\) is the set of
    types over a context of length \(n\);
  \item a term
    \((x_{0} : A_{0}, \dots, x_{n-1} : A_{n-1} \vdash a_{n} : A_{n})\)
    in \(\Sigma_{\theory}\) corresponds to an element of
    \(\theory(\El_{n})\) such that \(\typeof_{n} \act a_{n} = A_{n}\)
    and \(\ctxof_{i} \act A_{i} = A_{i-1}\) for \(i = n, \dots,
    1\). Thus, \(\theory(\El_{n})\) is the set of terms over a context
    of length \(n\).
  \end{itemize}
  From \cref{thm:gat-exp} one can see that every object
  \(\obj \in \thgat\) is a finite limit of \(\Ty_{n}\) and
  \(\El_{n}\), and thus the other sets \(\theory(\obj)\) are finite
  limits of \(\theory(\Ty_{n})\) and \(\theory(\El_{n})\).
\end{example}

\begin{example}
  Let \(\Sigma\) be a signature of the logical framework
  (\cref{sec:logical-framework}) and take the syntactic representable
  map category \(\sRM(\Sigma)\) (\cref{sec:synt-repr-map}). Every context
  \(\Sigma \mid \Gamma \vdash \ctx\) induces the
  \(\sRM(\Sigma)\)-theory
  \(\sRM(\Sigma)(\Gamma, \argu) : \sRM(\Sigma) \to \Set\). When
  \(\Gamma = (x_{1} : A_{1}, \dots, x_{n} : A_{n})\), we think of
  \(\sRM(\Sigma)(\Gamma, \argu)\) as the theory consisting of constants
  \(c_{1} : A_{1}, c_{2} : A_{2}[c_{1}/x_{1}], \dots, c_{n} :
  A_{n}[c_{1}/x_{1}, \dots, c_{n-1}/x_{n-1}]\). Indeed,
  \(\sRM(\Sigma)(\Gamma, \Delta)\) is the set of derivations from
  \(\Gamma\) to \(\Delta\), but they are equivalent to closed derivations of
  \(\Delta\) where \(x_{1}, \dots, x_{n}\) are regarded as constants. A
  \(\sRM(\Sigma)\)-theory of the form \(\sRM(\Sigma)(\Gamma, \argu)\) is thus a
  finite \(\sRM(\Sigma)\)-theory. The theory of locally presentable
  categories \parencite{adamek1994locally} shows that any
  \(\sRM(\Sigma)\)-theory is written as a filtered colimit of finite
  \(\sRM(\Sigma)\)-theories.
\end{example}

The internal language of a model of \(\tth\) is then easily defined.

\begin{definition}
  Let \(\tth\) be a type theory and \(\model\) a model of \(\tth\). We
  define a \(\tth\)-theory \(\iL_{\tth}\model\) to be
  \(\DFib_{\model}(\model, (\argu)^{\model})\), where we regard
  \(\model\) as a discrete fibration over \(\model\) with the identity
  functor \(\model \to \model\). Note that
  \(\iL_{\tth}\model(\obj) \cong \obj^{\model}(1)\) because \(\model\) has
  the terminal object \(1\). We call \(\iL_{\tth}\model\) the
  \emph{internal language of \(\model\)}.
\end{definition}

\begin{example}
  Let \(\model\) be a model of the type theory \(\thgat\), that is, a
  category with families. We think of \(\Ty_{0}^{\model}\) as a
  discrete fibration of types and \(\El_{0}^{\model}\) as a discrete
  fibration of terms. Consider the set
  \(\iL_{\thgat}\model(\Ty_{n}) \cong \Ty_{n}^{\model}(1)\). By
  \cref{item:20} of \cref{thm:gat-exp}, it is isomorphic to
  \((\poly_{\typeof_{0}}^{n}\Ty_{0})^{\model}(1)\). By
  \cref{prop:dfib-rep-pushforward} an element of
  \((\poly_{\typeof_{0}}^{n}\Ty_{0})^{\model}(1)\) is a sequence
  \((\el_{0}, \dots, \el_{n})\) of elements
  \(\el_{0} \in \Ty_{0}^{\model}(1), \el_{1} \in
  \Ty_{0}^{\model}(\{\el_{0}\}^{\typeof_{0}}), \dots, \el_{n} \in
  \Ty_{0}^{\model}(\{\el_{n-1}\}^{\typeof_{0}})\). In other words,
  \(\iL_{\thgat}\model(\Ty_{n})\) is the set of types in \(\model\)
  over a context of length \(n\). Similarly,
  \(\iL_{\thgat}\model(\El_{n})\) is the set of terms in \(\model\)
  over a context of length \(n\).
\end{example}

\begin{proposition}
  For a type theory \(\tth\), the map
  \(\model \mapsto \iL_{\tth}\model\) is part of a
  \(2\)-functor
  \(\iL_{\tth} : \Mod_{\tth} \to \Theory_{\tth}\),
  where we regard \(\Theory_{\tth}\) as a locally discrete
  \(2\)-category.
\end{proposition}

\begin{proof}
  For a morphism \(\fun : \model \to \modelI\) of models of
  \(\tth\), an object \(\obj \in \tth\) and a map
  \(\cst : \model \to \obj^{\model}\), we regard \(\cst\) as an element
  \(\cst \in \obj^{\model}(1)\) and define
  \(\iL_{\tth}\fun(\cst) : \modelI \to \obj^{\modelI}\) to be the map
  corresponding to the element
  \(\fun_{\obj} \cst \in \obj^{\modelI}(\fun 1) \cong \obj^{\modelI}(1)\). In other
  words, \(\iL_{\tth}\fun(\cst) : \modelI \to \obj^{\modelI}\) is the
  unique map such that the diagram
  \[
    \begin{tikzcd}
      \model \arrow[r,"\fun"] \arrow[d,"\cst"'] & \modelI
      \arrow[d,"\iL_{\tth}\fun(\cst)"] \\
      \obj^{\model} \arrow[r,"\fun_{\obj}"'] & \obj^{\modelI}
    \end{tikzcd}
  \]
  commutes. For a \(2\)-morphism
  \(\trans : \fun \To \funI : \model \to \modelI\) of models of
  \(\tth\), we have \(\iL_{\tth}\fun = \iL_{\tth}\funI\)
  because \(\fun_{\obj} \cst = \funI_{\obj} \cst \act \trans_{1}\) for any element
  \(\cst \in \obj^{\model}(1)\).
\end{proof}

The goal of this section is to show that the internal language
\(2\)-functor \(\iL_{\tth} : \Mod_{\tth} \to \Theory_{\tth}\) has a
left bi-adjoint \(\sM_{\tth} : \Theory_{\tth} \to \Mod_{\tth}\)
(\cref{thm:internal-language-adjunction}) and induces a bi-equivalence
\(\Mod_{\tth}^{\dem} \simeq \Theory_{\tth}\)
(\cref{thm:bi-equivalence-th-mod}).

The idea of the construction of the left bi-adjoint \(\sM_{\tth}\) of
\(\iL_{\tth}\) is as follows. We consider the case of
\(\tth = \thgat\). From \cref{exm:G-theories} a \(\tth\)-theory
\(\theory\) consists of sets \(\theory(\Ty_{n})\) of types, sets
\(\theory(\El_{n})\) of terms, and so on. We adjoin to \(\tth\) types
of \(\theory\) as type constructors and terms of \(\theory\) as term
constructors, yielding a new type theory \({\tth}[\theory]\) together
with an inclusion \(\tth \to {\tth}[\theory]\). We take the bi-initial
model \(\iM({\tth}[\theory])\) and then obtain a model
\(\sM_{\tth}\theory\) of \(\tth\) by restricting
\((\argu)^{\iM({\tth}[\theory])} : {\tth}[\theory] \to
\DFib_{\iM({\tth}[\theory])}\) along \(\tth \to {\tth}[\theory]\).

In \cref{sec:filt-pseudo-colim} we review filtered pseudo-colimits of
categories and show that representable map categories are closed under
filtered pseudo-colimits. The type theory \({\tth}[\theory]\) is
defined as a filtered pseudo-colimit in
\cref{sec:theories-over-type}. In \cref{sec:intern-lang-type} we show
that \(\sM_{\tth}\) is left bi-adjoint to \(\iL_{\tth}\). In
\cref{sec:bi-equiv-theor} we show that \(\iL_{\tth}\) induces a
bi-equivalence \(\Mod_{\tth}^{\dem} \simeq \Theory_{\tth}\).

\subsection{Filtered Pseudo-colimits of Representable Map Categories}
\label{sec:filt-pseudo-colim}

In this preliminary subsection we show that the \(2\)-category of
representable map categories has filtered pseudo-colimits.

\begin{definition}
  Let \(\cat : \idxcat \to \Cat\) be a pseudo-functor from a small
  category \(\idxcat\). We define a small category
  \(\plim_{\idx \in \idxcat}\cat_{\idx}\) called the \emph{pseudo-limit of
    \(\cat\)} as follows.
  \begin{itemize}
  \item An object of \(\plim_{\idx \in \idxcat}\cat_{\idx}\) consists of the
    following data:
    \begin{itemize}
    \item for each object \(\idx \in \idxcat\), an object
      \(\obj_{\idx} \in \cat_{\idx}\);
    \item for each arrow \(\iarr : \idx \to \idx'\) in \(\idxcat\), an isomorphism
      \(\obj_{\iarr} : \iarr \act \obj_{\idx} \cong \obj_{\idx'}\)
    \end{itemize}
    satisfying the following conditions:
    \begin{itemize}
    \item for any object \(\idx \in \idxcat\), the diagram
      \[
        \begin{tikzcd}
          \obj_{\idx} \arrow[r,"\cong"] \arrow[dr,equal] &
          \id_{\idx} \act \obj_{\idx} \arrow[d,"\obj_{\id_{\idx}}"] \\
          & \obj_{\idx}
        \end{tikzcd}
      \]
      commutes;
    \item for any arrows \(\iarr : \idx \to \idx'\) and \(\iarrI : \idx' \to \idx''\) in
      \(\idxcat\), the diagram
      \[
        \begin{tikzcd}
          \iarrI \act (\iarr \act \obj_{\idx}) \arrow[r,"\cong"] \arrow[d,"\iarrI \act
          \obj_{\iarr}"'] &
          (\iarrI\iarr) \act \obj_{\idx} \arrow[d,"\obj_{\iarrI\iarr}"] \\
          \iarrI \act \obj_{\idx'} \arrow[r,"\obj_{\iarrI}"'] & \obj_{\idx''}
        \end{tikzcd}
      \]
      commutes.
    \end{itemize}
  \item An arrow \(\obj \to \objI\) in \(\plim_{\idx \in \idxcat}\cat_{\idx}\)
    consists of an arrow \(\arr_{\idx} : \obj_{\idx} \to \objI_{\idx}\) for each object
    \(\idx \in \idxcat\) such that, for any arrow \(\iarr : \idx \to \idx'\) in \(\idxcat\),
    the diagram
    \[
      \begin{tikzcd}
        \iarr \act \obj_{\idx}
        \arrow[r,"\iarr \act \arr_{\idx}"]
        \arrow[d,"\obj_{\iarr}"'] &
        \iarr \act \objI_{\idx}
        \arrow[d,"\objI_{\iarr}"] \\
        \obj_{\idx'}
        \arrow[r,"\arr_{\idx'}"'] &
        \objI_{\idx'}
      \end{tikzcd}
    \]
    commutes.
  \end{itemize}
\end{definition}

\begin{definition}
  Let \(\idxcat\) be a category. We say \(\idxcat\) is \emph{filtered} if every
  finite diagram in \(\idxcat\) has a cocone. \(\idxcat\) is \emph{cofiltered} if
  \(\idxcat^{\op}\) is filtered.
\end{definition}

\begin{definition}
  Let \(\cat : \idxcat \to \Cat\) be a pseudo-functor from a filtered
  small category \(\idxcat\). We define a small category
  \(\pcolim_{\idx \in \idxcat}\cat_{\idx}\) as follows.
  \begin{itemize}
  \item The objects of \(\pcolim_{\idx \in \idxcat}\cat_{\idx}\) are the pairs
    \((\idx, \obj)\) of objects \(\idx \in \idxcat\) and \(\obj \in \cat_{\idx}\).
  \item For objects
    \((\idx_{1}, \obj_{1}), (\idx_{2}, \obj_{2}) \in \pcolim_{\idx \in
      \idxcat}\cat_{\idx}\) we define
    \[
      \pcolim_{\idx \in \idxcat}\cat_{\idx}((\idx_{1}, \obj_{1}), (\idx_{2}, \obj_{2})) =
      \colim_{\substack{\idx \in \idxcat \\ \iarr_{1} : \idx_{1} \to \idx \\ \iarr_{2} :
          \idx_{2} \to \idx}}\cat_{\idx}(\iarr_{1} \act \obj_{1}, \iarr_{2} \act
      \obj_{2}).
    \]
  \end{itemize}
  There are the obvious functors
  \(\inj_{\idx} : \cat_{\idx} \to \pcolim_{\idx \in \idxcat}\cat_{\idx}\) for
  objects \(\idx \in \idxcat\) and natural isomorphisms
  \(\inj_{\iarr} : \inj_{\idx} \cong \inj_{\idx'} \circ \cat_{\iarr}\) for
  arrows \(\iarr : \idx \to \idx'\) in \(\idxcat\), yielding an object
  \(\inj \in \plim_{\idx \in \idxcat}\Cat(\cat_{\idx}, \pcolim_{\idx \in
    \idxcat}\cat_{\idx})\). For a category \(\catI\), the canonical
  functor
  \[
    \inj^{*} : \Cat(\pcolim_{\idx \in \idxcat}\cat_{\idx}, \catI) \to
    \plim_{\idx \in \idxcat}\Cat(\cat_{\idx}, \catI)
  \]
  is an isomorphism of categories. We call
  \(\pcolim_{\idx \in \idxcat}\cat_{\idx}\) the \emph{filtered pseudo-colimit
    of \(\cat\)}.
\end{definition}

An important property of filtered pseudo-colimits is that filtered
pseudo-colimits in \(\Cat\) commute with finite bi-limits
\parencite[Theorem 3.2]{descotte2010construction}. We only use the
following special cases.

\begin{lemma}
  \label{lem:filtered-pseudo-colimit-vs-finit-limit}
  Filtered pseudo-colimits commute with finite cotensors. More
  precisely, for a pseudo-functor \(\cat : \idxcat \to \Cat\) from a
  filtered small category \(\idxcat\) and a finite category
  \(\idxcatI\), the canonical functor
  \[
    \pcolim_{\idx \in \idxcat}\cat_{\idx}^{\idxcatI} \to \left(\pcolim_{\idx \in
        \idxcat}\cat_{\idx}\right)^{\idxcatI}
  \]
  is an equivalence of categories. \qed
\end{lemma}

\begin{lemma}
  \label{lem:filtered-pseudo-colimit-slice}
  Filtered pseudo-colimits commute with slicing. More precisely, for a
  pseudo-functor \(\cat : \idxcat \to \Cat\) from a filtered small
  category \(\idxcat\) and objects \(\idx_{0} \in \idxcat\) and
  \(\obj \in \cat_{\idx_{0}}\), the canonical functor
  \[
    \pcolim_{(\iarr : \idx_{0} \to \idx) \in \idx_{0} \under \idxcat}\left(\cat_{\idx}/\iarr
      \act \obj\right) \to \left(\pcolim_{\idx \in
        \idxcat}\cat_{\idx}\right)/(\idx_{0}, \obj)
  \]
  is an equivalence of categories.
\end{lemma}
\begin{proof}
  By the bicategorical Yoneda lemma, the pair \((\idx_{0}, \obj)\)
  corresponds to a pseudo-natural transformation
  \(\obj : \idxcat(\idx_{0}, \argu) \to \cat\). Let
  \((\cat \downarrow \obj)\) be the oplax bi-limit of the \(1\)-cell
  \(\obj : \idxcat(\idx_{0}, \argu) \to \cat\) in the \(2\)-category
  of pseudo-functors \(\idxcat \to \Cat\), pseudo-natural
  transformations and modifications. \((\cat \downarrow \obj)\) is
  calculated pointwise, so \((\cat \downarrow \obj)_{\idx}\) is the
  category of pairs \((\iarr, \objI)\) consisting of an arrow
  \(\iarr : \idx_{0} \to \idx\) in \(\idxcat\) and an object
  \(\objI \in \cat/\iarr \act \obj\). One can check that the filtered
  pseudo-colimit of \((\cat \downarrow \obj) : \idxcat \to \Cat\) is
  equivalent to
  \(\pcolim_{(\iarr : \idx_{0} \to \idx) \in \idx_{0} \under
    \idxcat}(\cat_{\idx}/\iarr \act \obj)\) and that the filtered
  pseudo-colimit of \(\idxcat(\idx_{0}, \argu) : \idxcat \to \Cat\) is
  equivalent to the terminal category. Thus, it follows from the
  commutation of filtered pseudo-colimits and oplax bi-limits of a
  \(1\)-cell that
  \(\pcolim_{(\iarr : \idx_{0} \to \idx) \in \idx_{0} \under
    \idxcat}(\cat_{\idx}/\iarr \act \obj)\) is canonically equivalent
  to the oplax limit of the \(1\)-cell
  \((\idx_{0}, \obj) : 1 \to \pcolim_{\idx \in \idxcat}\cat_{\idx}\)
  in \(\Cat\), that is, the slice category
  \(\left(\pcolim_{\idx \in \idxcat}\cat_{\idx}\right)/(\idx_{0},
  \obj)\).
\end{proof}

\begin{proposition}
  \label{prop:finite-limits-in-pseudo-colimit}
  Let \(\cat : \idxcat \to \Cat\) be a pseudo-functor from a filtered
  small category \(\idxcat\).
  \begin{enumerate}
  \item \label{item:3}
    If all \(\cat_{\idx}\) are cartesian categories and all
    \(\cat_{\iarr} : \cat_{\idx} \to \cat_{\idx'}\) are cartesian
    functors, then \(\pcolim_{\idx \in \idxcat}\cat_{\idx}\) is a cartesian
    category and the functors \(\inj_{\idx} : \cat_{\idx} \to \pcolim_{\idx
      \in \idxcat}\cat_{\idx}\) are cartesian functors.
  \item \label{item:4}
    Suppose the hypotheses of \labelcref{item:3} hold. Let \(\idx_{0}\) be
    an object of \(\idxcat\) and \(\arr : \obj \to \objI\) an arrow in
    \(\cat_{\idx_{0}}\). Suppose that, for any arrow
    \(\iarr : \idx_{0} \to \idx\), the pushforwards along \(\iarr \act \arr\) exist
    and that, for any arrows \(\iarr : \idx_{0} \to \idx\) and
    \(\iarrI : \idx \to \idx'\), the functor
    \(\cat_{\iarrI} : \cat_{\idx} \to \cat_{\idx'}\) carries
    pushforwards along \(\iarr \cdot \arr\) to those along
    \(\iarrI\iarr \act \arr\). Then pushforwards along \(\inj_{\idx_{0}}(\arr)\) exist
    and the functor \(\inj_{\idx_{0}} : \cat_{\idx_{0}} \to \pcolim_{\idx
      \in \idxcat}\cat_{\idx}\) carries pushforwards along \(\arr\) to
    those along \(\inj_{\idx_{0}}(\arr)\).
  \end{enumerate}
\end{proposition}
\begin{proof}
  Since limits in a category \(\cat\) are given by the right adjoint
  to the diagonal functor \(\cat \to \cat^{\idxcatI}\), the
  statement~\labelcref{item:3} is an immediate consequence of
  \cref{lem:filtered-pseudo-colimit-vs-finit-limit}. For
  \labelcref{item:4} consider the diagram
  \[
    \begin{tikzcd}
      \pcolim_{(\iarr : \idx_{0} \to \idx) \in \idx_{0} \under
        \idxcat}\left(\cat_{\idx}/\iarr \act \objI\right)
      \arrow[r]
      \arrow[d,"\pcolim_{\iarr}(\iarr \act \arr)^{*}"'] &
      \left(\pcolim_{\idx \in \idxcat}\cat_{\idx}\right)/\objI
      \arrow[d,"\inj_{\idx_{0}}(\arr)^{*}"] \\
      \pcolim_{(\iarr : \idx_{0} \to \idx) \in \idx_{0} \under
        \idxcat}\left(\cat_{\idx}/\iarr \act \obj\right)
      \arrow[r] &
      \left(\pcolim_{\idx \in \idxcat}\cat_{\idx}\right)/\obj.
    \end{tikzcd}
  \]
  This diagram commutes up to canonical isomorphism by
  \labelcref{item:3}. The horizontal functors are equivalences by
  \cref{lem:filtered-pseudo-colimit-slice}. Thus, \(\inj_{\idx_{0}}(\arr)^{*}\) has a
  right adjoint because \(\pcolim_{\iarr}(\iarr \act \arr)^{*}\) has a right
  adjoint by assumption.
\end{proof}

Let \(\cat : \idxcat \to \RMCat\) be a pseudo-functor from a filtered
small category \(\idxcat\). We define an arrow in
\(\pcolim_{\idx \in \idxcat}\cat_{\idx}\) to be representable if it is
isomorphic to the image of a representable arrow in \(\cat_{\idx}\) by
\(\inj_{\idx}\) for some \(\idx \in \idxcat\). Then, by
\cref{prop:finite-limits-in-pseudo-colimit},
\(\pcolim_{\idx \in \idxcat}\cat_{\idx}\) forms a representable map category
and the functors
\(\inj_{\idx} : \cat_{\idx} \to \pcolim_{\idx \in \idxcat}\cat_{\idx}\) are
representable map functors. The following is immediate from the
construction.

\begin{proposition}
  Let \(\cat : \idxcat \to \RMCat\) be a pseudo-functor from a filtered
  small category \(\idxcat\) and \(\catI\) a representable map
  category. Then the canonical functor
  \[
    \inj^{*} : \RMCat(\pcolim_{\idx \in \idxcat}\cat_{\idx}, \catI) \to
    \plim_{\idx \in \idxcat}\RMCat(\cat_{\idx}, \catI)
  \]
  is an isomorphism of categories. \qed
\end{proposition}

\subsection{Type Theory Generated by a Theory}
\label{sec:theories-over-type}

Given a theory \(\theory\) over a type theory \(\tth\), we obtain
another type theory \({\tth}[\theory]\) by adjoining to \(\tth\) the
constants of \(\theory\) as inference rules with no premises. To make
it precise, we will show that, for an object \(\obj \in \tth\), the
slice category \(\tth/\obj\) is the type theory obtained from \(\tth\)
by freely adjoining a global section of \(\obj\)
(\cref{prop:indeterminate-on-lf-category}). Then the type theory
\({\tth}[\theory]\) is defined to be a suitable filtered
pseudo-colimit of slices \(\tth/\obj\).

\begin{lemma}
  \label{lem:cartesian-functor-el-filtered}
  Let \(\cat\) be a small cartesian category and
  \(\theory : \cat \to \Set\) a functor. Then \(\theory\) is cartesian if and
  only if its category of elements \(\int_{\cat}\theory\) is cofiltered.
\end{lemma}
\begin{proof}
  The proof can be found, for instance, in
  \parencite[Section VII.6]{maclane1992sheaves}.
\end{proof}

\begin{definition}
  Let \(\theory\) be a theory over a type theory \(\tth\). We define a
  type theory \({\tth}[\theory]\) to be the filtered pseudo-colimit of
  the composite pseudo-functor
  \[
    \begin{tikzcd}
      \left(\int_{\tth}\theory\right)^{\op} \arrow[r] &
      \tth^{\op} \arrow[r,"\tth/(\argu)"] &
      [2ex] \RMCat.
    \end{tikzcd}
  \]
\end{definition}

By definition, an object of \({\tth}[\theory]\) is a triple
\((\obj, \cst, \arr)\) consisting of an object \(\obj \in \tth\), an element
\(\cst \in \theory(\obj)\) and an object \(\arr \in \tth/\obj\). We think of an
object \(\obj \in \tth\) as an object of \({\tth}[\theory]\) via the
inclusion \(\obj \mapsto (1, {*}, \obj)\), where \({*}\) is the unique
element of \(\theory(1)\).

\begin{lemma}
  \label{lem:theory-and-global-section}
  Let \(\theory\) be a theory over a type theory \(\tth\). For an
  object \(\obj \in \tth\), we have a natural bijection \(\theory(\obj)
  \cong {\tth}[\theory](1, \obj)\).
\end{lemma}
\begin{proof}
  \begin{align*}
    \theory(\obj)
    &\cong \colim_{(\obj', \cst') \in \int_{\tth}\theory}\tth(\obj', \obj)
      \tag{Yoneda} \\
    &\cong \colim_{(\obj', \cst') \in \int_{\tth}\theory}\tth/\obj'(\obj',
      \obj' \times \obj) \\
    &\cong {\tth}[\theory](1, \obj)
  \end{align*}
\end{proof}

By \cref{lem:theory-and-global-section} we identify an element
\(\cst \in \theory(\obj)\) with the corresponding arrow \(\cst : 1 \to \obj\)
in \({\tth}[\theory]\).

\begin{proposition}
  \label{prop:type-theory-generated-by-theory}
  Let \(\theory\) be a theory over a type theory \(\tth\) and
  \(\cat\) a locally small representable map category. For a
  representable map functor \(\fun : \tth \to \cat\), we have an
  equivalence of categories
  \[
    (\tth \under \RMCat)({\tth}[\theory], \cat) \simeq
    \Theory_{\tth}(\theory, \cat(1, \fun \argu))
  \]
  that sends a representable map functor
  \(\funI : {\tth}[\theory] \to \cat\) equipped with a natural
  isomorphism \(\trans_{\obj} : \funI \obj \cong \fun \obj\) for \(\obj \in \tth\)
  to the natural transformation
  \(\theory(\obj) \ni \cst \mapsto \trans_{\obj} \circ \funI \cst \in \cat(1, \fun \obj)\).
\end{proposition}

To prove \cref{prop:type-theory-generated-by-theory}, we
show that the slice category \(\tth/\obj\) over an object
\(\obj \in \tth\) is the representable map category obtained from
\(\tth\) by \emph{freely adjoining an arrow \(1 \to \obj\)}. Let
\(\cat\) be a representable map category and \(\obj \in \cat\) an
object. We have a representable map functor
\(\obj^{*} : \cat \to \cat/\obj\) defined by \(\obj^{*}\objI = \obj \times \objI\)
and an arrow \(\diag_{\obj} : 1 \to \obj^{*}\obj\) in \(\cat/\obj\) represented
by the diagonal arrow \(\obj \to \obj \times \obj\).

\begin{lemma}
  \label{lem:slice-canonical-form}
  Let \(\cat\) be a cartesian category and \(\obj \in \cat\) an
  object. For every object \(\arr : \objI \to \obj\) of \(\cat/\obj\), we have
  the following pullback in \(\cat/\obj\).
  \[
    \begin{tikzcd}
      \arr \arrow[r]
      \arrow[d] &
      \obj^{*}\objI \arrow[d,"\obj^{*}\arr"] \\
      1 \arrow[r,"\diag_{\obj}"'] &
      \obj^{*}\obj
    \end{tikzcd}
  \]
\end{lemma}
\begin{proof}
  The square
  \[
    \begin{tikzcd}
      \objI \arrow[r,"{(\arr, \objI)}"]
      \arrow[d,"\arr"'] &
      \obj \times \objI \arrow[d,"\obj \times \arr"] \\
      \obj \arrow[r] &
      \obj \times \obj
    \end{tikzcd}
  \]
  is a pullback in \(\cat\).
\end{proof}

\begin{proposition}
  \label{prop:indeterminate-on-lf-category}
  Let \(\cat\) be a representable map category and
  \(\obj \in \cat\) an object. For any representable map category
  \(\catI\) and representable map functor
  \(\fun : \cat \to \catI\), the functor
  \[
    (\cat \under \RMCat)(\cat/\obj, \catI) \ni (\funI, \trans)
    \mapsto \trans_{\obj} \circ \funI \diag_{\obj} \in \catI(1, \fun \obj)
  \]
  is an equivalence of categories.
\end{proposition}
\begin{proof}
  Since \(\catI(1, \fun \obj)\) is a discrete category, it suffices to
  show that every fiber of the functor is contractible. Let
  \(\el : 1 \to \fun \obj\) be an arrow. By \cref{lem:slice-canonical-form},
  a representable map functor \(\funI : \cat/\obj \to \catI\) equipped
  with a natural isomorphism \(\trans : \funI \obj^{*} \cong \fun\) such that
  \(\trans_{\obj} \circ \funI \diag_{\obj} = \el\) must send an object \(\arr : \objI \to \obj\)
  of \(\cat/\obj\) to the pullback
  \[
    \begin{tikzcd}
      \funI \arr \arrow[r] \arrow[d] &
      \fun \objI \arrow[d,"\fun \arr"] \\
      1 \arrow[r,"\el"'] &
      \fun \obj.
    \end{tikzcd}
  \]
  Hence such a pair \((\funI, \trans)\) is unique up to unique
  isomorphism. Such a \((\funI, \trans)\) exists because the
  composite
  \begin{tikzcd}
    \cat/\obj \arrow[r,"\fun/\obj"] &
    \catI/\fun\obj \arrow[r,"\el^{*}"] &
    \catI
  \end{tikzcd}
  is a representable map functor such that
  \(\el^{*} (\fun/\obj) \obj^{*} \cong \fun\).
\end{proof}

\begin{proof}
  [Proof of \cref{prop:type-theory-generated-by-theory}]
  We have equivalences of categories
  \((\tth \under \RMCat)({\tth}[\theory], \cat) \simeq
  \plim_{(\obj, \cst) \in \int_{\tth}\theory}(\tth \under
  \RMCat)(\tth/\obj, \cat)\) and
  \(\Theory_{\tth}(\theory, \cat(1, \fun \argu)) \simeq \lim_{(\obj, \cst) \in
    \int_{\tth}\theory}\cat(1, \fun \obj)\). Then use
  \cref{prop:indeterminate-on-lf-category}.
\end{proof}

\subsection{The Bi-adjunction of Theories and Models}
\label{sec:intern-lang-type}

In this subsection we show the following theorem.

\begin{theorem}
  \label{thm:internal-language-adjunction}
  For a type theory \(\tth\), the \(2\)-functor
  \(\iL_{\tth} : \Mod_{\tth} \to \Theory_{\tth}\)
  has a left bi-adjoint.
\end{theorem}

\begin{definition}
  Let \(\theory\) be a theory over a type theory \(\tth\). We define a
  \(2\)-category \((\theory \downarrow \iL_{\tth})\) as follows:
  \begin{itemize}
  \item the objects are the pairs \((\model, \thmap)\) consisting of a
    model \(\model\) of \(\tth\) and a map \(\thmap : \theory \to
    \iL_{\tth}\model\) of \(\tth\)-theories;
  \item the morphisms \((\model, \thmap) \to (\modelI, \thmapI)\) are the
    morphisms \(\fun : \model \to \modelI\) of models of
    \(\tth\) such that \(\iL_{\tth}\fun \circ \thmap = \thmapI\);
  \item the \(2\)-morphisms are those of \(\Mod_{\tth}\).
  \end{itemize}
\end{definition}

\begin{lemma}
  \label{lem:bi-adjoint-char}
  For a type theory \(\tth\), the \(2\)-functor
  \(\iL_{\tth}\) has a left bi-adjoint if and only if the
  \(2\)-category \((\theory \downarrow \iL_{\tth})\) has a bi-initial
  object for every \(\tth\)-theory \(\theory\).
\end{lemma}
\begin{proof}
  Let \(\theory\) be a \(\tth\)-theory, \(\model\) a model of
  \(\tth\) and \(\thmap : \theory \to \iL_{\tth}\model\) a map of
  \(\tth\)-theories. For a model \(\modelI\) of
  \(\tth\), consider the functor
  \[
    \begin{tikzcd}
      \Mod_{\tth}(\model, \modelI)
      \arrow[r,"\iL_{\tth}"] &
      \Theory_{\tth}(\iL_{\tth}\model,
      \iL_{\tth}\modelI)
      \arrow[r,"\thmap^{*}"] &
      \Theory_{\tth}(\theory, \iL_{\tth}\modelI).
    \end{tikzcd}
  \]
  Since \(\Theory_{\tth}(\theory, \iL_{\tth}\modelI)\) is a
  discrete category, this functor is an equivalence if and only if
  every fiber is contractible. But the fiber over a map
  \(\thmapI : \theory \to \iL_{\tth}\modelI\) is just
  \((\theory \downarrow \iL_{\tth})((\model, \thmap), (\modelI,
  \thmapI))\). Thus, \((\model, \thmap)\) is a bi-universal map from \(\theory\) to
  \(\iL_{\tth}\) if and only if it is a bi-initial object of
  \((\theory \downarrow \iL_{\tth})\).
\end{proof}

We have a \(2\)-functor
\(\Mod_{{\tth}[\theory]} \ni \model \mapsto (\model|_{\tth},
\thmap_{\model}) \in (\theory \downarrow \iL_{\tth})\) where
\(\model|_{\tth}\) is the model of \(\tth\) obtained
from \(\model\) by restricting
\((\argu)^{\model} : {\tth}[\theory] \to \DFib_{\model}\) to
\(\tth\) and
\(\thmap_{\model} : \theory \to \iL_{\tth}(\model|_{\tth})\)
sends an element \(\cst \in \theory(\obj)\) to the map
\(\cst^{\model} : \model \to \obj^{\model}\) of discrete fibrations
over \(\model\).

\begin{lemma}
  \label{lem:models-of-type-theory-generated-by-theory}
  For a theory \(\theory\) over a type theory \(\tth\), the
  \(2\)-functor \(\Mod_{{\tth}[\theory]} \to (\theory \downarrow
  \iL_{\tth})\) is a bi-equivalence.
\end{lemma}
\begin{proof}
  It is clear that the \(2\)-functor is locally faithful. It remains
  to show that the \(2\)-functor is bi-essentially surjective on
  objects, locally essentially surjective on objects and locally
  full.

  To show that the \(2\)-functor is bi-essentially surjective on
  objects, let \(\model\) be a model of \(\tth\) and
  \(\thmap : \theory \to \iL_{\tth}\model\) a map of
  \(\tth\)-theories. By
  \cref{prop:type-theory-generated-by-theory} the
  representable map functor
  \((\argu)^{\model} : \tth \to \DFib_{\model}\) extends to a
  representable map functor
  \((\argu)^{\widetilde{\model}} : {\tth}[\theory] \to \DFib_{\model}\) that
  sends \(\cst : 1 \to \obj\) to
  \(\thmap(\cst) : \model \to \obj^{\model} \cong \obj^{\widetilde{\model}}\) for
  \(\cst \in \theory(\obj)\). Thus, \(\widetilde{\model}\) is a model of \({\tth}[\theory]\)
  such that
  \((\widetilde{\model}|_{\tth}, \thmap_{\widetilde{\model}}) \simeq (\model, \thmap)\).

  To show that the \(2\)-functor is locally essentially surjective on
  objects, let \(\model\) and \(\modelI\) be models of
  \({\tth}[\theory]\) and \(\fun : \model|_{\tth} \to
  \modelI|_{\tth}\) a morphism of models of \(\tth\)
  such that \(\iL_{\tth}\fun \circ \thmap_{\model} =
  \thmap_{\modelI}\). This equation means that, for any element \(\cst \in
  \theory(\obj)\), the diagram
  \begin{equation}
    \label{eq:2}
    \begin{tikzcd}
      \model
      \arrow[r,"\fun"]
      \arrow[d,"\cst^{\model}"'] &
      \modelI
      \arrow[d,"\cst^{\modelI}"] \\
      \obj^{\model}
      \arrow[r,"\fun_{\obj}"'] &
      \obj^{\modelI}
    \end{tikzcd}
  \end{equation}
  commutes. Recall from \cref{sec:anoth-defin-morph} that
  \(\fun_{(\argu)}\) can be regarded as a representable map functor
  \(\fun_{(\argu)} : \tth \to (\DFib^{\to})_{\fun}\) (\cref{prop:morphism-of-models}). Then
  \(\theory(\obj) \ni \cst \mapsto (\cst^{\model}, \cst^{\modelI}) \in
  (\DFib^{\to})(\fun, \fun_{\obj})\) defines a map
  \(\theory \to (\DFib^{\to})_{\fun}(\fun, \fun_{(\argu)})\) of
  \(\tth\)-theories. By \cref{prop:type-theory-generated-by-theory}
  the representable map functor
  \(\fun_{(\argu)} : \tth \to (\DFib^{\to})_{\fun}\) extends to a
  representable map functor
  \(\widetilde{\fun}_{(\argu)} : {\tth}[\theory] \to
  (\DFib^{\to})_{\fun}\), giving a morphism
  \(\widetilde{\fun} : \model \to \modelI\) of models of
  \({\tth}[\theory]\) whose restriction to \(\tth\) is \(\fun\).

  The local fullness is similarly proved using \cref{prop:2-morphisms-of-model} and
  \((\DFib^{\Theta})_{\trans}\) instead of \((\DFib^{\to})_{\fun}\).
\end{proof}

\begin{proof}[Proof of \cref{thm:internal-language-adjunction}]
  Let \(\theory\) be a \(\tth\)-theory. We have a bi-equivalence
  \(\Mod_{{\tth}[\theory]} \simeq (\theory \downarrow \iL_{\tth})\) by
  \cref{lem:models-of-type-theory-generated-by-theory}. Hence,
  \((\theory \downarrow \iL_{\tth})\) has a bi-initial object by
  \cref{thm:initial-model}. By \cref{lem:bi-adjoint-char}
  the \(2\)-functor \(\iL_{\tth}\) has a left bi-adjoint.
\end{proof}

We can extract an explicit construction of the left bi-adjoint of
\(\iL_{\tth}\) from the proof of
\cref{thm:internal-language-adjunction}. Let \(\theory\) be a theory
over a type theory \(\tth\). We will denote by
\((\sM_{\tth}\theory, \unit_{\theory})\) the bi-initial object of
\((\theory \downarrow \iL_{\tth})\). The model \(\sM_{\tth}\theory\)
of \(\tth\) is obtained from the bi-initial model
\(\iM({\tth}[\theory])\) of \({\tth}[\theory]\) by restricting
\((\argu)^{\iM({\tth}[\theory])} : {\tth}[\theory] \to
\DFib_{\iM({\tth}[\theory])}\) to \(\tth\). We call
\(\sM_{\tth}\theory\) the \emph{syntactic model of \(\tth\)
  generated by \(\theory\)}. The map
\(\unit_{\theory} : \theory \to \iL_{\tth}(\sM_{\tth}\theory)\) sends an
element \(\cst \in \theory(\obj)\) over \(\obj \in \tth\) to the arrow
\(\cst : 1 \to \obj\) in \({\tth}[\theory]\), which is an element of
\(\obj^{\sM_{\tth}\theory}(1)\).

\begin{example}
  Let \(\theory\) be a theory over the basic dependent type theory
  \(\thgat\), that is, a generalized algebraic
  theory. \(\thgat[\theory]\) satisfies the assumption of
  \cref{exm:bi-initial-model-ctx}, because every representable arrow
  in \(\thgat[\theory]\) is isomorphic to a representable arrow from
  some slice \(\thgat/\obj\). Hence, the objects of
  \(\sM_{\thgat}\theory\) are the finite sequences
  \((\obj_{1}, \dots, \obj_{n})\) of arrows
  \(\obj_{\idx} : |\obj_{\idx-1}| \to \Ty\), which corresponds to an
  arrow \(\obj : 1 \to \poly_{\typeof}^{n}1\) via the adjunction
  \(\typeof^{*} \adj \typeof_{*}\). By
  \cref{lem:theory-and-global-section}, the arrows
  \(1 \to \poly_{\typeof_{0}}^{n}1\) in \(\thgat[\theory]\) correspond
  to the elements of \(\theory(\poly_{\typeof_{0}}^{n}1)\). Since
  \(\theory(\poly_{\typeof_{0}}^{n}1) \cong \theory(\Ty_{n-1})\) by
  \cref{item:20} of \cref{thm:gat-exp} and elements of
  \(\theory(\Ty_{n-1})\) are contexts of length \(n\), the base
  category of \(\sM_{\thgat}\theory\) is the category of contexts in
  the generalized algebraic theory \(\theory\).
\end{example}

\subsection{The Bi-equivalence of Theories and Models}
\label{sec:bi-equiv-theor}

In this section we study the unit and counit of the bi-adjunction
\(\sM_{\tth} \adj \iL_{\tth}\) in more detail. For a model \(\model\)
of a type theory \(\tth\), we denote by
\(\counit_{\model} : \sM_{\tth}(\iL_{\tth}\model) \to \model\) the
counit of the bi-adjunction \(\sM_{\tth} \adj \iL_{\tth}\), that is,
one of those morphisms of models of \(\tth\) such that
\(\iL_{\tth}\counit_{\model} \circ \unit_{\iL_{\tth}\model} =
\id_{\iL_{\tth}\model}\). We first show that the unit
\(\unit : \id \To \iL_{\tth}\sM_{\tth}\) is an isomorphism
(\cref{prop:unit-invertible}). This implies that
\(\sM_{\tth} : \Theory_{\tth} \to \Mod_{\tth}\) is locally an
equivalence and thus induces a bi-equivalence between
\(\Theory_{\tth}\) and the bi-essential image of \(\sM_{\tth}\). We
then determine the bi-essential image of \(\sM_{\tth}\) by
characterizing those models \(\model\) such that the counit
\(\counit_{\model} : \sM_{\tth}\iL_{\tth}\model \to \model\) is an
equivalence. We show that the counit \(\counit_{\model}\) is an
equivalence precisely when \(\model\) is democratic
(\cref{cor:counit-invertible}). Hence, the bi-adjunction
\(\sM_{\tth} \adj \iL_{\tth}\) induces a bi-equivalence between
\(\tth\)-theories and democratic models of \(\tth\)
(\cref{thm:bi-equivalence-th-mod}).

\begin{proposition}
  \label{prop:unit-invertible}
  Let \(\tth\) be a type theory and \(\theory\) a
  \(\tth\)-theory. Then the map \(\unit_{\theory} : \theory \to
  \iL_{\tth}(\sM_{\tth}\theory)\) is an isomorphism.
\end{proposition}
\begin{proof}
  For an object \(\obj \in \tth\), the map \(\unit_{\theory}(\obj) : \theory(\obj) \to
  \iL_{\tth}(\sM_{\tth}\theory)(\obj)\) is the composite of
  isomorphisms
  \begin{align*}
    \theory(\obj)
    &\cong {\tth}[\theory](1, \obj)
      \tag{\cref{lem:theory-and-global-section}} \\
    &\cong \obj^{\sM_{\tth}\theory}(1) \\
    &\cong \iL_{\tth}(\sM_{\tth}\theory)(\obj).
  \end{align*}
\end{proof}

Thus, the unit of the bi-adjunction
\(\sM_{\tth} \adj \iL_{\tth}\) is always an
isomorphism. On the other hand, the counit is not an equivalence in
general, but we can say that it is always an embedding in the
following sense.

\begin{proposition}
  \label{prop:internal-language-counit}
  Let \(\model\) be a model of a type theory \(\tth\).
  \begin{enumerate}
  \item \label{item:10}
    The functor \(\counit_{\model} :
    \sM_{\tth}(\iL_{\tth}\model) \to \model\) is fully
    faithful.
  \item \label{item:11}
    For any object \(\obj \in \tth\), the square
    \[
      \begin{tikzcd}
        {\obj^{\sM_{\tth}(\iL_{\tth}\model)}}
        \arrow[r,"{(\counit_{\model})_{\obj}}"]
        \arrow[d] &
        \obj^{\model} \arrow[d] \\
        {\sM_{\tth}(\iL_{\tth}\model)}
        \arrow[r,"\counit_{\model}"'] &
        \model
      \end{tikzcd}
    \]
    is a pullback.
  \end{enumerate}
\end{proposition}

To prove \cref{prop:internal-language-counit} we need a
little more work. Let \(\theory\) be a theory over a type theory
\(\tth\) and \(\bobj \in \sM_{\tth}\theory\) and
\(\obj \in {\tth}[\theory]\) objects. Recall that \({\tth}[\theory]\) is the
filtered pseudo-colimit
\(\pcolim_{(\objI, \cst) \in \int_{\tth}\theory}\tth/\objI\). Then, by
\cref{lem:slice-canonical-form}, the objects
\(\bobj, \obj \in {\tth}[\theory]\) can be written as pullbacks
\[
  \begin{tikzcd}
    \bobj
    \arrow[r,"\proj_{\objII}"] \arrow[d] &
    \objII
    \arrow[d,"\arr"] \\
    1 \arrow[r,"\cst"'] &
    \objI
  \end{tikzcd}
  \quad
  \begin{tikzcd}
    \obj
    \arrow[r,"\proj_{\objIII}"] \arrow[d] &
    \objIII
    \arrow[d,"\arrI"] \\
    1 \arrow[r,"\cst"'] &
    \objI
  \end{tikzcd}
\]
for some objects \(\objI, \objII, \objIII \in \tth\), arrows
\(\arr : \objII \to \objI\) and \(\arrI : \objIII \to \objI\) and
element \(\cst \in \theory(\objI)\). By the definition of
representable arrows in \({\tth}[\theory]\), we may choose \(\arr\) to
be representable. Let \(\arrII : \objIV \to \objI\) be the local
exponent \(\objII \To_{\objI} \objIII\) in \(\tth/\objI\), that is,
\(\objIV = \arr_{*}\arr^{*}\objIII\). Let \(\model\) be a model of
\(\tth\) and \(\fun : \sM_{\tth}\theory \to \model\) a morphism of
models of \(\tth\). We denote by
\(\thmap : \theory \to \iL_{\tth}\model\) the corresponding map of
\(\tth\)-theories defined by \(\thmap(\cst') = \fun_{\obj'}(\cst')\)
for \(\cst' \in \theory(\obj')\).

\begin{lemma}
  \label{lem:technical-lemma-1}
  Under the assumptions above, the following properties hold.
  \begin{enumerate}
  \item Suppose \(\obj \in \sM_{\tth}\theory\). We may choose
    \(\arrI : \objIII \to \objI\) to be representable. Then we have
    isomorphisms
    \(\trans : \cst^{*}\theory(\objIV) \cong \sM_{\tth}\theory(\bobj,
    \obj)\) and
    \(\transI : \thmap(\cst)^{*}(\iL_{\tth}\model(\objIV)) \cong
    \model/\fun \obj\) such that the diagram
    \[
      \begin{tikzcd}
        \cst^{*}\theory(\objIV)
        \arrow[r,"\thmap"]
        \arrow[d,"\trans"',"\cong"] &
        \thmap(\cst)^{*}(\iL_{\tth}\model(\objIV))
        \arrow[d,"\transI","\cong"'] \\
        \sM_{\tth}\theory(\bobj, \obj)
        \arrow[r,"\fun"'] &
        \model(\fun \bobj, \fun \obj)
      \end{tikzcd}
    \]
    commutes.
  \item Suppose \(\obj \in \tth\). We may choose
    \(\objIII = \objI \times \obj\). Then we have isomorphisms
    \(\trans : \cst^{*}\theory(\objIV) \cong
    \obj^{\sM_{\tth}\theory}(\bobj)\) and
    \(\transI : \thmap(\cst)^{*}(\iL_{\tth}\model(\objIV)) \cong
    \obj^{\model}\) such that the diagram
    \[
      \begin{tikzcd}
        \cst^{*}\theory(\objIV)
        \arrow[r,"\thmap"]
        \arrow[d,"\trans"',"\cong"] &
        \thmap(\cst)^{*}(\iL_{\tth}\model(\objIV))
        \arrow[d,"\transI","\cong"'] \\
        \obj^{\sM_{\tth}\theory}(\bobj)
        \arrow[r,"\fun_{\obj}"'] &
        \obj^{\model}(\fun \bobj)
      \end{tikzcd}
    \]
    commutes.
  \end{enumerate}
\end{lemma}
\begin{proof}
  We have an isomorphism \(\trans : \cst^{*}\theory(\objIV) \cong
  {\tth}[\theory](\bobj, \obj)\) by
  \begin{align*}
    \cst^{*}\theory(\objIV)
    &\cong \{\cstI \in {\tth}[\theory](1, \objIV) \mid \arrII\cstI =
      \cst\}
      \tag{\cref{lem:theory-and-global-section}} \\
    &\cong \{\cstI \in {\tth}[\theory](\bobj, \objIII) \mid
      \text{\(
      \begin{tikzcd}[ampersand replacement=\&,sep=2ex]
        \bobj
        \arrow[r,"\cstI"]
        \arrow[d] \&
        \objIII
        \arrow[d,"\arr"] \\
        1
        \arrow[r,"\cst"'] \&
        \objI
      \end{tikzcd}
    \) commutes}
    \}
    \tag{\(\objIV = (\objII \To_{\objI} \objIII)\)} \\
    &\cong {\tth}[\theory](\bobj, \obj).
  \end{align*}
  Concretely \(\trans\) sends \(\cstI \in \cst^{*}\theory(\objIV)\) to
  the dotted arrow below,
  \[
    \begin{tikzcd}
      \bobj
      \arrow[rr,"{(\cstI, \proj_{\objII})}"]
      \arrow[ddr]
      \arrow[dr,dotted,"\trans(\cstI)"] & &
      \objIV \times_{\objI} \objII
      \arrow[d,"\ev"] \\
      & \obj
      \arrow[r,"\proj_{\objIII}"]
      \arrow[d]
      \arrow[dr,phantom,"\lrcorner"{very near start}] &
      \objIII
      \arrow[d,"\arrI"] \\
      & 1
      \arrow[r,"\cst"'] &
      \objI
    \end{tikzcd}
  \]
  where
  \(\ev : \objIV \times_{\objI} \objII \cong (\objII \To_{\objI}
  \objIII) \times_{\objI} \objII \to \objIII\) is the
  evaluation. Similarly, we have an isomorphism
  \(\transI : \thmap(\cst)^{*}(\iL_{\tth}\model(\objIV)) \cong
  \DFib_{\model}(\model/\fun\bobj, \thmap(\cst)^{*}\objIII^{\model})\)
  which sends
  \(\cstI' \in \thmap(\cst)^{*}(\iL_{\tth}\model(\objIV))\) to the
  dotted arrow below.
  \[
    \begin{tikzcd}
      \model/\fun\bobj
      \arrow[rr,"{(\cstI', \fun_{\objII}(\proj_{\objII}))}"]
      \arrow[ddr]
      \arrow[dr,dotted,"\transI(\cstI')"] & &
      \objIV^{\model} \times_{\objI^{\model}} \objII^{\model}
      \arrow[d,"\ev^{\model}"] \\
      & \thmap(\cst)^{*}\objIII^{\model}
      \arrow[r]
      \arrow[d]
      \arrow[dr,phantom,"\lrcorner"{very near start}] &
      \objIII^{\model}
      \arrow[d,"\arrI^{\model}"] \\
      & \model
      \arrow[r,"\thmap(\cst) = \fun_{\objI}(\cst)"'] &
      \objI^{\model}
    \end{tikzcd}
  \]

  Suppose that \(\obj \in \sM_{\tth}\theory\). By definition
  \({\tth}[\theory](\bobj, \obj) \cong \sM_{\tth}\theory(\bobj,
  \obj)\), and thus we regard \(\trans\) as an isomorphism
  \(\cst^{*}\theory(\objIV) \cong \sM_{\tth}\theory(\bobj,
  \obj)\). Choose \(\arrI : \objIII \to \objI\) to be
  representable. Then \(\obj\) is the context extension of
  \(\cst \in \objI^{\sM_{\tth}\theory}(1)\) along
  \(\arrI : \objIII \to \objI\), and thus
  \(\thmap(\cst)^{*}\objIII^{\model} \cong \model/\fun\obj\) because
  \(\fun\) preserves context extensions. Hence,
  \(\DFib_{\model}(\model/\fun\bobj, \thmap(\cst)^{*}\objIII^{\model})
  \cong \DFib_{\model}(\model/\fun\bobj, \model/\fun\obj) \cong
  \model(\fun\bobj, \fun\obj)\) by Yoneda, and we regard \(\transI\) as
  an isomorphism
  \(\thmap(\cst)^{*}(\iL_{\tth}\model(\objIV)) \cong \model(\fun\bobj,
  \fun\obj)\). For any element \(\cstI \in \cst^{*}\theory(\objIV)\),
  both \(\fun(\trans(\cstI))\) and \(\transI(\thmap(\cstI))\) make the
  diagram
  \[
    \begin{tikzcd}
      \model/\fun\bobj
      \arrow[rr,"{(\fun_{\objIV}(\cstI), \fun_{\objII}(\proj_{\objII}))}"]
      \arrow[ddr]
      \arrow[dr,dotted] & &
      \objIV^{\model} \times_{\objI^{\model}} \objII^{\model}
      \arrow[d,"\ev^{\model}"] \\
      & \model/\fun\obj
      \arrow[r,"\fun_{\objIII}(\proj_{\objIII})"]
      \arrow[d]
      \arrow[dr,phantom,"\lrcorner"{very near start}] &
      \objIII^{\model}
      \arrow[d,"\arrI^{\model}"] \\
      & \model
      \arrow[r,"\fun_{\objI}(\cst)"'] &
      \objI^{\model}
    \end{tikzcd}
  \]
  commute, and thus \(\fun\trans = \transI\thmap\).

  Suppose that \(\obj \in \tth\). By definition
  \({\tth}[\theory](\bobj, \obj) \cong
  \obj^{\sM_{\tth}\theory}(\bobj)\), and thus we regard \(\trans\) as
  an isomorphism
  \(\cst^{*}\theory(\objIV) \cong
  \obj^{\sM_{\tth}\theory}(\bobj)\). Choose \(\objIII\) to be
  \(\objI \times \obj\). Then
  \(\thmap(\cst)^{*}\objIII^{\model} \cong \obj^{\model}\), and we
  regard \(\transI\) as an isomorphism
  \(\thmap(\cst)^{*}(\iL_{\tth}\model(\objIV)) \cong
  \obj^{\model}(\fun\bobj)\) by Yoneda. This time \(\trans(\cstI)\)
  and \(\transI(\cstI')\) are just composites
  \begin{gather*}
    \begin{tikzcd}[ampersand replacement=\&]
      \bobj
      \arrow[r,"{(\cstI, \proj_{\objII})}"] \&
      \objIV \times_{\objI} \objII
      \arrow[r,"\ev"] \&
      \objIII \cong \objI \times \obj
      \arrow[r] \&
      \obj
    \end{tikzcd}
    \\
    \begin{tikzcd}[ampersand replacement=\&]
      \model/\bobj
      \arrow[r,"{(\cstI', \fun_{\objII}(\proj_{\objII}))}"] \&
      [6ex]
      \objIV^{\model} \times_{\objI^{\model}} \objII^{\model}
      \arrow[r,"\ev^{\model}"] \&
      \objIII^{\model} \cong \objI^{\model} \times \obj^{\model}
      \arrow[r] \&
      \obj^{\model},
    \end{tikzcd}
  \end{gather*}
  respectively, for \(\cstI \in \cst^{*}\theory(\objIV)\) and \(\cstI'
  \in \thmap(\cst)^{*}(\iL_{\tth}\model(\objIV))\). Therefore,
  \(\fun_{\obj}\trans = \transI\thmap\).
\end{proof}

\begin{proof}
  [Proof of \cref{prop:internal-language-counit}]
  Use \cref{lem:technical-lemma-1} for \(\fun =
  \counit_{\model}\). In this case, \(\thmap\) is the identity.
\end{proof}

By \cref{prop:internal-language-counit}, \(\counit_{\model}\) is an
equivalence in the \(2\)-category \(\Mod_{\tth}\) if and only if the
underlying functor
\(\counit_{\model} : \sM_{\tth}(\iL_{\tth}\model) \to \model\) is
essentially surjective on objects, because the base change of a
discrete fibration along an equivalence induces a fibred
equivalence. The next goal is to determine the essential image of the
functor \(\counit_{\model}\).

\begin{proposition}
  \label{prop:syntactic-model-democratic}
  For any theory \(\theory\) over a type theory \(\tth\), the model
  \(\sM_{\tth}\theory\) of \(\tth\) is democratic.
\end{proposition}
\begin{proof}
  We have already seen that every object \(\bobj \in \sM_{\tth}\theory\)
  is written as a pullback
  \[
    \begin{tikzcd}
      \bobj
      \arrow[r] \arrow[d] &
      \objI
      \arrow[d,"\arr"] \\
      1
      \arrow[r,"\cst"'] &
      \obj
    \end{tikzcd}
  \]
  for some representable arrow \(\arr : \objI \to \obj\) in \(\tth\) and
  element \(\cst \in \theory(\obj)\). This means that \(\bobj\) is the context
  extension of \(\cst \in \obj^{\sM_{\tth}\theory}(1)\) with respect to
  \(\arr\).
\end{proof}

\begin{proposition}
  Let \(\model\) be a model of a type theory \(\tth\). Then
  the essential image of the functor
  \(\counit_{\model} :
  \sM_{\tth}(\iL_{\tth}\model) \to \model\) is the
  class of contextual objects.
\end{proposition}
\begin{proof}
  By \cref{prop:syntactic-model-democratic}, the essential
  image of the functor
  \(\counit_{\model} :
  \sM_{\tth}(\iL_{\tth}\model) \to \model\) consists
  of contextual
  objects. \Cref{prop:internal-language-counit} implies
  that the essential image of \(\counit_{\model}\) is closed
  under context extensions. Hence, the essential image of
  \(\counit_{\model}\) is precisely the class of contextual
  objects.
\end{proof}

\begin{corollary}
  \label{cor:counit-invertible}
  Let \(\model\) be a model of a type theory \(\tth\).
  \begin{enumerate}
  \item \(\counit_{\model} : \sM_{\tth}(\iL_{\tth}\model) \to \model\)
    induces an equivalence \(\sM_{\tth}(\iL_{\tth}\model) \simeq
    \model^{\heart}\) in \(\Mod_{\tth}\).
  \item \(\counit_{\model}\) is an equivalence in \(\Mod_{\tth}\) if
    and only if \(\model\) is democratic.
  \end{enumerate}
  \qed
\end{corollary}

In summary, we get a bi-equivalence of theories and democratic models.

\begin{theorem}
  \label{thm:bi-equivalence-th-mod}
  For a type theory \(\tth\), the \(2\)-functor
  \(\iL_{\tth} : \Mod_{\tth} \to \Theory_{\tth}\)
  induces a bi-equivalence
  \[
    \Mod^{\dem}_{\tth} \simeq \Theory_{\tth}.
  \]
  \qed
\end{theorem}

\section{Conclusion and Future Directions}
\label{sec:conclusion}

We proposed an abstract notion of a type theory and established a
correspondence between theories and models for each type theory. This
is the first step in a new development of categorical type theory.

We should first mention the development of the theory of
\emph{\(\infty\)-type theories} \parencite{nguyen2022type-arxiv}. Since our
definition of a type theory is purely categorical, it is
straightforward to generalize it to a higher dimensional one, and we
obtain analogous results of
\cref{sec:bi-initial-models,sec:internal-languages}. An advantage of
this higher dimensional generalization is that one can handle (higher)
categorical models of type theories in the style of (non-split)
comprehension category \parencite{jacobs1993comprehension} within the
same framework. Our notion of a model of a type theory
(\cref{def:model-of-type-theory}) is never a categorical model in the
sense that the notion of identification of types is equality rather
than isomorphism. Categorical models of a type theory should instead
be understood as models of a suitable \(\infty\)-type theory.
\(\infty\)-type theories are thus a more convenient framework for the
semantics of type theory.

The main tools for studying type theories are the \(2\)-category
\(\RMCat\), the categories \(\Theory_{\tth}\), and the
\(2\)-categories \(\Mod_{\tth}\). The \(2\)-category \(\RMCat\) allows
us to compare different kinds of type theory directly in the sense
that a morphism \(\fun : \tth \to \tthI\) in \(\RMCat\) is thought of as
an \emph{interpretation} of the type theory \(\tth\) in \(\tthI\). In
particular, an equivalence in the \(2\)-category \(\RMCat\) is a
natural notion of an equivalence of type theories. The universal
properties of syntactic representable map categories
(\cref{thm:syntactic-LF-category}) help us to build various
interpretations of type theories.

Sometimes we need weaker notions of equivalences of type theories. For
example, one can ask if the interpretation of the Book HoTT
\parencite{hottbook} in cubical type theory
\parencite{cohen2016cubical} is an equivalence in any sense. This
interpretation will never be an equivalence in the \(2\)-category
\(\RMCat\), but one would expect it to be \emph{conservative} in a
weak sense: every type in cubical type theory is homotopy-equivalent
to some type from the Book HoTT; every term in cubical type theory is
path-connected to some term from the Book HoTT. One approach to
formulate this conjecture is to equip \(\Theory_{\tth}\) with a
(semi-)model structure, following
\textcite{kapulkin2016homotopy,isaev2017model}, and discuss if a
functor between such (semi-)model categories is a Quillen equivalence
\parencite{isaev2018morita}. Another possibility is to work in the
framework of \(\infty\)-type theories. Suppose that type theories
\(\tth_{1}\) and \(\tth_{2}\) are equipped with classes of arrows
called weak equivalences. An interpretation \(\tth_{1} \to \tth_{2}\) is
conservative with respect to these weak equivalences if it induces an
equivalences between the \(\infty\)-type theories obtained from
\(\tth_{1}\) and \(\tth_{2}\) by freely inverting weak equivalences.

Comparing \(\Theory_{\tth} \simeq \Mod_{\tth}^{\dem}\) and
\(\Mod_{\tth}\), the former is easier to understand and more
convenient to work with since it is a full subcategory of
\(\Set^{\tth}\). However, \(\Theory_{\tth}\) throws away all the
information about representable arrows in \(\tth\), so we can never
reconstruct the type theory \(\tth\) from the category
\(\Theory_{\tth}\). The \(2\)-category \(\Mod_{\tth}\), on the other
hand, seems to keep information about \(\tth\), and we expect that the
type theory \(\tth\) can be reconstructed from \(\Mod_{\tth}\). A
precise formulation is as follows. First, we regard \(\Mod_{\tth}\) as
a \(2\)-category over \(\Cat_{1}\) with the forgetful \(2\)-functor
\(\Mod_{\tth} \to \Cat_{1}\) that maps a model of \(\tth\) to its base
category. Then a representable map functor \(\fun : \tthI \to \tth\)
between type theories induces a \(2\)-functor
\(\fun^{*} : \Mod_{\tth} \to \Mod_{\tthI}\) over \(\Cat_{1}\) defined by
\(\fun^{*}(\model, (\argu)^{\model}) = (\model, (\fun
\argu)^{\model})\). Thus, \(\tth \mapsto \Mod_{\tth}\) is part of a
\(2\)-functor \(\Mod_{(\argu)} : \RMCat^{\op} \to \twoCAT/\Cat_{1}\) to
the (huge) \(2\)-category of (large) \(2\)-categories over
\(\Cat_{1}\).

\begin{question}
  Is \(\Mod_{(\argu)} : \RMCat^{\op} \to \twoCAT/\Cat_{1}\) monic (in a
  suitable higher categorical sense) and can we characterize its
  image?
\end{question}

In unpublished work of John Bourke and the author, it is shown that
the \(2\)-category \(\Mod_{\tth}\) is locally presentable in a
bicategorical sense (the same idea is used in the
\(\infty\)-categorical setting \parencite{nguyen2022type-arxiv} to show the
presentability of \(\Mod_{\tth}\), but there only invertible
\(2\)-morphisms of models are considered). Analogously to the
Gabriel-Ulmer duality \parencite{gabriel1971lokal}, the type theory
\(\tth\) is expected to be reconstructed using finitely bi-presentable
objects in \(\Mod_{\tth}\).

\section*{Acknowledgement}
\label{sec:acknowledgement}

The author is grateful to Benno van den Berg for helpful feedback and
corrections on drafts of this paper. The author would also like to thank Daniel
Gratzer, Thomas Streicher and John Bourke for useful questions,
comments and discussions. The author also thanks the reviewers for
comments and corrections. This work is part of the research programme
``The Computational Content of Homotopy Type Theory'' with project
number 613.001.602, which is financed by the Netherlands Organisation
for Scientific Research (NWO).

\printbibliography

\end{document}